\documentclass{article}
\usepackage{amsmath, amsthm, amssymb, amstext}
\usepackage[utf8]{inputenc}
\usepackage[T1]{fontenc}
\usepackage{array,amsfonts,mathrsfs,bm}
\usepackage{cancel}
\usepackage[shortlabels]{enumitem}
\theoremstyle{plain}
\newtheorem{theorem}{Theorem}[section]
\newtheorem{lemma}[theorem]{Lemma}
\newtheorem{prop}[theorem]{Proposition}

\title{Stability estimates in determination of non-orientable surface from its Dirichlet-to-Neumann map}
\author{D.V.Korikov\thanks {St.Petersburg Department of Steklov Mathematical Institute,
		\newline 
		St. Petersburg, Russia,
        \newline
        e-mail: thecakeisalie@list.ru,
        \newline
        ORCID: 0000-0002-3212-5874}.}

\date{\today}

\begin{document}
\maketitle
\begin{abstract}
Let $(M,g)$ and $(M',g')$ be non-orientable Riemannian surfaces with fixed boundary $\Gamma$ and fixed Euler characterictic $m$, and $\Lambda$ and $\Lambda'$ be their Dirichlet-to-Neumann maps, respectively. We prove that the closeness of $\Lambda'$ to $\Lambda$ in the operator norm implies the existence of of the near-conformal diffeomorphism $\beta$ between $(M,g)$ and $(M',g')$ which does not move the points of $\Gamma$. Hence we establish the continuity of the determination $\Lambda\mapsto [(M,g)]$, where $[(M,g)]$ is the conformal class of $(M,g)$ and the set of such conformal classes is endowed with the natural Teichm\"uller-type metric $d_T$. In both orientable and non-orientable case we provide quantitative estimates of $d_T([(M,g)],[(M',g')])$ via the operator norm of the difference $\Lambda'-\Lambda$. We also obtain generalizations of the results above to the case in which the Dirichlet-to-Neumann map is given only on a segment of the boundary.
\end{abstract}

\smallskip

\noindent{\bf Keywords:}\,\,\,electric impedance tomography of
surfaces, holomorphic immersions, Di\-rich\-let-to-Neumann map,
stability of determination, stability estimates, Teichm\"uller distance.

\smallskip


\section{Introduction}
\label{sec intro}
Let us further agree that a {\it surface} is a two-dimensional compact smooth Riemann manifold with a smooth boundary. Let $M$ be a surface with boundary $\Gamma$ and metric tensor $g$. Denote the harmonic function in $(M,g)$ with trace $f$ on $\Gamma$ by $u^f$. The operator $\Lambda$ that maps $f$ to the normal derivative $\partial_{\nu}u^f$ of $u^f$ on $\Gamma$ is called the {\it Dirichlet-to-Neumann map} (DN map) of the surface $(M,g)$. The two-dimensional {\it Electric Impedance Tomography} (EIT) consists in the determination of unknown surface $(M,g)$ via its DN map $\Lambda$.

Let $(M,g)$ and $(M',g')$ be surfaces with joint boundary $\Gamma$ and the metrics $g$ on $M$ and $g'$ on $M'$ induce the same length element $dl$ on $\Gamma$. We write $[(M',g')]=[(M,g)]$ if there is a conformal map between $(M,g)$ and $(M',g')$ which does not move points of $\Gamma$. The well-known result of M.Lassas and G.Uhlmann \cite{LU} states that the DN maps of $(M,g)$ and $(M',g')$ coincide if and only if $(M',g')\sim (M,g)$. So, the DN map $\Lambda$ determines not the surface $(M,g)$ itself, but only its conformal class $[(M,g)]$. This conformal class is understood as a solution to the EIT problem.

The question on {\it stability of solutions} to the EIT problem is the following. Let $(M',g')$ and $(M,g)$ are surfaces with the joint boundary $(\Gamma,dl)$. Suppose that the DN map $\Lambda'$ of $(M',g')$ is close to the DN map $\Lambda$ of $(M,g)$ (in the operator norm). Can one claim that the surfaces $(M',g')$ and $(M,g)$ (resp., their conformal classes) are also close? Here the crucial moment is how to understand the closeness between the surfaces $(M',g')$ and $(M,g)$, or in other words, how to compare $(M',g')$ and $(M,g)$. 

As clarified in \cite{ZNS}, the surface topology is unstable under small perturbations of the DN map i.e. the surfaces with the different topologies may have arbitrarily close DN maps. For this reason, we discuss a closeness of the surfaces $(M',g')$ and $(M,g)$ under the assumption that they are diffeomorphic i.e. have the same Euler characteristic $\chi(M')=\chi(M)=m$. Next, in the approach of \cite{BKstab}, it is not the surfaces themselves that are compared to, but images of their holomorphic immersions in $\mathbb{C}^n$. The measure of closeness of the images is the Hausdorff distance between them which is the infimum of positive $r$ such that $r-$neighbourhood of the first image contains the second one and $r-$neighbourhood of the second image contains the first one. The result of \cite{BKstab} states that, if the surface $(M,g)$ is orientable and $\mathcal{E}$ is its holomorphic immersion to $\mathbb{C}^n$, then, for any $(M',g')$ diffeomorphic to $(M,g)$ and such that its DN map $\Lambda'$ is close to the DN map $\Lambda$ of $(M,g)$, there is some ``induced'' holomorphic embedding $\mathcal{E}_{M'}: \ (M',g')\to\mathbb{C}^n$ such that the images $\mathcal{E}_{M'}(M')$ and $\mathcal{E}(M)$ are close in Hausdorff distance. 

A natural stabiliy result in the orientable case is obtained by the author and Mikhail Belishev in \cite{BKTeich} of . Here, a natural measure of closeness between conformal classes $[(M,g)]$ and $[(M',g')]$ is the {\it Teichm\"uller distance} between them \cite{Alf,Ga,Nag}. Let us write $o(M)=+$ (resp. $o(M)=-$) if $M$ is orientable (resp. non-orientable). Put $[(M,g)]\in\mathscr{T}^{o}_{m,\Gamma}$ and $\Lambda\in\mathscr{D}^{o}_{m,\Gamma}$ if $(M,g)$ is a surface with the boundary $(\Gamma,dl)$, $o(M)=o$, and its Euler characteristic $\chi(M)$ is equal to $m$ while $\Lambda$ is the DN map of $(M,g)$. Let $[(M,g)],[(M',g')]\in \mathscr{T}^{o}_{m,\Gamma}$ and $\beta: \ M\to M'$ be a diffeomorphism. The {\it dilatation} $K_\beta(x)$ of the map $\beta$ at the point $x$ of $M$ is the square of the ratio between major and minor axis (in the metric $g$) of ellipse obtained by pulling back the circle (in the metric $g'$) in the tangent space $T_{\beta(x)}M'$. The number $K_\beta:=\sup_{x\in M} K_\beta(x)$ is called the dilatation of the map $\beta$. Note that the diffeomorphism $\beta: \ (M,g)\to (M',g')$ is conformal if and only if $K_\beta=1$ and, in such a case, we have $[(M,g)]=[(M',g')]$. Otherwise, the value of ${\rm log}K_\beta$ shows how far the map $\beta$ from being conformal. The Teichm\"uller distance between $[(M,g)]$ and $[(M',g')]$ is given by 
\begin{equation}
\label{Tdist}
d_{T}([(M,g)],[(M',g')]):=\frac{1}{2}\inf_\beta {\rm log}K_\beta,
\end{equation}
where the infimum is taken over all diffeomorphisms $\beta: \ M\to M'$ which do not move points of $\Gamma$. Note that, in the orientable case, all diffeomorphisms in (\ref{Tdist}) are assumed to be orientation-preserving. The distance $d_{T}$ is well-defined (i.e. $d_{T}([(M,g)],[(M',g')])$ does not depend on the choise of the representatives $(M,g)$, $(M',g')$ of conformal classes $[(M,g)]$, $[(M',g')]$) and it is actually the metric on $\mathscr{T}^{o}_{m,\Gamma}$. Endowed with this metrics, $\mathscr{T}^o_{m,\Gamma}$ is the natural space of solutions to the EIT problem. Note that, if $o=+$ and $\Gamma=\varnothing$, then $\mathscr{T}^o_{m,\Gamma}$ is a {\it Teichm\"uller modili space} (see, e.g., \cite{Alf,Ga,Nag}). At the same time, the space $\mathscr{D}^{o}_{m,\Gamma}$ of DN maps is endowed with the metrics 
\begin{equation}
\label{D op}
d_{op}(\Lambda',\Lambda):=\parallel\Lambda'-\Lambda\parallel_{H^{1}(\Gamma)\to L^{2}(\Gamma)}.
\end{equation}
We introduce the map $\mathscr{R}^{o}_{m,\Gamma}: \ \mathscr{D}^{o}_{m,\Gamma}\to \mathscr{T}^{o}_{m,\Gamma}$ which solves the EIT problem i.e. $\mathscr{R}^{o}_{m,\Gamma}(\Lambda)=[(M,g)]$ if and only if $\Lambda$ is DN map of $(M,g$). Then the result of \cite{BKTeich} states that the solving map $\mathscr{R}^{o}_{m,\Gamma}: \ (\mathscr{D}^{o}_{m,\Gamma},d_{op})\to (\mathscr{T}^{o}_{m,\Gamma},d_T)$ is {\it continuous} in the case $o=+$. The proof of this result is based on the result of \cite{BKstab} and consists in the construction of near-isometric diffeomorphism between the image $\mathcal{E}(M)$ of the holomorphic embedding of $(M,g)$ and the image $\mathcal{E}_{M'}(M')$ of the induced holomorphic embedding of $(M',g')$ which is close to $\mathcal{E}(M)$ in the Hausdorff distance as the DN map $\Lambda'$ of $(M',g')$ is close to the DN map $\Lambda$ of $(M,g)$. Such a consequence of the Hausdorff stability result of \cite{BKstab} is, in spirit, somewhat reminiscent of Gromov's theorem on the connection between Gromov-Hausdorff and Lipschitz distances for manifolds (see 8.19, \cite{G}).

In this paper, we generalize these stability results to the non-orientable case $o=-$. Namely, we prove the following statement.
\begin{theorem}
\label{main theorem}
The map $\mathscr{R}^{-}_{m,\Gamma}: \ (\mathscr{D}^{-}_{m,\Gamma},d_{op})\to (\mathscr{T}^{-}_{m,\Gamma},d_T)$ is continuous.
\end{theorem}
The proof of Theorem \ref{main theorem} is constructive. Namely, for any surface $(M',g')\in \mathscr{R}^{-}_{m,\Gamma}$ such that its DN map $\Lambda'$ is sufficiently close to the DN map $\Lambda$ of $(M,g)$, we construct the map $\beta_{M,M'}: \ M\to M'$ whose dilatation is subject to the estimate
\begin{equation}
\label{stability estimate}
{\rm log}K_{\beta_{M,M'}}\le c_{[(M,g)]}t^{\frac{1}{3}},
\end{equation}
where $t=\parallel\Lambda'-\Lambda\parallel_{H^{1}(\Gamma)\to L^{2}(\Gamma)}$ and the constant $c_{[(M,g)]}$ depends only on $[(M,g)]$. The scheme used in the construction of $\beta_{M,M'}$ is similar to that used in \cite{BKstab,BKTeich} but its implementation is much more complicated. The first obstacle is that the scheme of \cite{BKstab,BKTeich} is based on the holomorphic immersions of surfaces $(M,g)$ and $(M',g')$ while in the nonorientable case there are no holomorphic functions on $(M,g)$ and $(M',g')$. We avoid this obstacle by implementing the scheme of \cite{BKstab,BKTeich} not to the surfaces $(M,g)$ and $(M',g')$ themselves but to their orientable double covers $({\rm M},{\rm g})$ and $({\rm M}',{\rm g}')$. The second obstacle is that the DN maps of the covers $({\rm M},{\rm g})$ and $({\rm M}',{\rm g}')$ cannot be expressed explicitly in terms of DN maps $\Lambda$, $\Lambda'$ of the original surfaces $(M,g)$ and $(M',g')$. Nevertheless, due to the result of \cite{BKor_SIAM}, the boundary traces of holomorphic functions on $({\rm M},{\rm g})$ and $({\rm M}',{\rm g}')$ can still be determinated as solutions to certain non-linear equations including $\Lambda$ and $\Lambda'$, respectively. So, the key point in the generalization of the scheme of \cite{BKstab,BKTeich} to the nonorientable case is the study of stability of solutions to such class of non-linear equations. This study is given in Section \ref{sec map iota} while the complete proof of Theorem \ref{main theorem} is provided by Sections \ref{sec map iota}--\ref{sec map alpha}. Along with Theorem \ref{main theorem}, we obtain some estimates of continuity of $\mathscr{R}^{o}_{m,\Gamma}$ (see Proposition \ref{continuity estmates}) which are optimal in the orientable case $o=+$.

\section{Preliminaries}
\label{sec preliminaries} 
In this section, we present the facts and notions along with the results of \cite{BKor_SIAM} that are of use in the proof of Theorem \ref{main theorem}. 
\paragraph{Orientable double cover.}
Let us recall the notion and properties of {\it orientable double cover} (o.d.c.) of surface (for more details, see, e.g., \cite[Chapter 15]{Lee}). The o.d.c. of the surface $(M,g)$ is the triple $(({\rm M},{\rm g}),\pi,\tau)$, where the {\it covering space} ${\rm M}$ is an orientable surface, the {\it projection} $\pi: \ {\rm M}\to M$ is an unramified two-sheeted smooth map, and the {\it involution} $\tau$ is an automorphism of ${\rm M}$ that interchanges the points of ${\rm M}$ with joint projection. The metrics ${\rm g}$ on ${\rm M}$ is chosen in such a way that the projection $\pi$ is a local isometry, ${\rm g}:=\pi_{*}g$. Then $\tau$ is an isometric automorphism of $({\rm M},{\rm g})$. The o.d.c. of $(M,g)$ always and it is unique in the following sense: if $((\tilde{{\rm M}},\tilde{{\rm g}}),\tilde{\pi},\tilde{\tau})$ is another o.d.c. of $(M,g)$ then there is an isometric diffeomorphism $\phi: \ ({\rm M},{\rm g})\to (\tilde{{\rm M}},\tilde{{\rm g}})$ such that $\pi=\tilde{\pi}\circ\phi$ and $\phi\circ\tau=\tilde{\tau}\circ\phi$. Note that the covering space ${\rm M}$ is connected if and only if $M$ is non-orientable; otherwise, it is diffeomorphic to the disjoint union of two copies of $M$. In particular, since the boundary $\Gamma$ of $M$ is orientable, the boundary $\partial {\rm M}$ of ${\rm M}$ consists of two connected components $\partial {\rm M}_\pm$ and the restriction of $\pi$ on each $\partial {\rm M}_\pm$ is an isometric diffeomorphism. The Euler characteristic $\chi({\rm M})$ of the covering space ${\rm M}$ is equal to twice the Euler characteristic $\chi(M)$ of $M$.

\smallskip

The orientation on $({\rm M},{\rm g})$ is fixed by choosing one of two smooth families $\Phi=\{\Phi_x\}_{x\in {\rm M}}$, where each $\Phi_x\in {\rm End}T_x{\rm M}$ is a rotation on the right angle in the tangent space i.e.
\begin{equation}
\label{rotation}
{\rm g}(\Phi_xa,\Phi_xb)={\rm g}(a,b), \quad {\rm g}(\Phi_xa,a)=0, \qquad \forall a,b\in T_x{\rm M}, \ x\in M
\end{equation}
obeying
\begin{equation}
\label{rotation symmetry}
\tau^{*}\Phi=-\Phi\tau^*.
\end{equation}
Note that (\ref{rotation}) implies 
\begin{equation*}
\Phi^2=-I,
\end{equation*}
where $I$ is the identity. Also, if $M$ is non-orientable, then ${\rm M}$ is connected and (\ref{rotation symmetry}) follows from (\ref{rotation}) and the continuity of the family $\Phi$. 

\smallskip

The boundary $\partial {\rm M}$ of ${\rm M}$ is oriented by the unit tangent vector field $\Phi\nu$, where $\nu$ is an outward normal to $\partial {\rm M}$. Note that $\tau^*\nu=\nu$ and, due to (\ref{rotation symmetry}), $\tau^*\gamma=-\gamma$. Let $\Upsilon:=\Gamma\times\{+,-\}$ be a union of two copies $\Gamma_\pm:=\Gamma\times\{\pm\}$ of $\Gamma$ endowed with the involution $\sigma$ acting by $\sigma(x,\pm):=(x,\mp)$ for $x\in\Gamma$. In what follows, we consider $\Gamma$ as a part $\Gamma_+$ of $\Upsilon$. We fix the orientation on $\Upsilon$ by the unit tangent vector field $\gamma$ obeying $\sigma^*\gamma=-\gamma$. Then the map $\varpi: \ \partial {\rm M}\to \Upsilon$ acting by the rule $\varpi(x):=(\pi(x),\pm)$ on $\partial{\rm M}_\pm$ is an isometric diffeomorphism and the rotation $\Phi$ on ${\rm M}$ can be chosen in such a way that $\varpi$ is orientation preserving i.e. $\varpi^*[\Phi\nu]=\gamma$. For any covering space considered in the paper, we agree this choice of rotation and identify $\partial {\rm M}$ and $\Upsilon$ by the rule $\varpi(x)\equiv x$. Under this agreements, we have $\Phi\nu\equiv\gamma$ and $\sigma\equiv\tau$.

\paragraph{Harmonic fields.} Let $L_2({\rm M},{\rm g})$ and $\vec{L}_2({\rm M},{\rm g})$ be the spaces of square integrable functions and vector fields on $({\rm M},{\rm g})$ endowed with the inner products
$$(u,v)_{L_2({\rm M},{\rm g})}:=\int_{\rm M}|u|^2 dS_{\rm g}, \qquad (a,b)_{\vec{L}_2({\rm M},{\rm g})}:=\int_{\rm M}{\rm g}(a,b)dS_{\rm g},$$
respectively. The field $a\in\vec{L}_2({\rm M},{\rm g})$ is called {\it potential} if $a=\nabla u$ and {\it harmonic} if ${\rm div}_{\rm g}a={\rm div}_{\rm g}\Phi a=0$ holds in ${\rm int}\,{\rm M}$. Denote the sets of all harmonic fields by $\mathcal{H}$ and the set of all harmonic potential fields by $\mathcal{Q}$. Note that $\mathcal{H}$ and $\mathcal{Q}$ are (closed) subspaces of $\vec{L}_2({\rm M},{\rm g})$. Denote the orthogonal complement of $\mathcal{Q}$ in $\mathcal{H}$ by $\mathcal{N}$. The dimension of $\mathcal{N}$ is finite and it is determined by the Euler characteristic $\chi({\rm M})$ of ${\rm M}$ by 
\begin{equation}
\label{defect spaces dim 0}
{\rm dim}\mathcal{N}=1-\chi({\rm M})
\end{equation}
(see, e.g., \cite{Sch}). Also, the elements of $\mathcal{N}$ are smooth on ${\rm M}$. Denote the length element on $\Upsilon\equiv\partial{\rm M}$ induced by the metrics ${\rm g}$  by $dl$. In view of the divergence theorem, the equality
$$\int_\Upsilon u\,{\rm g}(a,\nu)dl=(\nabla u,a)_{\vec{L}_2({\rm M},{\rm g})}$$
is valid for any $u\in C^{\infty}({\rm M})$ and $a\in \mathcal{H}$. In particular, the field $a\in \mathcal{H}$ belongs to $\mathcal{N}$ if and only if it is tangent on $\Upsilon$. Also, denote $\mathcal{D}:=\Phi\mathcal{N}$; then $a\in \mathcal{H}$ belongs to $\mathcal{D}$ if and only if it is normal on $\Upsilon$. For $a\in \mathcal{N}$ and $b\in \mathcal{D}$ denote ${\rm Tr}_\gamma a:={\rm g}(a,\gamma)|_{\Gamma_+}$ and ${\rm Tr}_\nu b:={\rm g}(b,\nu)|_{\Gamma_+}$, respectively. Note that the linear operators ${\rm Tr}_\gamma: \ \mathcal{N}\to{\rm Tr}_\gamma\mathcal{N}$ and ${\rm Tr}_\gamma: \ \mathcal{D}\to{\rm Tr}_\nu\mathcal{D}$ are isomorphisms. Indeed, if ${\rm Tr}_\gamma a=0$ (${\rm Tr}_\nu b=0$) on any segment $\tilde{\Upsilon}$ of non-zero length of $\Upsilon$, then $a=0$ ($b=0$) on $\tilde{\Upsilon}$ and $a=0$ ($b=0$) on ${\rm M}$ due to the uniqueness of continuation of solutions to first-order elliptic systems.

\smallskip

In view of (\ref{rotation}), the rotation $\Phi$ is unitary operator on $\vec{L}_2({\rm M})$ obeying $\Phi H=H$. Also, since the involution $\tau$ is an isometric automorphism of $({\rm M},{\rm g})$, its differential $\tau_{*}$ is a unitary involution on $\vec{L}_2({\rm M})$ preserving potential fields, $\tau_{*}\nabla u=\nabla (u\circ\tau)$. Hence, in view of (\ref{rotation symmetry}), one has $\tau_{*}\mathcal{H}=\mathcal{H}$, $\tau_{*}\mathcal{Q}=\mathcal{Q}$, and $\tau_{*}\mathcal{N}=\mathcal{N}$, $\tau_{*}\mathcal{D}=\mathcal{D}$. We say that $a\in\mathcal{H}$ is symmetric (anti-symmetric) with respect to the involution $\tau$ and write $a=a_+$ ($a=a_-$) if $\tau_{*}a=a$ ($\tau_{*}a=-a$). The symmetric and anti-symmetric fields are orthogonal in $\vec{L}_2({\rm M})$ since ${\rm g}(a_+,a_-)={\rm g}(\tau_{*}a_+,\tau_{*}a_-)={\rm g}(a_+,-a_-)$. For any $a\in\vec{L}_2({\rm M})$ there is the unique decomposition $a=a_++a_-$, where $a_+=(a+\tau_*a)/2$ is symmetric and $a_-=(a-\tau_*a)/2$ is anti-symmetric. In particular, one has $\mathcal{N}=\mathcal{N}_+\oplus\mathcal{N}_-$ and $\mathcal{D}=\mathcal{D}_+\oplus\mathcal{D}_-$, where the indices $+$ and $-$ denote the corresponding subspaces of symmetric and anti-symmetric elements, respectively. Also, formula (\ref{rotation symmetry}) implies $\Phi\mathcal{N}_\pm=\mathcal{D}_\mp$ and 
\begin{equation}
\label{defect spaces dim 1}
{\rm dim}\mathcal{N}_\pm={\rm dim}\mathcal{D}_\mp.
\end{equation}
Let $P_\mathcal{N}$ and $P_\mathcal{D}$ be projections in $\vec{L}_2({\rm M},{\rm g})$ onto $\mathcal{N}$ and $\mathcal{D}$, respectively. The operators $\mathcal{P}_\mathcal{N}:=P_\mathcal{N}|_\mathcal{D}: \ \mathcal{D}\to \mathcal{N}$ and $\mathcal{P}_\mathcal{D}:=P_\mathcal{D}|_\mathcal{N}: \ \mathcal{N}\to \mathcal{D}$ are mutually conjugate since 
$$(\mathcal{P}_\mathcal{D}a,b)_{\vec{L}_2({\rm M},{\rm g})}=(a,b)_{\vec{L}_2({\rm M},{\rm g})}=(a,\mathcal{P}_\mathcal{N}b)_{\vec{L}_2({\rm M},{\rm g})} \qquad (a\in\mathcal{N}, \ b\in\mathcal{D}).$$
Let $a\in\mathcal{D}$. Due to decomposition $\mathcal{H}=\mathcal{Q}\oplus\mathcal{N}$, we have $a=\nabla y_a+\mathcal{P}_\mathcal{N}a$, where $y_a$ is a smooth harmonic function on ${\rm M}$. Since $\tau_*a$ is again an element of $\mathcal{D}$ and $\tau_*a=\tau_*\nabla y_a+\tau_*\mathcal{P}_\mathcal{N}a=\nabla(y_a\circ\tau)+\tau_*\mathcal{P}_\mathcal{N}a$, we obtain $\tau_*\mathcal{P}_\mathcal{N}a=\mathcal{P}_\mathcal{N}\tau_* a$. Also, $\partial_{\nu}y_a={\rm g}(a,\nu)$ since the field $\mathcal{P}_\mathcal{N}$ is tangent on $\Upsilon$. Therefore, $\mathcal{P}_\mathcal{N}a=a-\nabla y_a$, where $y_a$ is a solution to the Neumann problem 
$$\Delta_{\rm g}y_a=0 \text{ in } {\rm M}, \qquad \partial_\nu y_a={\rm g}(a,\nu) \text{ on } \Upsilon$$ 
determined up to an additive constant. In particular, $\mathcal{P}_\mathcal{N}a=0$ if and only if the field $a=\nabla y_a$ is normal on $\Upsilon$ i.e. $y_a$ is constant on each component $\Gamma_\pm$ of $\Upsilon$. So, we obtain ${\rm Ker}\mathcal{P}_\mathcal{N}=\{{\rm const}\nabla \phi_0\}$, where $\phi_0$ is a harmonic function equal to $\pm 1$ on $\Gamma_{\pm}$. In addition, we have proved that $\tau_*\mathcal{P}_\mathcal{N}=\mathcal{P}_\mathcal{N}\tau_*$ and, hence, $\mathcal{P}_\mathcal{N}\mathcal{D}_\pm\subset\mathcal{N}_\pm$. Since $\phi_0\circ\tau=-\phi_0$, we have ${\rm Ker}\mathcal{P}_\mathcal{N}\subset\mathcal{D}_-$. In view of these facts, we obtain
\begin{equation}
\label{defect spaces dim 2}
{\rm dim}\mathcal{N}_+\ge {\rm dim}\mathcal{D}_+, \qquad {\rm dim}\mathcal{N}_-\ge {\rm dim}\mathcal{D}_- - 1.
\end{equation}
Now, let $b\in\mathcal{N}$. Due to decomposition $\mathcal{H}=\Phi\mathcal{H}=\Phi\mathcal{Q}\oplus\Phi\mathcal{N}=\Phi\mathcal{Q}\oplus\mathcal{D}$, we have $b=\Phi\nabla q_b+\mathcal{P}_\mathcal{D}b$, where $y_b$ is a smooth harmonic function on ${\rm M}$. Note that $\tau_*b\in\mathcal{N}$ and $\tau_*b=\tau_*\Phi\nabla q_b+\tau_*\mathcal{P}_\mathcal{D}b=\Phi\nabla(-y_a\circ q_b)+\tau_*\mathcal{P}_\mathcal{D}b$ due to (\ref{rotation symmetry}). Thus, $\tau_*\mathcal{P}_\mathcal{D}b=\mathcal{P}_\mathcal{D}\tau_* b$. Also, $\partial_{\nu}q_b={\rm g}(\nabla q_b,\nu)={\rm g}(\Phi\nabla q_b,\Phi\nu)={\rm g}(\Phi\nabla q_b,\gamma)={\rm g}(b,\gamma)$ since the field $\mathcal{P}_\mathcal{D}b$ is normal on $\Upsilon$. Therefore, $\mathcal{P}_\mathcal{D}b=b-\Phi\nabla q_b$, where $q_b$ is a solution to the Neumann problem 
$$\Delta_{\rm g}q_b=0 \text{ in } {\rm M}, \qquad \partial_\nu q_b={\rm g}(b,\gamma) \text{ on } \Upsilon$$ 
determined up to an additive constant. In particular, $\mathcal{P}_\mathcal{D}b=0$ if and only if the field $b=\Phi\nabla q_b$ is tangent on $\Upsilon$ i.e. $q_b$ is constant on each $\Gamma_\pm$. So, we obtain ${\rm Ker}\mathcal{P}_\mathcal{D}=\{{\rm const}\Phi\nabla \phi_0\}$. In addition, we have proved that $\tau_*\mathcal{P}_\mathcal{D}=\mathcal{P}_\mathcal{D}\tau_*$, whence $\mathcal{P}_\mathcal{D}\mathcal{N}_\pm\subset\mathcal{D}_\pm$. Note that $\tau_*\Phi\nabla \phi_0=-\Phi\nabla(\phi_0\circ\tau)=\Phi\nabla\phi_0$ in view of (\ref{rotation symmetry}), and, thus, ${\rm Ker}\mathcal{P}_\mathcal{D}\subset\mathcal{N}_+$. Hence,
\begin{equation}
\label{defect spaces dim 2a}
{\rm dim}\mathcal{D}_+\ge {\rm dim}\mathcal{N}_+-1, \qquad {\rm dim}\mathcal{D}_-\ge {\rm dim}\mathcal{N}_-.
\end{equation}
Now, combining (\ref{defect spaces dim 0}),(\ref{defect spaces dim 1}),(\ref{defect spaces dim 2}),(\ref{defect spaces dim 2a}) with the equality $\chi({\rm M})=2\chi(M)$ yields
\begin{equation}
\label{defect spaces dim 3}
{\rm dim}\mathcal{N}_-={\rm dim}\mathcal{D}_+=-\chi(M), \qquad {\rm dim}\mathcal{N}_+={\rm dim}\mathcal{D}_-=1-\chi(M).
\end{equation}

\paragraph{Holomorphic functions.}
A smooth function $w$ on $({\rm M},{\rm g})$ is called {\it holomorphic} if its real and imaginary parts ${\rm u}:=\Re w$, ${\rm v}:=\Im w$ satisfy the Cauchy-Riemann equation $\nabla v=\Phi\nabla u$ on ${\rm M}$. Then $u,v$ are harmonic: $\Delta_{\rm g}{\rm u}=\Delta_{\rm g}{\rm v}=0$ in ${\rm int}{\rm M}$. Note that the rotation $\Phi$ induce the complex structure (a biholomorphic atlas) on ${\rm M}$ (see, e.g., \cite{Chirka}). With respect to this structure, $w$ is holomorphic in the usual sense. In the subsequent, $\mathcal{H}({\rm M})$ denotes the space of holomorphic smooth functions on $({\rm M},{\rm g})$. Of course, the definitions of the rotation (\ref{rotation}) and holomorphic functions remain valid for $M$ if $M$ is orientable. 

\smallskip

Due to the symmetry (\ref{rotation symmetry}) of the covering space $({\rm M},{\rm g})$, the map $ w\mapsto  w^*:=\overline{ w\circ\tau}$ is an involution on the space of holomorphic functions. The holomorphic function $ w$ is called {\it symmetric} if $ w^*= w$. Each holomorphic $w$ can be represented as
\begin{equation}
\label{graduation}
w=w_++i w_-, \text{ where } w_+:=\frac{ w+ w^*}{2}=w_+^*, \qquad w_-:=\frac{ w- w^*}{2i}=w_-^*.
\end{equation}
In the subsequent, the sub-space of symmetric elements of $\mathcal{H}({\rm M})$ is denoted by by $\mathcal{H}_+({\rm M})$. Let $w$ be an element of $\mathcal{H}_+({\rm M})$; then its real part obeys $\Re\,w\circ\tau=\Re\,w$, and, therefore, it can be rewritten as $\Re\,w=u\circ\pi$, where $u$ is a harmonic smooth function on $M$. Conversely, let $u^f$ is a smooth harmonic function on $M$ with the trace $f$ on $\Gamma$. Due to decomposition $\mathcal{H}=\mathcal{Q}\oplus\mathcal{N}$, we have 
$$\Phi\nabla(u^f\circ\pi)=\nabla v+a,$$ 
where $v$ is a harmonic smooth function on ${\rm M}$ and $a\in\mathcal{N}$. Then, $u^f\circ\pi$ is a real part of some function $w\in \mathcal{H}_+({\rm M})$ if and only if $a=0$. Also,
\begin{align*}
\Phi\nabla(u^f\circ\pi)&=\Phi\nabla(u^f\circ\pi\circ\tau)=\Phi\tau_*\nabla(u^f\circ\pi)=\\
=&-\tau_*\Phi\nabla(u^f\circ\pi)=-\tau_*\nabla v-\tau_*a=-\nabla(v\circ\tau)-\tau_*a
\end{align*}
due to (\ref{rotation symmetry}). Comparing the last two formulas yields $v\circ\tau=-v$ and $\tau_*a=-a$ i.e. $a\in\mathcal{N}_-$. Let $b$ be any element of $\mathcal{N}_-$ (i.e. $d=-\Phi b$ be any element of $\mathcal{D}_+$). Since $\nabla v\perp b$, we have 
\begin{align*}
(a,b)_{\Vec{L}_2({\rm M},{\rm g})}=(\Phi\nabla(u^f\circ\pi),b)_{\Vec{L}_2({\rm M},{\rm g})}=(\nabla(u^f\circ\pi),d)_{\Vec{L}_2({\rm M},{\rm g})}=\\=\int_\Upsilon f\circ\pi\cdot{\rm g}(d,\nu)dl=2\int_\Gamma f\cdot {\rm Tr}_\nu d\,dl
\end{align*}
due to (\ref{rotation}) and the divergence theorem. As a result, we obtain
\begin{equation}
\label{real part of hol sym func}
u^f\circ\pi\in \Re\,\mathcal{H}_+({\rm M}) \ \Longleftrightarrow \ \int_\Gamma f\tilde{d}dl=0 \quad \forall\tilde{d}\in{\rm Tr}_\nu \mathcal{D}_+.
\end{equation}

\smallskip

Denote the restriction of the trace map $w\mapsto  w|_{\partial{\rm M}}$ on the space of holomorphic functions by ${\rm Tr}$; then ${\rm Tr}$ is a bijection between holomorphic functions and their boundary traces. Also, the {\it generalized argument principle} holds: if $ w$ and $\tilde{w}$ are holomorphic and $z\in\mathbb{C}\backslash  w(\partial{\rm M})$, then 
\begin{equation}
\label{generalized argument principle}
\frac{1}{2\pi i}\int\limits_{\partial{\rm M}}\tilde{\eta}\frac{d\eta}{\eta-z}=\sum_{\bf x} \tilde w(x){\rm ord}_x( w),
\end{equation}
where $\eta:={\rm Tr} w$, $\tilde{\eta}:={\rm Tr} \tilde{w}$, ${\rm ord}_x( w)$ is the order of the zero $x$ of $ w$, and the summation is taken over all zeroes of $ w$ in ${\rm int}{\rm M}$. Equality (\ref{generalized argument principle}) is a result of the application of the Stokes theorem (see Theorem 3.16, \cite{M}) and the residue theorem (see Lemma 3.12, \cite{M}) to the meromorphic 1-form $\tilde{w}\ \! dw/(w-z)$. If $\tilde w=1$, then (\ref{generalized argument principle}) becomes the usual argument principle 
\begin{equation}
\label{argument principle}
{\rm wind}(\eta,z)={\rm ord}(w-z),
\end{equation}
where ${\rm wind}(\eta,z)$ is the winding number of the curve $\eta: \ \Gamma\to \mathbb{C}$ around the point $z$ and ${\rm ord}(w)$ is the total multiplicity of zeroes of $w$ in ${\rm int}{\rm M}$.

\paragraph{Determining the traces of holomorphic functions on ${\rm M}$ via DN map $\Lambda$.} In \cite{BKor_SIAM} it has been proved that the space of boundary traces of holomorphic functions on the covering space ${\rm M}$ is determined by the DN map $\Lambda$ of the (non-orientable) surface $M$. In this paragraph, we present a detailed formulation of this result.

\smallskip

Let $\partial_\gamma$ be the derivative with respect to the length on $\Upsilon$ in the direction $\gamma$. In this paragraph, we consider $\partial_\gamma$ as a linear operator acting on $C^{\infty}(\Gamma;\mathbb{R})$. Note that $\partial_{\gamma}C^{\infty}(\Gamma;\mathbb{R})$ is the space of smooth functions with zero mean values on $(\Gamma,dl)$ and $\Lambda C^{\infty}(\Gamma;\mathbb{R})=\partial_{\gamma}C^{\infty}(\Gamma;\mathbb{R})$. The integration $J$ is the operator inverse to $\partial_{\gamma}$ on $\partial_{\gamma}C^{\infty}(\Gamma;\mathbb{R})$. Introduce the linear operator $\mathfrak{D}$ and the non-linear map $\mathfrak{N}$ acting on $C^{\infty}(\Gamma;\mathbb{R})$ by the following rules
\begin{align}
\label{numerator denominator}
\begin{split}
\mathfrak{N}(f):&=\frac{1}{2}\Lambda[f^2-(J\Lambda f)^2]-f\Lambda f-(J\Lambda f)\partial_\gamma f,\\
\mathfrak{D}f:&=(\partial_{\gamma}+\Lambda J\Lambda)f.
\end{split}
\end{align}
Also, denote the function equal to $\pm 1$ on $\Gamma_\pm$ by $\sigma$.
\begin{prop}[See Lemma 1, \cite{BKor_SIAM} and Theorem 2.3, \cite{BKor_JIIPP}]
\label{trace-DN-map connection}
The following statements are valid:
\begin{enumerate}[a{\rm)}]
\item Let $f\in C^{\infty}(\Gamma;\mathbb{R})$ be non-constant and $\tilde{\Gamma}$ be arbitrary segment of $\Gamma$ of non-zero length. Then $\mathfrak{D}f=0$ on $\tilde{\Gamma}$ if and only if $M$ is orientable, and $f+iJ\Lambda f$ is a trace on $\Gamma$ of holomorphic function on $M$.

\item Suppose that $M$ is non-orientable. The function $\eta:=f+ih$ with $f,h\in C^{\infty}(\Gamma;\mathbb{R})$ is a trace on $\Gamma\equiv\Gamma_+$ of symmetric holomorphic function on ${\rm M}$ if and only if there is a constant $c\in\mathbb{R}$ such that $\mathfrak{N}(f)=c\mathfrak{D}f$ and $h=\sigma (J\Lambda f)\circ\pi+c$ hold on $\Upsilon$.
\end{enumerate}
\end{prop}
Let us comment Proposition \ref{trace-DN-map connection}. The statement {\it a}) provides the criterion of orientability of the surface $M$ which is obtained in \cite{BKor_JIIPP}. In the case of orientable $M$, it also provides the criterion for $\eta$ to be a trace on $\Gamma$ of a holomorphic function on $M$, which has been first obtained in \cite{B}. Due to this criterion, the boundary traces of holomorphic functions on $M$ can be founded as solutions to the linear system  $\partial_\gamma\Im\eta=\Lambda\Re\eta$, $\partial_\gamma\Re\eta=-\Lambda\Im\eta$ obtained by restriction of the Cauchy-Riemann equation on the boundary $\Gamma$. In contrast to this, in the non-orientable case, the traces of symmetric holomorphic functions on ${\rm M}$ are founded by solving the {\it non-linear} equation $\mathfrak{N}(f)={\rm const}\mathfrak{D}f$ provided by {\it b}). In this case, the statement {\it a}) implies that $\mathfrak{D}f$ does not vanish on any segment of $\Gamma$ of non-zero length for any nonconstant $f$ and the ratio 
\begin{equation}
\label{ratio}
c_f:=\mathfrak{N}(f)/\mathfrak{D}f.
\end{equation}
is well-defined (for constant fuctions $f$, we set $c_f=0$). The statement {\it b}) is proved in \cite{BKor_SIAM}. Together with (\ref{graduation}), it shows that the set of traces of holomorphic functions on $\partial {\rm M}\equiv\Upsilon$ is determined by $\Lambda$. 

\smallskip

For more convenience, we rewrite {\it b}) in the following form. Introduce the non-linear map $\mathcal{G}$ acting on functions from $C^{\infty}(\Gamma;\mathbb{R})$ by 
\begin{equation}
\label{non-linear map}
\mathcal{G}(f):=\mathfrak{D}f\cdot\partial_{\gamma}\mathfrak{N}(f)-\mathfrak{N}(f)\partial_{\gamma}\mathfrak{D}f=(\mathfrak{D}f)^2\partial_{\gamma}\Big(\frac{\mathfrak{N}(f)}{\mathfrak{D}f}\Big).
\end{equation}
\begin{prop} 
\label{trace-DN-map connection 1}
Let $M$ be non-orientable. Then the following statements are valid:
\begin{enumerate}[a{\rm)}]
\item the map $\mathcal{G}$ is positive homogeneous of degree 3 i.e. $\mathcal{G}(cf)=c^3\mathcal{G}(f)$ for any $f\in C^{\infty}(\Gamma;\mathbb{R})$ and $c>0$.
\item The nonconstant function $f\in C^{\infty}(\Gamma;\mathbb{R})$ satisfies $\mathcal{G}(f)=0$ if and only if $f\circ\pi+i\sigma((J\Lambda f)\circ\pi+c_f)$ is a trace on $\partial{\rm M}\equiv\Upsilon$ of symmetric holomorphic function on ${\rm M}$. In particular, the null set $\mathcal{G}^{-1}(\{0\})$ of $\mathcal{G}$ is a linear space and $f\mapsto c_f$ is a linear functional on $\mathcal{G}^{-1}(\{0\})$.
\item The codimension of $\mathcal{G}^{-1}(\{0\})$ in $C^{\infty}(\Gamma;\mathbb{R})$ is finite and it is equal to $-\chi(M)$.
\end{enumerate}
\end{prop}
\begin{proof}
The statement {\it a}) follows from definitions (\ref{numerator denominator}) and (\ref{non-linear map}) while {\it b}) is equivalent to Proposition \ref{trace-DN-map connection}, {\it b}). The linearity of $f\mapsto c_f$ on $\mathcal{G}^{-1}(\{0\})$ follows from the fact that ${\rm Tr}\mathcal{H}_+({\rm M})$ is a real linear space and $c_f$ is a unique constant which makes $f\circ\pi+i\sigma((J\Lambda f)\circ\pi+c_f)$ an element of ${\rm Tr}\mathcal{H}_+({\rm M})$. Finally, {\it c}) is a corollary of {\it b}) and formulas (\ref{defect spaces dim 3}), (\ref{real part of hol sym func}).
\end{proof}

\paragraph{Holomorphic embeddings.} Denote the $k$-th component of the point $\xi\in\mathbb{C}^n$ by $\xi_k$ and introduce the component-wise complex conjugation in $\mathbb{C}^n$ by $\overline{\xi}:=(\overline{\xi}_1,\dots,\overline{\xi}_n)$. Consider the differentiable map
\begin{equation*}
\label{embedding}
\mathcal{E}: \ {\rm M}\to\mathbb{C}^n, \quad x\mapsto \{w_1(x),\dots,w_n(x)\}
\end{equation*}
and denote $\eta_k:=w_k|_\Upsilon$. Recall that $\mathcal{E}$ is an embedding if its differential $d\mathcal{E}$ is injective everywhere on ${\rm M}$ and the map $\mathcal{E}: \ {\rm M}\mapsto \mathcal{E}({\rm M})$ is a homeomorphism. In this case, the surface $\mathcal{E}({\rm M})\subset\mathbb{C}^n$ is endowed with the metrics induced by the standard metrics on $\mathbb{C}^n$.

We say that the embedding $\mathcal{E}$ is {\it holomorphic} if its components $w_k$ are holomorphic on ${\rm M}$. The crucial fact is that, in this case, the map $\mathcal{E}: \ {\rm M}\mapsto \mathcal{E}({\rm M})$ is conformal (see, e.g., Lemma 3, \cite{BKTeich}). Also, the image $\mathcal{E}({\rm M})$ does not depend on the choice of the surface representing the conformal class $[(M,g)]$ or its covering space, but only on $[(M,g)]$ and the choice of the traces $\eta_k$. Indeed, let $(M,g)$ and $(\tilde{M},\tilde{g})$ be surfaces, $(({\rm M},{\rm g)}),\pi,\tau)$ and $((\tilde{\rm M},\tilde{\rm g}),\tilde{\pi},\tilde{\tau})$ be their o.d.c.-s, and $\mathcal{E}=\{w_1,\dots,w_n\}$ and $\tilde{\mathcal{E}}=\{\tilde{w}_1,\dots,\tilde{w}_n\}$ be their holomorphic embeddings, respectively. By the agreement above, $\partial {\rm M}\equiv \partial \tilde{\rm M}\equiv\Upsilon$. Suppose that $[(\tilde{M},\tilde{g})]=[(M,g)]$ and $\tilde{w}_k|_\Upsilon=w_k|_\Upsilon$. Due to the first equality, there is a conformal diffeomorphism $\beta_0: (M,g)\to (\tilde{M},\tilde{g})$ which does not move points of $\Gamma=\partial M=\partial \tilde{M}$. The lifting of the map $\beta_0$ on the covering spaces provides the conformal diffeomorphism $\beta: \ ({\rm M},{\rm g}) \to (\tilde{\rm M},\tilde{\rm g})$ obeying $\tilde{\pi}\circ\beta=\beta_0\circ\pi$ and $\beta(x)=x$ for $x\in\Upsilon$. By composing $\beta$ with the involution $\tau$ if needed, one makes $\beta$ orientation-preserving and, hence, a biholomorphism. Then the functions $w_k$ and $\tilde{w}_k\circ\beta$ are holomorphic on ${\rm M}$ and do coincide on $\Upsilon$. In view of the uniqueness of holomorphic continuation, we have $w_k=\tilde{w}_k\circ\beta$ and, hence, $\tilde{\mathcal{E}}(\tilde{\rm M})=\mathcal{E}({\rm M})$. 

\smallskip

For $\hat{\xi}\in\mathbb{C}^n$, introduce the projection $\mathfrak{p}_{\hat{\xi}}: \ \mathbb{C}^n\to\mathbb{C}$ and the functions $w_{\hat{\xi}}: \ {\rm M}\to\mathbb{C}$ and $\eta_{\hat{\xi}}: \ \Upsilon\to\mathbb{C}$ by 
\begin{equation}
\label{projection}
\mathfrak{p}_{\hat{\xi}}\xi:=\sum_{j=1}^n \hat{\xi}_j \xi_j, \qquad w_{\hat{\xi}}:=\mathfrak{p}_{\hat{\xi}}\circ\mathcal{E}=\sum_{j=1}^n \hat{\xi}_j w_j, \qquad \eta_{\hat{\xi}}:=w_{\hat{\xi}}|_\Upsilon=\sum_{j=1}^n \hat{\xi}_j \eta_j.
\end{equation}
The pre-image $\Pi_{\hat{\xi},\mathcal{D}}:=\mathfrak{p}_{\hat{\xi}}^{-1}(\mathcal{D})$ of the domain $\mathcal{D}\subset\mathbb{C}$ is a cylinder in $\mathbb{C}^n$. We say that the cylinder $\Pi_{\hat{\xi},\mathcal{D}}$ is projective for $\mathcal{E}({\rm M})$ if it has non-empty intersection with $\mathcal{E}({\rm M})$ and the restriction of $\mathfrak{p}_{\hat{\xi}}$ on $\Pi_{\hat{\xi},\mathcal{D}}\cap\mathcal{E}({\rm M})$ is an embedding. The embedding $\mathcal{E}$ is called {\it projective} if any $\xi\in \mathcal{E}({\rm M})$ belongs to some projective cylinder. Due to the divisor theorem (Theorem 26.5, \cite{Forster}, see also Proposition 2, \cite{BKstab}), each holomorphic projective embeddings (of sufficiently large dimension) do exist for o.d.c. of any surface $(M,g)$.

\smallskip

The crucial property of the holomorphic projective embeddings is that their images are determined by boundary traces of their components via the generalized argument principle (\ref{generalized argument principle}). Indeed, let $\mathcal{E}=\{w_1,\dots,w_n\}$ be a holomorphic projective embedding. Let us say that the pair $(\hat{\xi},z)$ with $\xi\in\mathbb{C}^{n}$ and $z\in\mathbb{C}$ is admissible if ${\rm wind}(\eta_{\hat{\xi}},z)=1$ or, equivalently, if $w_{\hat{\xi}}-z$ has a unique and simple zero $x\in {\rm M}$. Then formula (\ref{generalized argument principle}) with $w=w_{\hat{\xi}}$, $\tilde{w}=w_k$ provides the point $\xi:=\mathcal{E}(x)$ of $\mathcal{E}({\rm int}{\rm M})$ with coordinates
\begin{equation}
\label{GAP application}
\frac{1}{2\pi i}\int\limits_{\Upsilon}\eta_{k}\frac{d\eta_{\hat{\xi}}}{\eta_{\hat{\xi}}-z}=w_k(x)=\xi_k, \qquad k=1,\dots,n.
\end{equation}
By applying (\ref{GAP application}) for all admissible $(\hat{\xi},z)$, one finds $\mathcal{E}({\rm int}{\rm M})$. Indeed, any $\xi=\mathcal{E}(x)\in \mathcal{E}({\rm int}{\rm M})$ is contained in some projective cylinder $\Pi_{\hat{\xi},\mathcal{D}}$. The restriction of $w_{\hat{\xi}}$ on the neighbourhood $w_{\hat{\xi}}^{-1}(\mathcal{D})$ of $x$ is a composition of the embeddings $\mathcal{E}$ and $\mathfrak{p}_{\hat{\xi}}|_{\mathcal{E}({\rm M})\cap\Pi_{\hat{\xi},\mathcal{D}}}$. Hence, $x$ is a unique and simple zero of $w_{\hat{\xi}}-z$, where $z:=w_{\hat{\xi}}(x)$. Therefore, $(\hat{\xi},z)$ is admissible and, hence, $\xi$ can be founded by using (\ref{GAP application}).

\smallskip

We say that the embedding $\mathcal{E}=\{w_1,\dots,w_n\}$ of the covering space ${\rm M}$ is {\it symmetric} if its components are symmetric with respect to the involution $\tau$ of ${\rm M}$ i.e. $w_k^*=w_k$. In this case, we have $\mathcal{E}\circ\tau=\overline{\mathcal{E}}$ and, hence, $\overline{\mathcal{E}({\rm M})}=\mathcal{E}({\rm M})$. If $\mathcal{E}=\{w_1,\dots,w_n\}$ is a projective holomorphic embedding of ${\rm M}$ into $\mathbb{C}^n$, then $\mathcal{E}=\{w_{1+},\dots,w_{n+},w_{1-},\dots,w_{n-}\}$ (with $w_{\cdot,\pm}$ defined by (\ref{graduation})) is a symmetric holomorphic projective embedding of ${\rm M}$ to $\mathbb{C}^{2n}$. In the rest of the paper, we deal with {\it symmetric holomorphic projective embeddings} only and call them `embeddings' for short. 

\paragraph*{Outline of the proof and the sketch of the paper.} In Sections \ref{sec map iota}--\ref{sec map alpha}, the following notation is accepted. First, $(M,g)$ is an arbitrarily fixed non-orientable surface with the boundary $(\Gamma,dl)$ and the DN map $\Lambda$. We write $(M',g')\in\mathbb{B}(M,t)$ if $(M',g')$ is a surface diffeomorphic to $(M,g)$, $\partial M'=\Gamma$, the length element on $\Gamma$ induced by $g'$ does coincide with $dl$, and the DN map $\Lambda'$ of $(M',g')$ is subject to the estimate 
\begin{equation}
\label{small parameter}
\|\Lambda-\Lambda'\|_{H^{1}(\Gamma;\mathbb{R})\to L_2(\Gamma;\mathbb{R})}\le t
\end{equation}
with small $t>0$. The objects associated with the surface $(M,g)$ are denoted by unprimed symbols, while objects associated with the surface $(M',g')\in\mathbb{B}(M,t)$ are denoted by primed symbols. For example, $({\rm M},{\rm g})$ and $({\rm M}',{\rm g}')$ are the covering spaces of $(M,g)$ and $(M',g')$, respectively, $\mathcal{H}_+({\rm M})$ and $\mathcal{H}_+({\rm M}')$ are the space of symmetric holomorphic smooth functions on them, and so on.

To prove Theorem \ref{main theorem}, it suffices to construct, for any $M'$ above, the diffeomorphism $\beta_{M,M'}$ between $(M,g)$ and $(M',g')$ which is near-conformal for small $t$ (i.e. its dilatation tends to $1$ as $t\to 0$ uniformly with respect to $(M',g')$). The construction of $\beta$ consists of the following steps.

In Section \ref{sec map iota}, we construct the family $\{\iota_{M,M'}\}_{(M',g')\in \Xi(M,t_0)}$ ($t_0>0$) of real-linear maps $\iota_{M,M'}: \ \mathcal{H}_+({\rm M})\to \mathcal{H}_+({\rm M}')$ with the property
\begin{equation}
\label{iota map}
{\rm Tr}'\iota_{M,M'}w\underset{t\to 0}{\longrightarrow} {\rm Tr}w \text{ in } C^{\infty}(\Upsilon;\mathbb{C}) \qquad \forall w\in \mathcal{H}_+({\rm M}).
\end{equation}
where the convergence is uniform with respect to $(M',g')$. In view of the uniqueness of holomorphic continuation and Proposition \ref{trace-DN-map connection 1}, the construction is reduced to the study of stability of solutions to the equation $\mathcal{G}(\eta)=0$ under small perturbations of the DN map $\Lambda$ which defines $\mathcal{G}$ via (\ref{non-linear map}).

In Section \ref{sec embeddings}, we consider the embedding $\mathcal{E}:=\{w_1,\dots,w_n\}$ of the covering space ${\rm M}$ and the {\it induced embedding} $\mathcal{E}_{\rm M'}:=\{w'_1,\dots,w'_n\}$ of the covering space ${\rm M}'$ which is connected with $\mathcal{E}$ by $w'_k:=\iota_{M,M'}w_k$. By using the generalized argument principle (\ref{GAP application}), we show that the images $\mathcal{E}({\rm M})$ and $\mathcal{E}_{\rm M'}({\rm M}')$ are close in $\mathbb{C}^n$ (in Hausdorff distance) for small $t$.

In Section \ref{sec map alpha}, we consider the map $\alpha_{{\rm M},{\rm M}'}: \ \mathcal{E}({\rm M})\to \mathcal{E}_{\rm M'}({\rm M}')$ which sends the point $\xi$ of $\mathcal{E}({\rm M})$ to the nearest point $\xi'=\alpha_{{\rm M},{\rm M}'}(\xi)$ of $\mathcal{E}_{\rm M'}({\rm M}')$. After slight modification near the boundary $\mathcal{E}(\Upsilon)$ of $\mathcal{E}({\rm M})$, the map $\alpha_{{\rm M},{\rm M}'}$ becomes a near-isometric diffeomorphism between $\mathcal{E}({\rm M})$ and $\mathcal{E}_{\rm M'}({\rm M}')$. Since the images $\mathcal{E}({\rm M})$ and $\mathcal{E}_{\rm M'}({\rm M}')$ are conformally equivalent to ${\rm M}$ and ${\rm M}'$, respectively, the map $\alpha_{{\rm M},{\rm M}'}$ provides the near-conformal diffeomorphism $\beta_{{\rm M},{\rm M}'}$ between ${\rm M}$ and ${\rm M}'$. Due to the symmetry of the covering spaces ${\rm M}$ and ${\rm M}'$ and their embeddings, the map $\beta_{{\rm M},{\rm M}'}$ provides the required near-conformal diffeomorphism $\beta_{M,M'}$ between $M$ and $M'$. By this, we complete the proof of Theorem \ref{main theorem}.

In Section \ref{sec generalizations}, we describe the generalizations of the stability result of Theorem \ref{main theorem} for orientable surfaces with multicomponent boundaries. In this case, their DN maps are defined only on the external components of the boundaries. The scheme of proofs is basically the same as described in Sections \ref{sec map iota}--\ref{sec map alpha}. For this reason, we describe only the required modifications in the reasoning.

\section{Connection between spaces of holomorphic functions}
\label{sec map iota}
In this subsection, the operator $\iota_{M,M'}: \ \mathcal{H}_+({\rm M})\to \mathcal{H}_+({\rm M}')$ obeying (\ref{iota map}) is constructed for small $t>0$ and any $(M',g')\in\mathbb{B}(M,t)$. To this end, we consider the non-linear maps $\mathcal{G}$ the non-linear maps $\mathcal{G}$ and $\mathcal{G}'$ related by formula (\ref{non-linear map}) with the DN maps $\Lambda$ and $\Lambda'$ of $(M,g)$ and $(M',g')$, respectively. We write $\mathcal{G}'\in \mathbb{B}_{\mathcal{G}}(M,t)$ if the corresponding surface $(M',g')$ belongs to $\mathbb{B}(M,t)$. In this case, $\Lambda'$ is close to $\Lambda$ in the operator norm and one can consider $\mathcal{G}'$ as small perturbation of $\mathcal{G}$. First, we study the stability of solutions to the non-linear equation $\mathcal{G}(f)=0$ under such perturbations. The key tool in this study is Lemma \ref{main lemma}. As a result, we construct the family $\mathfrak{Y}_{\mathcal{G},\cdot}$ of {\it linear} maps 
$$\mathfrak{Y}_{\mathcal{G},\mathcal{G}'}: \ \mathcal{G}^{-1}(\{0\})\to \mathcal{G}^{'-1}(\{0\})$$
between the spaces of solutions to $\mathcal{G}(f)=0$ and $\mathcal{G}'(f')=0$ which satisfies 
\begin{equation}
\label{C infty convergence}
\mathfrak{Y}_{\mathcal{G},\mathcal{G}'}(f)\underset{t\to 0}{\longrightarrow} f \text{ in } C^{\infty}(\Gamma;\mathbb{R}) \qquad \forall f\in C^{\infty}(\Gamma;\mathbb{R}),
\end{equation}
where the convergence is uniform with respect to $\mathcal{G}'\in \mathbb{B}_{\mathcal{G}}(M,t)$. Next, for each element $\eta:=f\circ\pi+i\sigma((J\Lambda f)\circ\pi+c_f)$ of ${\rm Tr}\mathcal{H}_+({\rm M})$, we put
\begin{align}
\label{iota trace map 1}
\tilde{\iota}_{M,M'}\eta:=f'\circ\pi'\pm i\sigma((J\Lambda'f')\circ\pi'+c'_{f'}, \\
\label{iota trace map 2}
f':=\mathfrak{Y}_{\mathcal{G},\mathcal{G}'}(f), \qquad c'_{f'}:=\mathfrak{N}'(f')/\mathfrak{D}'f'
\end{align}
(as before, we put $c'_{f'}=0$ in the case $f={\rm const}$). In view of Proposition \ref{trace-DN-map connection 1}, formulas (\ref{iota trace map 1}), (\ref{iota trace map 2}) define the (real) linear map $\tilde{\iota}_{M,M'}: \ {\rm Tr}\mathcal{H}_+({\rm M})\to {\rm Tr}\mathcal{H}_+({\rm M}')$. By using (\ref{C infty convergence}), we prove the (uniform with respect to $(M',g')\in\mathbb{B}(M,t)$) convergence 
\begin{equation}
\label{C infty convergence 1}
\tilde{\iota}_{M,M'}\eta\underset{t\to 0}{\longrightarrow}\eta \text{ in } C^{\infty}(\Gamma;\mathbb{C}) \qquad \forall \eta\in {\rm Tr}\mathcal{H}_+({\rm M}).
\end{equation}
Finally, by using to the uniqueness of holomorphic continuation, we define the map $\iota_{M,M'}:={\rm Tr}^{'-1}\tilde{\iota}_{M,M'}{\rm Tr}$. Then the property (\ref{iota map}) follows from (\ref{C infty convergence 1}).

\paragraph{Preliminary estimates.}
In this paragraph, we show that the maps $\mathcal{N}'$, $\mathcal{D}'$, $\mathcal{G}'$, $c'_\cdot$ associated with surfaces $(M',g')\in\mathbb{B}(M,t)$, are close to the maps $\mathcal{N}$, $\mathcal{D}$, $\mathcal{G}$, $c_\cdot$, respectively, for small $t$.
\begin{lemma}
\label{preliminary estimates lemma}
Let $(M,g)$ be a non-orientable surface with the boundary $(\Gamma,dl)$ and DN map $\Lambda$. Then, for any $l=0,1,\dots$, the following statements are valid:
\begin{enumerate}[a{\rm)}]
\item The maps $\mathcal{N}$, $\mathcal{D}$ associated with $\Lambda$ via {\rm(\ref{numerator denominator})} can be extended to continuous operators acting from $C^{l+3}(\Gamma;\mathbb{R})$ to $C^{l}(\Gamma;\mathbb{R})$. The map $\mathcal{G}$ defined by {\rm (\ref{non-linear map})}, can be extended to a continuous operator acting from $C^{l+4}(\Gamma;\mathbb{R})$ to $C^{l}(\Gamma;\mathbb{R})$.
\item For sufficiently small $t\in(0,t_0(M,g)$ and any surface $(M',g')\in\mathbb{B}(M,t)$, the estimates 
\begin{equation}
\label{closeness of NDG}
\begin{split}
\|J\Lambda'f'-J\Lambda f\|_{C^{l}(\Gamma;\mathbb{R})}\le c(t\|f\|_{C^{l+2}(\Gamma;\mathbb{R})}+\|f'-f\|_{C^{l+2}(\Gamma;\mathbb{R})}),\\
\|\mathfrak{N}'(f)-\mathfrak{N}(f)\|_{C^{l}(\Gamma;\mathbb{R})}\le ct\|f\|^2_{C^{l+3}(\Gamma;\mathbb{R})},\\
\|\mathfrak{D}'f-\mathfrak{D}f\|_{C^{l}(\Gamma;\mathbb{R})}\le ct\|f\|_{C^{l+3}(\Gamma;\mathbb{R})},\\
\|\mathcal{G}'(f)-\mathcal{G}(f)\|_{C^{l}(\Gamma;\mathbb{R})}\le ct\|f\|_{C^{l+4}(\Gamma;\mathbb{R})}^{3}.
\end{split}
\end{equation}
are valid for all $f\in C^{\infty}(\Gamma;\mathbb{R})$. Here $\mathcal{N}'$,$\mathcal{D}'$, and $\mathcal{G}'$ are the maps associated with the DN map $\Lambda'$ of $(M',g')$ via {\rm(\ref{numerator denominator})} and {\rm (\ref{non-linear map})}. The constants in {\rm(\ref{closeness of NDG})} depend only on $(M,g)$ and $l$.
\item For any non-constant $f\in \mathcal{G}^{-1}(\{0\})$, the convergence
\begin{equation}
\label{coeff convergence}
\sup |c'_{f'}-c_f|\underset{t\to 0}{\longrightarrow} 0,
\end{equation}
holds, $c_f$ and $c'_{f'}$ are defined by {\rm (\ref{ratio})} and {\rm (\ref{iota trace map 2})}, respectively, and the supremum is taken over all $(M',g')\in\mathbb{B}(M,t)$ and all $f'\in \mathcal{G}^{'-1}(\{0\})$ obeying $\|f'-f\|_{C^{3}(\Gamma;\mathbb{R})}\le t$.
\end{enumerate}
\end{lemma}
\begin{proof}
It is established in \cite{LeeU} that any DN map is a first order pseudo-differential operator. On the space of such operators, the equivalence of the norms 
$$\|\cdot\|_{H^{l}(\Gamma;\mathbb{R})\to H^{l-1}(\Gamma;\mathbb{R})}\asymp\|\cdot\|_{H^{l'}(\Gamma;\mathbb{R})\to H^{l'-1}(\Gamma;\mathbb{R})}$$ 
(where $H^{l}(\Gamma;\mathbb{R})$ denotes the Sobolev space on the curve $(\Gamma,dl)$) is valid for each $l,l'\in\mathbb{N}$. Also, the embeddings $C^{l}(\Gamma;\mathbb{R})\subset H^{l}(\Gamma;\mathbb{R})$ and $H^{l+1}(\Gamma;\mathbb{R})\subset C^{l}(\Gamma;\mathbb{R})$ are continuous for each $l\in\mathbb{N}$. Due to the facts above, for each $l=0,1,\dots$, the DN map $\Lambda$ can be extended to a continuous operator acting from $C^{l+2}(\Gamma;\mathbb{R})$ to $C^{l}(\Gamma;\mathbb{R})$ and condition (\ref{small parameter}) provides the estimate
\begin{equation}
\label{small parameter 1}
\|\Lambda'-\Lambda\|_{C^{l+2}(\Gamma;\mathbb{R})\to C^{l}(\Gamma;\mathbb{R})}\le c_l t,
\end{equation}
which is valid for any DN map $\Lambda'$ of a surface $(M',g')\in\mathbb{B}(M,t)$. Now, {\it a}) and {\it b}) follow from (\ref{small parameter 1}) and the facts that, for each $l=0,1,\dots$, the differentiation $\partial_\gamma$ acts continuously from $C^{l+1}(\Gamma;\mathbb{R})$ to $C^{l}(\Gamma;\mathbb{R})$, the integration $J$ is a continuous operator acting from $\partial_\gamma C^{l+1}(\Gamma;\mathbb{R})\subset C^{l}(\Gamma;\mathbb{R})$ to $C^{l+1}(\Gamma;\mathbb{R})$, and the multiplication $f,h\to fh$ is continuous on $C^{l}(\Gamma;\mathbb{R})$, $\|fh\|_{C^{l}(\Gamma;\mathbb{R})}\le c\|f\|_{C^{l}(\Gamma;\mathbb{R})}\|h\|_{C^{l}(\Gamma;\mathbb{R})}$. It remains to prove {\it c}). Let $\gamma_0$ be a point in which $|\mathfrak{D}f|$ attains maximum on $\Gamma$. In view of (\ref{ratio}) and (\ref{iota trace map 2}), we have 
\begin{equation}
\label{non-small denominator}
\sup|c'_{f'}-c_{f}|=\sup\frac{|[\mathfrak{N}(f)](\gamma_0)[\mathfrak{D}'f'](\gamma_0)-[\mathfrak{N}'(f')](\gamma_0)[\mathfrak{D}f](\gamma_0)|}{|[\mathfrak{D}f](\gamma_0)[\mathfrak{D}'f'](\gamma_0)|}
\end{equation}
(here and in the subsequent, the supremum is taken over the same set as in (\ref{coeff convergence})). Due to the continuity of $\mathfrak{N}$, we have ${\rm sup}|[\mathfrak{N}(f')](\gamma_0)-[\mathfrak{N}(f)](\gamma_0)|\to 0$ as $t\to 0$ while (\ref{closeness of NDG}) implies $\sup |[\mathfrak{N}(f')](\gamma_0)-[\mathfrak{N}(f')](\gamma_0)|\le ct\to 0$. The same convergences remain valid after replacing $\mathfrak{N}$,$\mathfrak{N}'$ by $\mathfrak{D}$,$\mathfrak{D}'$. In particular, for small $t$, the denominator in (\ref{non-small denominator}) is is greater that $\|\mathfrak{D}f\|_{C(\Gamma;\mathbb{R})}^2/2$. From these facts, the convergence (\ref{coeff convergence}) follows.
\end{proof}

\paragraph{Auxiliary lemma.} Let $E$ and $F$ be normed spaces over $\mathbb{R}$ and $\mathcal{G}: \ E\to F$ be a continuous (possibly non-linear) map. We say that $\mathcal{G}$ is $(m,\alpha)-${\it admissible} if 
\begin{enumerate}[(a)]
\item $\mathcal{G}$ is a positive homogeneous map of degree $\alpha$ i.e. $\mathcal{G}(cf)=c^{\alpha}\mathcal{G}(f)$ for all $f\in E$, $c>0$;
\item the null set $\mathcal{G}^{-1}(\{0\})$ of $\mathcal{G}$ is a subspace of codimension $m$ in $\mathcal{E}$.
\end{enumerate}
The set $\mathfrak{Q}_{m,\alpha}(E;F)$ of all $(m,\alpha)-$admissible maps from $E$ to $F$ is endowed with the metrics $\epsilon$,
\begin{equation}
\label{unterlinear metrics}
\epsilon(\mathcal{G}',\mathcal{G})=\sup_{f}\|\mathcal{G}'(f)-\mathcal{G}(f)\|_F,
\end{equation}
where the supremum is taken over all $f\in E$ such that $\|f\|_{E}=1$. We denote the closed ball of radius $t$ in $\mathfrak{Q}_{m,\alpha}(E;F)$ with center at $\mathcal{G}$ by $\mathfrak{B}_{m,\alpha}(\mathcal{G},t)$.

An example of $(m,\alpha)-$admissible map is $\mathcal{G}: \ E\to F$, where 
\begin{equation}
\label{spaces E F}
\begin{split}
E&=\mathcal{C}^{\infty}(\Gamma;\mathbb{R}), \qquad \|\cdot\|_E:=\|\cdot\|_{C^{l+4}(\Gamma;\mathbb{R})}, \\
F&=\mathcal{C}^{\infty}(\Gamma;\mathbb{R}), \qquad \|\cdot\|_F:=\|\cdot\|_{C^{l}(\Gamma;\mathbb{R})}, \\
m&=-\chi(M), \quad \alpha=3.
\end{split}
\end{equation}
with arbitrary $l=0,1,\dots$, and $\mathcal{G}$ is associated with the DN operator $\Lambda$ of the surface $(M,g)$ via formula (\ref{non-linear map}). Indeed, in view of Lemma \ref{preliminary estimates lemma}, $\mathcal{G}: \ E\to F$ is continuous. The property (a) with $\alpha=3$ follows from Proposition \ref{trace-DN-map connection 1}, {\it a}) while (b) is provided by Proposition \ref{trace-DN-map connection 1}, {\it c}). Therefore, we have $\mathcal{G}\in\mathfrak{Q}_{m,\alpha}(E;F)$. Also, the inclusion $\mathcal{G}'\in\mathfrak{Q}_{m,\alpha}(E;F)$ is valid for any map $\mathcal{G}'$ associated with the DN operator $\Lambda'$ of the surface $(M',g')\in\mathbb{B}(M,t)$ since, for any such surface, we have $\chi(M')=\chi(M)$. Moreover, the last estimate in (\ref{closeness of NDG}) implies that $\mathcal{G}'\in\mathfrak{B}_{m,\alpha}(\mathcal{G},ct)$ for any $(M',g')\in\mathbb{B}(M,t)$.

Let us return to the general case. The following lemma states the existence of a family of `almost identical' (in some sense) linear maps between the null set of the map $\mathcal{G}\in\mathfrak{Q}_{m,\alpha}(E;F)$ and the null sets of maps from a sufficiently small neighbourhood $\mathfrak{B}_{m,\alpha}(\mathcal{G},t)$ of $\mathcal{G}$.
\begin{lemma}
\label{main lemma}
Let $\alpha\ge 1$ and $m\in\mathbb{N}$. For any $\mathcal{G}\in \mathfrak{Q}_{m,\alpha}(E;F)$, there exists a positive number $t_0=t_0(\mathcal{G})$ and a family $\{\mathfrak{Y}_{\mathcal{G},\mathcal{G}'}\}_{\mathcal{G}'\in \mathfrak{B}_{m,\alpha}(\mathcal{G},t_0)}$ of linear maps $\mathfrak{Y}_{\mathcal{G},\mathcal{G}'}: \ \mathcal{G}^{-1}(\{0\})\to \mathcal{G}^{'-1}(\{0\})$ satisfying
$$\sup_{\mathcal{G}'\in \mathfrak{B}_{m,\alpha}(\mathcal{G},t)}\|\mathfrak{Y}_{\mathcal{G},\mathcal{G}'}f-f\|_E\underset{t\to 0}{\longrightarrow} 0, \qquad \forall f\in \mathcal{G}^{-1}(\{0\}).$$
\end{lemma}
\begin{proof}
Note that the map $\tilde{\mathcal{G}}: \ f\mapsto \|f\|_E^{1-\alpha}\mathcal{G}(f)$ satisfies the conditions of Lemma \ref{main lemma} while $\tilde{\mathcal{G}}^{-1}(\{0\})=\mathcal{G}^{-1}(\{0\})$. Thus, it is sufficient to consider the case $\alpha=1$.

Let $\mathcal{G}\in \mathfrak{Q}_{m,1}(E;F)$. In view of (b), one can chose $h_1,\dots,h_m\in E$ which are linear independent modulo $\mathcal{G}^{-1}(\{0\})$. For $f\in E$, introduce the continuous function $\mathcal{G}_\cdot: \ E\times\mathbb{R}^m\mapsto [0,+\infty)$ by 
\begin{equation*}
\mathcal{G}_f(\vec{d}):=\|\mathcal{G}(f-\sum_{k=1}^{m}d_k h_k)\|_F,
\end{equation*}
where $\vec{d}=(d_1,\dots,d_m)^T$. For any $f\in \mathcal{G}^{-1}(\{0\})$, the function $\mathcal{G}_f$ is strictly positive $\mathcal{G}_f\ge c_{\mathcal{G},f}>0$ on the unit sphere $S^{m-1}$. 

Note that the norm $|\vec{d}|_E=\|\sum_{k=1}^{m}d_k h_k\|_E$ is equivalent to the standard norm $|\vec{d}|$ in $\mathbb{R}^n$. Let $\vec{d}\in S^{n-1}$, then $\mathcal{G}_0\ge c_{\mathcal{G},0}>0$ on $S^{n-1}$. Since the function $f\mapsto\mathcal{G}_f(\vec{d})$ is continuous, we have 
\begin{equation}
\label{main lemma ineq 1}
\mathcal{G}_f(\vec{d})\ge c_{\mathcal{G},0}/2 \qquad (f\in E, \|f\|_E\le q_{\mathcal{G}}(\vec{d})),
\end{equation}
where $q_{\mathcal{G}}(\vec{d})>0$ is sufficiently small. Let $U(\vec{d})$ be a neighbourhood of $\vec{d}$ in $S^{m-1}$ which consists of all $\vec{e}$ such that $|\vec{d}-\vec{e}|_E\le q_{\mathcal{G}}(\vec{d})/2$. By substitution $f=\sum_{k=1}^{m}(d_k-e_k)h_k+\tilde{f}$ into (\ref{main lemma ineq 1}), we show that $\mathcal{G}_{\tilde{f}}\ge c_{\mathcal{G},0}/2$ on $U(\vec{d})$ for any $\tilde{f}\in E$ such that $\|\tilde{f}\|_E\le q_{\mathcal{G}}(\vec{d})/2$. Since the sphere $S^{m-1}$ is a compact, one can choose the finite sub-cover from the neighbourhoods $U(\vec{d})$ ($\vec{d}\in S^{m-1}$) and obtain the inequality 
\begin{equation}
\label{main lemma ineq 2}
\mathcal{G}_f\ge c_{\mathcal{G},0}/2 \text{ on } S^{m-1} \qquad (f\in E, \|f\|_E\le q_{\mathcal{G}})
\end{equation}
with sufficiently small $q_{\mathcal{G}}$. In particular, we obtain
\begin{equation}
\label{main lemma ineq 3}
\mathcal{G}_f(\vec{d})=|\vec{d}|\mathcal{G}_{|\vec d|^{-1}f}\Big(\frac{\vec{d}}{|\vec{d}|}\Big)\ge c_{\mathcal{G},0}|\vec{d}|/2 \qquad \Big(f\in E, \vec{d}\in\mathbb{R}^m |\vec{d}|\ge\frac{\|f\|_E}{q_{\mathcal{G}}}\Big).
\end{equation}

Due to (a) and (\ref{unterlinear metrics}), we have
$$|\mathcal{G}'_0(\vec{d})|\ge |\mathcal{G}_0(\vec{d})|-|\vec{d}|_E\Big|\mathcal{G}'_0\Big(\frac{\vec{d}}{|\vec{d}|_E}\Big)-\mathcal{G}_0\Big(\frac{\vec{d}}{|\vec{d}|_E}\Big)\Big|\ge c_{\mathcal{G},0}-t|\vec{d}|_E\ge c_{\mathcal{G},0}/2$$
for sufficiently small $t\in [0,t_0]$ and any $\mathcal{G}'\in B_{m,1}(\mathcal{G},t)$, $\vec{d}\in S^{n-1}$. Thus, for any $\mathcal{G}'\in \mathfrak{B}_{m,1}(\mathcal{G},t)$ the elements $h_1,\dots,h_m\in E$ are linear independent modulo $\mathcal{G}^{'-1}(\{0\})$. So, for any $f\in \mathcal{G}^{-1}(\{0\})$, $t\in[0,t_0]$ and $\mathcal{G}'\in \mathfrak{B}_{m,1}(\mathcal{G},t)$, there is the unique decomposition 
\begin{equation}
\label{main lemma decomposition}
f=\sum_{k=1}^m d_{k,\mathcal{G}'}(f)h_k+\tilde{f}_{\mathcal{G}'},
\end{equation}
where $d_{k,\mathcal{G}'}\in\mathbb{R}$ and $\tilde{f}_{\mathcal{G}'}\in \mathcal{G}^{'-1}(\{0\})$. Denote $\vec{d}_{\mathcal{G}'}(f)=(d_{1,\mathcal{G}'}(f),\dots,d_{m,\mathcal{G}'}(f))^T$. We introduce the family $\{\mathfrak{Y}_{\mathcal{G},\mathcal{G}'}\}_{\mathcal{G}'\in \mathfrak{B}_{m,1}(\mathcal{G},t_0)}$ by the rule
\begin{equation}
\label{main lemma rule}
\mathfrak{Y}_{\mathcal{G},\mathcal{G}'}f:=\tilde{f}_{\mathcal{G}'} \qquad (f\in \mathcal{G}^{-1}(\{0\})),
\end{equation}
where $\tilde{f}_{\mathcal{G}'}$ is the remainder in (\ref{main lemma decomposition}). Note that the coefficients $d_{k,\mathcal{G}'}(f)$ and the reminder $\tilde{f}_{\mathcal{G}'}$ depend linearly on $f$ due to (b). Thus, each map $\mathfrak{Y}_{\mathcal{G},\mathcal{G}'}$ is linear.

Assume that $f\in \mathcal{G}^{-1}(\{0\})$ and $\mathcal{G}'\in \mathfrak{B}_{m,1}(\mathcal{G},t)$. Due to a) and (\ref{unterlinear metrics}), the inequality
\begin{align*}
\|\mathcal{G}'(\tilde{f}_{\mathcal{G}'})-\mathcal{G}(\tilde{f}_{\mathcal{G}'})\|_F=\|\tilde{f}_{\mathcal{G}'}\|_E \Big\|\mathcal{G}'\Big(\frac{\tilde{f}_{\mathcal{G}'}}{\|\tilde{f}_{\mathcal{G}'}\|_E}\Big)-\mathcal{G}\Big(\frac{\tilde{f}_{\mathcal{G}'}}{\|\tilde{f}_{\mathcal{G}'}\|_E}\Big)\Big\| \le \\ \le t\|\tilde{f}_{\mathcal{G}'}\|_E \le t(\|f\|_E+|\vec{d}_{\mathcal{G}'}(f)|_E)\le ct(\|f\|_E+|\vec{d}_{\mathcal{G}'}(f)|)
\end{align*}
is valid for the remainder $\tilde{f}_{\mathcal{G}'}$ in (\ref{main lemma decomposition}). Since $0=\mathcal{G}'(\tilde{f}_{\mathcal{G}'})=\mathcal{G}(\tilde{f}_{\mathcal{G}'})+\mathcal{G}'(\tilde{f}_{\mathcal{G}'})-\mathcal{G}(\tilde{f}_{\mathcal{G}'})$, we obtain
\begin{equation}
\label{main lemma eq 1}
\mathcal{G}_f(\vec{d}_{\mathcal{G}'}(f))=\|\mathcal{G}(\tilde{f}_{\mathcal{G}'})\|_F\le ct(\|f\|_E+|\vec{d}_{\mathcal{G}'}(f)|),
\end{equation}
where $c$ is independent of $\mathcal{G}'$, $f$, and $t$. 

Let us show that the set 
$$\{\|f\|_{E}^{-1}|\vec{d}_{\mathcal{G}'}(f)| \ | \ \mathcal{G}'\in \mathfrak{B}_{m,1}(\mathcal{G},t_0), f\in \mathcal{G}^{-1}(\{0\})\backslash\{0\}\}$$
is bounded for sufficiently small $t_0>0$. Suppose the contrary; then there exist sequences $t_k\to 0$, $\{f_k\}\subset \mathcal{G}^{-1}(\{0\})\backslash\{0\}$, and $\{\mathcal{G}'_k\}\subset \mathfrak{B}_{m,1}(\mathcal{G},\varepsilon_k)$ such that $\alpha_k:=\|f_k\|_{E}^{-1}||\vec{d}_{\mathcal{G}'_k}(f_k)|\to +\infty$. Denote $\vec{e}_{k}=(\alpha_k\|f_k\|_{E})^{-1}\vec{d}_{\mathcal{G}'}(f_k)$ and, $p_k:=(\alpha_k\|f_k\|_{E})^{-1}f_k$. Due to a), we have $\mathcal{G}_{f_k}(\vec{d}_{\mathcal{G}'_k})=\alpha_k\|f_k\|_{E} \mathcal{G}_{p_k}(\vec{e}_{k})$. Hence, estimate (\ref{main lemma eq 1}) yields
$$\mathcal{G}_{p_k}(\vec{e}_{k})\le ct_k(\alpha_k^{-1}+1)\to 0.$$
On the contrary, we have $\|p_k\|_E=\alpha_k^{-1}\to 0$ and $|\vec{e}_{k}|=1$. Hence, inequality (\ref{main lemma ineq 2}) yields  $\mathcal{G}_{p_k}(\vec{e}_{k})\ge c_{\mathcal{G},0}/2$ for sufficiently large $k$, a contradiction. 

So, one can choose $t_0>0$ and $C>0$ in such a way that $|\vec{d}_{\mathcal{G}'}(f)|\le C\|f\|_{E}$ for any $\mathcal{G}'\in \mathfrak{B}_{m,1}(\mathcal{G},t_0)$ and $f\in \mathcal{G}^{-1}(\{0\})$. Now, (\ref{main lemma eq 1}) takes the form 
\begin{equation}
\label{main lemma eq 3}
\mathcal{G}_f(\vec{d}_{\mathcal{G}'}(f))\le ct\|f\|_{E},
\end{equation}
where the constant $c$ does not depend on $f\in \mathcal{G}^{-1}(\{0\})$, $\mathcal{G}'\in \mathfrak{B}_{m,1}(\mathcal{G},t)$, and $t\in [0,t_0]$. Introduce the function 
\begin{equation}
\label{epsilon}
t\mapsto\varepsilon_{\mathcal{G},f}(t):=\sup\{|\vec{d}|_E \ | \ \vec{d}\in\mathcal{G}_f^{-1}([0,ct\|f\|_{E}])\}.
\end{equation}
Then (\ref{main lemma eq 3}) means that 
$$\|\tilde{f}_{\mathcal{G}'}-f\|_E\le c|\vec{d}_{\mathcal{G}'}(f)|_E\le c\varepsilon_{\mathcal{G},f}(t)$$ 
for any $f\in \mathcal{G}^{-1}(\{0\})$, $\mathcal{G}'\in \mathfrak{B}_{m,1}(\mathcal{G},t)$, and $t\in [0,t_0]$. To complete the proof, it remains to show that $\varepsilon_{\mathcal{G},f}(t)\to 0$ as $t\to 0$ for any $f\in \mathcal{G}^{-1}(\{0\})$. Suppose the contrary, then there is a sequence $\{\vec{d}_k\}\subset\mathbb{R}^m$ such that $|\vec{d}_k|\ge c>0$ and $\mathcal{G}_f(\vec{d}_k)\to 0$. In view of (\ref{main lemma ineq 3}), the sequence $\{\vec{d}_k\}$ is bounded and, hence, one can choose a subsequence which converges to some point $\vec{d}_\infty$ in $\mathbb{R}^m$. We can assume that this subsequence coincides with $\{\vec{d}_k\}$. Since $\mathcal{G}_f$ is continuous on $\mathbb{R}^n$, we have $\mathcal{G}_f(\vec{d}_\infty)=\lim\limits_{k\to\infty}\mathcal{G}_f(\vec{d}_k)=0$. However, since $f\in \mathcal{G}^{-1}(\{0\})$, the function $\mathcal{G}_f$ is strictly positive on $\mathbb{R}^{n}\backslash\{0\}$, a contradiction.
\end{proof}

\paragraph{Constructing the map $\iota_{M,M'}$.} Let $\mathcal{G}$ be a map (\ref{non-linear map}) associated with the surface $(M,g)$, $t_0>0$ be a sufficiently small number, and $l=4,5,\dots$.  Accepting notation (\ref{spaces E F}) and applying Lemma \ref{main lemma}, we construct the family $\{\mathfrak{Y}_{\mathcal{G},\mathcal{G}'} \ | \ \mathcal{G}'\in \Xi_{\mathcal{G}}(M,t_0)\}$ of linear maps $\mathfrak{Y}_{\mathcal{G},\mathcal{G}'}: \ \mathcal{G}^{-1}(\{0\})\to \mathcal{G}^{'-1}(\{0\})$ such that the convergence
\begin{equation}
\label{application of main lemma - est}
\sup_{\mathcal{G}'\in  \mathbb{B}_{\mathcal{G}}(M,t)}\|\mathfrak{Y}_{\mathcal{G},\mathcal{G}'}f-f\|_{C^{l}(\Gamma;\mathbb{R})}\underset{t\to 0}{\longrightarrow} 0
\end{equation}
is valid for any $f\in \mathcal{G}^{-1}(\{0\})$. The maps $\mathfrak{Y}_{\mathcal{G},\mathcal{G}'}$ are defined by formulas (\ref{main lemma decomposition}), (\ref{main lemma rule}), and, in particular, they do not depend on $l$. Note that (\ref{C infty convergence}) follows from formulas (\ref{application of main lemma - est}) with $l=4,5,\dots$.

Now, introduce the family $\{\tilde{\iota}_{M,M'}\}_{(M',g')\in \Xi(M,t_0)}$ of linear maps $\tilde{\iota}_{M,M'}: \ {\rm Tr}\mathcal{H}_+({\rm M})\to {\rm Tr}\mathcal{H}_+({\rm M}')$ by formulas (\ref{iota trace map 1}), (\ref{iota trace map 2}) where $\Lambda'$ is a DN map of $(M',g')$ and the maps $\mathfrak{N}'$, $\mathfrak{D}'$, $\mathcal{G}'$ are related to $\Lambda'$ via (\ref{numerator denominator}) and (\ref{non-linear map}). Also, introduce the family $\{\iota_{M,M'}\}_{(M',g')\in \Xi(M,t_0)}$ of maps $\iota_{M,M'}:={\rm Tr}^{'-1}\tilde{\iota}_{M,M'}{\rm Tr}$. In view of (\ref{application of main lemma - est}), the first estimate in (\ref{closeness of NDG}), and (\ref{coeff convergence}), the convergence 
\begin{equation}
\label{application of main lemma - est 1}
\sup_{(M',g')\in \mathbb{B}(M,t)}\|\tilde{\iota}_{M,M'}\eta-\eta\|_{C^{l}(\Gamma;\mathbb{R})}\underset{t\to 0}{\longrightarrow} 0
\end{equation}
is valid for any $\eta\in {\rm Tr}\mathcal{H}_+({\rm M})$ and $l=0,1,\dots$. In particular, we obtain (\ref{C infty convergence 1}) and (\ref{iota map}).

\paragraph{Estimates of $\tilde{\iota}_{M,M'}\eta-\eta$.} Let us estimate the left-hand sides of (\ref{application of main lemma - est}) and (\ref{application of main lemma - est 1}). To this end, it is sufficient to estimate the function (\ref{epsilon}) from the proof of Lemma \ref{main lemma}. First, note that $\mathfrak{N}(f)$ from (\ref{numerator denominator}) can be represented as $\mathfrak{N}(f)$ as $\mathfrak{N}(f)=\frac{1}{2}\mathfrak{Q}(f,f)$, where $\mathfrak{Q}$ is a bilinear map given by
$$\mathfrak{Q}(f,h):=\Lambda[fh-(J\Lambda f)(J\Lambda h)]-f\Lambda h-h\Lambda f-(J\Lambda f)\partial_\gamma h-(J\Lambda h)\partial_\gamma f.$$
Then $\mathfrak{N}(f+h)=\mathfrak{N}(f)+\mathfrak{Q}(f,h)+\mathfrak{N}(h)$. Hence, in view of (\ref{non-linear map}), (\ref{numerator denominator}), we have $\mathcal{G}(f+h)=\sum_{k=0}^{3}\mathcal{G}^{(k)}_{f}(h)$, where $\mathcal{G}^{(0)}_{f}(h):=\mathcal{G}(f)$, $\mathcal{G}^{(3)}_{f}(h):=\mathcal{G}(h)$, and 
\begin{align*}
\mathcal{G}^{(1)}_{f}(h):=\partial_{\gamma}\mathfrak{N}(f)\mathfrak{D}h-\mathfrak{N}(f)\partial_{\gamma}\mathfrak{D}h+(\mathfrak{D}f)\partial_{\gamma}\mathfrak{Q}(f,h)-(\partial_{\gamma}\mathfrak{D}f)\mathfrak{Q}(f,h),\\
\mathcal{G}^{(2)}_{f}(h):= (\mathfrak{D}f)\partial_{\gamma}\mathfrak{N}(h)-(\partial_{\gamma}\mathfrak{D}f)\mathfrak{N}(h)+\partial_{\gamma}\mathfrak{Q}(f,h)\cdot\mathfrak{D}h-\mathfrak{Q}(f,h)\partial_{\gamma}\mathfrak{D}h.
\end{align*}
In view of the formulas above, each $\mathcal{G}^{(k)}_{f}$ is a positive homogeneous map of degree $k$ i.e. $\mathcal{G}^{(k)}_{f}(sh)=s^k \mathcal{G}^{(k)}_{f}(h)$ for any $s>0$ and $f\in E$. Then, in the notation of the proof of Lemma \ref{main lemma}, we have 
\begin{equation}
\label{function G f 1}
\mathcal{G}_f(\vec{d}):=\|\mathcal{G}(f-h)\|_F=\|\sum_{k=1}^{3}\mathcal{G}^{(k)}_{f}(-h)\|_F=\|\sum_{k=1}^{3}\mathcal{G}^{(k)}_{f}(\hat{h})|\vec{d}|^k\|_F,
\end{equation}
where $f\in\mathcal{G}^{-1}(\{0\})$, $h:=\sum_{k=1}^{m}d_k h_k$, and $\hat{h}:=-|\vec{d}|^{-1}h$. Recall that $\mathcal{G}_f\ge c_{\mathcal{G},f}>0$ on $S^{m-1}$. Let us show that there exists a positive number $\tilde{c}_{\mathcal{G},f}$ such that
\begin{equation}
\label{epsilon est 1}
\mathcal{G}_f(\vec{d})\ge \tilde{c}_{\mathcal{G},f}|\vec{d}|^3
\end{equation}
for sufficiently small $\vec{d}$. Suppose the contrary, then there exist the bounded sequences $\vec{d}^{(j)}$ and $\vec{e}^{(j)}:=|\vec{d}^{(j)}|^{-1}\vec{d}^{(j)}$ in $\mathbb{R}^m$, and the sequences $h^{(j)}:=\sum_{k=1}^{m}d^{(j)}_k h_k$ and $\hat{h}^{(j)}:=-\sum_{k=1}^{m}e^{(j)}_k h_k$ in $E$ such that $|\vec{d}^{(j)}|^{-3}\mathcal{G}_f(\vec{d}^{(j)})\to 0$ as $j\to\infty$. Since the sphere $S^{m-1}$ is compact, one can assume that $\vec{d}^{(j)}\to \vec{d}^{(\infty)}$ in $S^{m-1}$ while $\hat{h}^{(j)}\to \hat{h}^{(\infty)}$ in $E$ while $\vec{d}^{(j)}\to \vec{d}^{(\infty)}$ in $\mathbb{R}^m$. From (\ref{function G f 1}), it follows that
\begin{equation}
\label{function G f 2}
\|\sum_{k=1}^{3}\mathcal{G}^{(k)}_{f}(\hat{h}^{(j)})|\vec{d}^{(j)}|^{k-3}\|_F\to 0.
\end{equation}
Also, $\mathcal{G}^{(k)}_{f}(\hat{h}^{(j)})\to \mathcal{G}^{(k)}_{f}(\hat{h}^{(\infty)})$ in $F$ due to the continuity of the maps $\mathcal{G}^{(k)}_{f}: \ E\to F$. If $\vec{d}^{(\infty)}\ne 0$, then passage to the limit in (\ref{function G f 2}) yields $\mathcal{G}_f(\vec{d}^{(\infty)})=0$. If $\vec{d}^{(\infty)}=0$, then (\ref{function G f 2}) implies $\mathcal{G}^{(k)}_{f}(\hat{h}^{(\infty)})=0$ and $\mathcal{G}_f(\vec{e}^{(\infty)})=0$. However, $\mathcal{G}_f\ne 0$ on $\mathbb{R}^m\backslash\{0\}$ a contradiction. Thus, formula (\ref{epsilon est 1}) is valid. In particular, the function (\ref{epsilon}) in the proof of Lemma \ref{main lemma} obeys the inequality
\begin{equation}
\label{epsilon est}
\varepsilon_{\mathcal{G},f}(t)\le c(\|f\|_E t)^{\frac{1}{3}}.
\end{equation}
In view of (\ref{epsilon est}), in addition to convergence (\ref{application of main lemma - est}), we have the estimate
\begin{equation}
\label{application of main lemma - est 1}
\sup_{\mathcal{G}'\in  \mathbb{B}_{\mathcal{G}}(M,t)}\|\mathfrak{Y}_{\mathcal{G},\mathcal{G}'}f-f\|_{C^{l}(\Gamma;\mathbb{R})}\le c_{f,l}t^{\frac{1}{3}}.
\end{equation}
Also, since $\mathfrak{N}(f')-\mathfrak{N}(f)=\mathfrak{Q}(f'-f,f)+\mathfrak{Q}(f,f'-f)$, by repeating of the proof of Lemma \ref{preliminary estimates lemma}, c), we obtain
\begin{equation}
\label{coeff convergence 1}
\sup |c'_{f'}-c_f|\le ct
\end{equation}
instead of (\ref{coeff convergence 1}). Now, from (\ref{application of main lemma - est 1}), (\ref{closeness of NDG}), and (\ref{coeff convergence 1}), it follows that
\begin{equation}
\label{application of main lemma - est 2}
\sup_{(M',g')\in \mathbb{B}(M,t)}\|\tilde{\iota}_{M,M'}\eta-\eta\|_{C^{l}(\Gamma;\mathbb{R})}\le c_{\eta,l}t^{\frac{1}{3}}.
\end{equation}

\section{Closeness of embeddings}
\label{sec embeddings}
Let $(M,g)$ be a non-orientable surface with the boundary $\Gamma$, $({\rm M},{\rm g})$ be its covering space, and $\mathcal{E}=\{w_1,\dots,w_n\}$ be a fixed embedding of $({\rm M},{\rm g})$. For sufficiently small $t\in (0,t_0)$ (where $t_0$ depends on $(M,g)$ and $\mathcal{E}$) and any surface $(M',g')\in\mathbb{B}(M,t)$, we introduce the {\it induced embedding} $\mathcal{E}_{\rm M'}=\iota_{M,M'}\mathcal{E}$ of its covering space $({\rm M}',{\rm g}')$ by the formula
$$\mathcal{E}_{\rm M'}=\{w'_1,\dots,w'_n\}, \qquad w'_k:=\iota_{M,M'}w_k,$$
where the map $\iota_{M,M'}$ is introduced in Section \ref{sec map iota}. In what follows, we denote $\eta_k:=w_k|_\Upsilon$ and $\eta'_k:=w'_k|_\Upsilon=\tilde{\iota}_{M,M'}\eta_k$. In view of (\ref{application of main lemma - est 2}), we have
\begin{equation}
\label{epsilon k}
\sup_{(M',g')\in \mathbb{B}(M,t)}\sum_{k=1}^{n}\|\tilde{\iota}_{M,M'}\eta_k-\eta_k\|_{C^{l}(\Gamma;\mathbb{R})}\le  c_{\mathcal{E},l}t^{\frac{1}{3}}. \qquad (l=0,1,\dots),
\end{equation}
where the constant $c_{\mathcal{E},l}$ depends only on $(M,g)$, $\mathcal{E}$, and $l$. Recall that the {\it Hausdorff distance} $d_{H}(K_1,K_2)$ between two compact subsets $K_1$ and $K_2$ of $\mathbb{C}^n$ is a minimal number $r\ge 0$ such that $r-$neighbourhood of $K_1$ contains $K_2$ and $r-$neighbourhood of $K_2$ contains $K_1$. Note that (\ref{epsilon k}) implies the closeness of boundaries of $\mathcal{E}({\rm M})$ and $\mathcal{E}_{\rm M'}({\rm M}')$, 
\begin{equation}
\label{emb b closeness}
\sup_{(M',g')\in \mathbb{B}(M,t)}d_{H}(\mathcal{E}_{\rm M'}(\Upsilon),\mathcal{E}(\Upsilon))\le c_{\mathcal{E},0}t^{\frac{1}{3}}.
\end{equation}
The purpose of this section is to prove that, for sufficiently small $t$, the map $\mathcal{E}_{\rm M'}: \ {\rm M}'\to\mathbb{C}^n$ is actually embedding and its image $\mathcal{E}_{\rm M'}({\rm M}')$ is close to $\mathcal{E}({\rm M})$ in Hausdorff distance. Namely, we prove the estimate
\begin{equation}
\label{emb closeness}
\sup_{(M',g')\in \mathbb{B}(M,t)}d_{H}(\mathcal{E}_{\rm M'}({\rm M}'),\mathcal{E}({\rm M}))\le c_{\mathcal{E},3}t^{\frac{1}{3}}.
\end{equation}
Along with (\ref{emb closeness}), we prove some more informative estimates of the closeness of the surfaces $\mathcal{E}_{\rm M'}({\rm M}')$ and $\mathcal{E}({\rm M})$, which will also be needed for the proof of Theorem \ref{main theorem}.

The proof of (\ref{emb closeness}) follows the scheme of the paper \cite{BKstab}. We construct the finite cover of the image $\mathcal{E}({\rm M})\subset\mathbb{C}^n$ by projective cylinders of the following two types. The closure of the cylinder $\Pi_{\hat{\xi},\mathcal{D}}$ of the first type in $\mathbb{C}^n$ does not intersect the boundary $\mathcal{E}(\Upsilon)$ of $\mathcal{E}({\rm M})$. Due to (\ref{emb b closeness}), it also does not intersect the boundary $\mathcal{E}_{\rm M'}(\Upsilon)$ of $\mathcal{E}_{\rm M'}({\rm M}')$ for sufficiently small $t$. In this case, we prove the closeness of the points $\xi\in \mathcal{E}({\rm M})$ and $\xi'\in\mathcal{E}_{\rm M'}({\rm M}')$ with the joint projection $z=\mathfrak{p}_{\hat{\xi}}\xi=\mathfrak{p}_{\hat{\xi}}\xi'$ on the cylinder base $D$ by applying the generalized argument principle (\ref{GAP application}) and formula (\ref{epsilon k}). The cylinder $\Pi_{\hat{\xi},\mathcal{D}}$ of the second type does intersect the boundary $\mathcal{E}(\Upsilon)$ while its base $D$ is sufficiently small. In this case, the points $\xi\in \mathcal{E}({\rm M})$ and $\xi'\in\mathcal{E}_{\rm M'}({\rm M}')$ with the joint projection $\mathfrak{p}_{\hat{\xi}}\xi=\mathfrak{p}_{\hat{\xi}}\xi'$ may not exist simultaneously and so we need a less explicit correspondence between $\xi$ and $\xi'$. Also, due to the presence of a small denominator in (\ref{GAP application}), the application of the generalized argument principle for the proof of closeness of $\xi$ and $\xi'$ is no longer straightforward and requires standard but more complicated estimates of the Cauchy-type integrals. As a result, we estimate the Hausdorff distance between parts of $\mathcal{E}({\rm M})$ and $\mathcal{E}_{\rm M'}({\rm M}')$ contained in each cylinder $\Pi_{\hat{\xi},\mathcal{D}}$ of the cover. Also, we prove that these cylinders are also projective for $\mathcal{E}_{\rm M'}({\rm M}')$ and provide the cover of $\mathcal{E}_{\rm M'}({\rm M}')$ for sufficiently small $t$. By this, we prove (\ref{emb closeness}).

For the convenience of the reader, we give a detailed proof of (\ref{emb closeness}) in the rest of the section. Throughout the section, we assume that all constants do not depend on $t$ and $(M',g')\in\mathbb{B}(M,t)$.

\subsection{Estimates in cylinders of the first type.} 
\label{1st type est}
Let $\Pi_{\hat{\xi},\mathcal{D}}:=\mathfrak{p}_{\hat{\xi}}^{-1}(\mathcal{D})$ be a projective cylinder for $\mathcal{E}({\rm M})$ with the base $\mathcal{D}$ and the projection $\mathfrak{p}_{\hat{\xi}}$ given by (\ref{projection}). Suppose that the closure of $\Pi_{\hat{\xi},\mathcal{D}}$ in $\mathbb{C}^n$ does not intersect $\mathcal{E}(\Upsilon)$. Then the closure of its base $\mathcal{D}$ in $\mathbb{C}$ is contained in $w_{\hat{\xi}}({\rm M})\backslash\eta_{\hat{\xi}}(\Upsilon)$, where $w_{\hat{\xi}}$, $\eta_{\hat{\xi}}$ are defined in (\ref{projection}). In particular, we have  $\sup\limits_{z\in \mathcal{D}}|\eta_{\hat{\xi}}-z|\ge c_0>0$ on $\Upsilon$. Let $z\in\mathcal{D}$. Since the restriction of $\mathfrak{p}_{\hat{\xi}}$ on $\mathcal{E}({\rm M})\cap\Pi_{\hat{\xi},\mathcal{D}}$ is an embedding, the restriction of the function $w_{\hat{\xi}}=\mathfrak{p}_{\hat{\xi}}\circ\mathcal{E}=\sum_{j=1}^n\hat{\xi}_j w_j$ on $w_{\hat{\xi}}^{-1}(\mathcal{D})$ is also an embedding. Then there is a unique and simple zero $x$ of the function $w_{\hat{\xi}}-z$ and, in view of the argument principle (\ref{argument principle}), we have ${\rm wind}(\eta_{\hat{\xi}},z)=1$. The $\xi=\Xi(z):=\mathcal{E}(x)$ is a unique point of $\mathcal{E}({\rm M})$ with the projection $\mathfrak{p}_{\hat{\xi}}\xi=z$. Note that the components $\Xi_k:=\mathfrak{p}_{\hat{\xi}}\Xi=w_{k}\circ w_{\hat{\xi}}^{-1}$ of $\Xi$ are holomorphic on $\mathcal{D}$.

Let $t>0$ be sufficiently small, $(M',g')$ be a surface from $\mathbb{B}(M,t)$, and $\mathcal{E}_{\rm M'}=\{w'_1,\dots,w'_n\}$ be the embedding of its covering space $({\rm M}',{\rm g}')$ induced by $\mathcal{E}$. Denote the unit function on ${\rm M}$ (${\rm M'}$) by $w_0$ ($w'_0$). Also, put $w'_{\hat{\xi}}:=\mathfrak{p}_{\hat{\xi}}\circ\mathcal{E}_{\rm M'}=\sum_{j=1}^n\hat{\xi}_j w'_j$ and $\eta'_{\hat{\xi}}:=w'_{\hat{\xi}}|_\Upsilon$. In view of (\ref{epsilon k}), we have 
\begin{equation}
\label{boundary curve est}
\|\eta'_{\hat{\xi}}-\eta_{\hat{\xi}}\|_{C^{l}(\Upsilon;\mathbb{C})}\le c_{\mathcal{E},l}|\hat{\xi}|t^{\frac{1}{3}} \qquad (l=0,1,\dots)
\end{equation}
(here and in the subsequent, we omit dependence of the differentiability index $l$ in the notation of constants). In particular, we have $|\eta'_{\hat{\xi}}-z|\ge c/2>0$ on $\Upsilon$. Hence, we have
\begin{equation}
\label{int estimate out of b}
\Big|\partial_z^l\Big[\frac{1}{2\pi}\int\limits_\Upsilon \eta'_j\frac{\partial_\gamma \eta'_{\hat{\xi}}}{\eta'_{\hat{\xi}}-z}dl-\frac{1}{2\pi}\int\limits_\Upsilon \eta_j\frac{\partial_\gamma \eta_{\hat{\xi}}}{\eta_{\hat{\xi}}-z}dl\Big]\Big|\le c_{\mathcal{E},1}t^{\frac{1}{3}}
\end{equation}
for $j=0,\dots,n$ and $l=0,1,\dots$. Due to the argument principle (\ref{argument principle}), estimate (\ref{int estimate out of b}) with $j=l=0$ implies ${\rm wind}(\eta'_{\hat{\xi}},z)=1$. This means that there is a unique and simple zero $x'$ of the function $w'_{\hat{\xi}}-z$ on ${\rm M}'$ and $\xi'=\Xi_{\rm M'}(z):=\mathcal{E}_{\rm M'}(x')$ is a unique point of $\mathcal{E}_{\rm M'}({\rm M}')$ with the projection $\mathfrak{p}_{\hat{\xi}}\xi'=z$. Since $z\in\mathcal{D}$ is arbitrary, the cylinder $\Pi_{\hat{\xi},\mathcal{D}}$ is projective for $\mathcal{E}_{\rm M'}({\rm M}')$. Also, due to the generalized argument principle (\ref{GAP application}), estimates (\ref{int estimate out of b}) imply
\begin{equation}
\label{closenes out of b}
\|\Xi_{\rm M'}-\Xi\|_{C^l(\mathcal{D};\mathbb{C}^n)}\le c_{\mathcal{E},1}t^{\frac{1}{3}}, \qquad z\in\mathcal{D}, \ l=0,1,\dots.
\end{equation}
Note that the components $\Xi'_k:=\mathfrak{p}_{\hat{\xi}}\Xi'=w'_{k}\circ w_{\hat{\xi}}^{'-1}$ of $\Xi'$ are holomorphic on $\mathcal{D}$. Denote $\mathfrak{A}_{\rm M'}:=\Xi_{\rm M'}\circ\Xi^{-1}$, then the map $\mathfrak{A}_{\rm M'}: \ \mathcal{E}({\rm M})\cap\Pi_{\hat{\xi},\mathcal{D}}\to \mathcal{E}_{\rm M'}({\rm M}')\cap\Pi_{\hat{\xi},\mathcal{D}}$ is a diffeomorphism. Also, estimate (\ref{closenes out of b}) with $l=0$  implies
\begin{equation}
\label{closenes out of b 1}
|\mathfrak{A}_{\rm M'}\xi-\xi|\le c_{\mathcal{E},1}t^{\frac{1}{3}}, \qquad \xi\in\mathcal{E}({\rm M})\cap\Pi_{\hat{\xi},\mathcal{D}}.
\end{equation}
So, any point $\xi_0\in\mathcal{E}({\rm M})\backslash\mathcal{E}(\Upsilon)$ is contained in some cylinder $\Pi_{\hat{\xi},\mathcal{D}}$ such that,for sufficiently small $t\in(0,t(\xi_0))$ and any $(M',g')\in\mathbb{B}(M,t)$, the following properties are valid:
\begin{enumerate}[(1')]
\item $\Pi_{\hat{\xi},\mathcal{D}}$ is projective for both $\mathcal{E}({\rm M})$ and $\mathcal{E}_{\rm M'}({\rm M}')$ and its closure does not intersect both $\mathcal{E}(\Upsilon)$ and $\mathcal{E}_{\rm M'}(\Upsilon)$.
\item the maps $\Xi:\ \mathcal{D}\to\mathcal{E}({\rm M})\cap\Pi_{\hat{\xi},\mathcal{D}}$ and $\Xi_{\rm M'}:\ \mathcal{D}\to\mathcal{E}_{\rm M'}({\rm M}')\cap\Pi_{\hat{\xi},\mathcal{D}}$, which are inverse to the projection $\mathfrak{p}_{\hat{\xi}}$, satisfy estimate (\ref{closenes out of b}). 
\item the map $\mathfrak{A}_{\rm M'}:=\Xi_{\rm M'}\circ\Xi^{-1}$ is a diffeomorphism obeying estimate (\ref{closenes out of b 1}).
\end{enumerate}
Indeed, since the embedding $\mathcal{E}$ is projective, $\xi_0$ is contained in some projective cylinder $\Pi_{\hat{\xi},\mathcal{D}}$ for $\mathcal{E}({\rm M})$ and the property $\mathcal{E}(\Upsilon)\cap \overline{\Pi_{\hat{\xi},\mathcal{D}}}=\varnothing$ can be achieved by decreasing the diameter of its base $\mathcal{D}$. Then a corollary to the above are the properties (1')-(3').

\subsection{Estimates in cylinders of the second type.} 
Let $s_0$ be an arbitrary point of $\Upsilon$. Since $\mathcal{E}$ is projective, the point $\xi_0:=\mathcal{E}(s_0)$ is contained in some projective cylinder $\Pi_{\hat{\xi},\mathcal{D}}$ for $\mathcal{E}({\rm M})$. Denote $\xi_0:=\mathfrak{p}_{\hat{\xi}}\xi_0=w_{\hat{\xi}}(l_0)=\eta_{\hat{\xi}}(s_0)$ (for definitions of $\mathfrak{p}_{\hat{\xi}}$, $w_{\hat{\xi}}$ and $\eta_{\hat{\xi}}$, see (\ref{projection})). Since $\Pi_{\hat{\xi},\mathcal{D}}$ is projective for $\mathcal{E}({\rm M})$, the restriction of the map $w_{\hat{\xi}}=\mathfrak{p}_{\hat{\xi}}\circ\mathcal{E}$ on $w_{\hat{\xi}}^{-1}(\mathcal{D})$ is an embedding. In particular, we have $\partial_\gamma\eta_{\hat{\xi}}(s_0)\ne 0$. For $s\in\Upsilon$, denote 
\begin{equation}
\label{theta}
\theta(s):=\Re\frac{\eta_{\hat{\xi}}(s)-\eta_{\hat{\xi}}(s_0)}{\partial_\gamma\eta_{\hat{\xi}}(s_0)}.
\end{equation}
The diameter of $\mathcal{D}$ can be chosen sufficiently small so that $\mathcal{D}\backslash\eta_{\hat{\xi}}(\Upsilon)$ has two connected components $\mathcal{D}_\pm$ obeying 
\begin{equation}
\label{connect comp}
\mathfrak{p}_{\hat{\xi}}^{-1}(\mathcal{D}_-)\cap\mathcal{E}({\rm M})=\varnothing, \qquad \mathfrak{p}_{\hat{\xi}}^{-1}(\mathcal{D}_+)\cap\mathcal{E}({\rm M})=\Pi_{\hat{\xi},\mathcal{D}}\cap\mathcal{E}({\rm int}{\rm M}).
\end{equation}
while $\partial_\gamma\theta\ge c_0>0$ on the segment $\Upsilon_{\mathcal{D}}:=\eta_{\hat{\xi}}^{-1}(\eta_{\hat{\xi}}(\Upsilon)\cap\mathcal{D})$ of $\Upsilon$. In particular, the restriction of $\theta$ on $\Upsilon_{\mathcal{D}}$ is an injection. For points $z\in\mathcal{D}$, introduce the coordinates $\Psi(z):=(s(z),\rho(z))$ by
\begin{equation}
\label{loc coordinates near b}
s(z):=\theta^{-1}\Big(\Re\frac{z-\eta_{\hat{\xi}}(s_0)}{\partial_\gamma\eta_{\hat{\xi}}(s_0)}\Big), \qquad \rho(z):=\frac{z-\eta_{\hat{\xi}}(s(z))}{i\partial_\gamma\eta_{\hat{\xi}}(s_0)}.
\end{equation}
or, conversely,
\begin{equation}
\label{loc coordinates near b inv}
z=\eta_{\hat{\xi}}(s)+i\partial_\gamma\eta_{\hat{\xi}}(s_0)\rho.
\end{equation}
Let us explain the choice of coordinates $(l,\rho)$. Denote the tangent line to the curve $\eta_{\hat{\xi}}(\Upsilon)$ in $\mathbb{C}$ at the point $\eta_{\hat{\xi}}(s_0)$ by $L_0$. Then $\eta_{\hat{\xi}}(s(z))$ is a point of the curve $\eta_{\hat{\xi}}(\Upsilon)\cap\mathcal{D}$ which has the same orthogonal projection on $L_0$ as $z$. Also, for $z\in\mathcal{D}_\pm$, the number $\pm|\partial_\gamma\eta_{\hat{\xi}}(s_0)|\rho(z)$ is the distance from $z$ to $\eta_{\hat{\xi}}(l(z))$. Note that $\Psi(\eta_{\hat{\xi}}(l))=(l,0)$.

Let $t>0$ be sufficiently small, $(M',g')$ be an arbitrary surface from $\mathbb{B}(M,t)$, and $\mathcal{E}_{\rm M'}=\{w'_1,\dots,w'_n\}$ be the embedding of its covering space $({\rm M}',{\rm g}')$ induced by $\mathcal{E}$. Let $w'_{\hat{\xi}}:=\mathfrak{p}_{\hat{\xi}}\circ\mathcal{E}_{\rm M'}=\sum_{j=1}^n\hat{\xi}_j w'_j$ and $\eta'_{\hat{\xi}}:=w'_{\hat{\xi}}|_\Upsilon$; then estimate (\ref{boundary curve est}) is valid. Define the function $\theta'$ by formula (\ref{theta}) with $\theta$, $\eta_{\hat{\xi}}$ replaced by $\theta'$, $\eta'_{\hat{\xi}}$. In view of (\ref{boundary curve est}) and the inequality $\partial_\gamma\eta_{\hat{\xi}}(s_0)\ne 0$, the estimate 
\begin{equation}
\label{theta estimate}
\|\theta'-\theta\|_{C^{l}(\Upsilon_{\mathcal{D}};\mathbb{R})}\le c_{\mathcal{E},l+1}t^{\frac{1}{3}}
\end{equation}
is valid for each $l=0,1,\dots$. In particular, $\partial_\gamma\theta'\ge c_0/2$ on $\Upsilon_{\mathcal{D}}$. In view of (\ref{boundary curve est}), by halving the diameter of $\mathcal{D}$, one can obtain that $\mathcal{D}\backslash\eta'_{\hat{\xi}}(\Upsilon)$ has two connected components while the restriction $\theta'$ on $\eta_{\hat{\xi}}^{'-1}(\eta'_{\hat{\xi}}(\Upsilon)\cap\mathcal{D})$ is injective. Chose some points $z_\pm\in\mathcal{D}_\pm$, then $|\eta_{\hat{\xi}}-z_\pm|\ge c_1>0$ on $\Gamma$. Denote the connected component of $\mathcal{D}\backslash\eta'_{\hat{\xi}}(\Upsilon)$ containing $z_\pm$ by $\mathcal{D}'_\pm$. From formula (\ref{connect comp}) and the argument principle (\ref{argument principle}) it follows that ${\rm wind}(\eta_{\hat{\xi}},z_+)=1$ and ${\rm wind}(\eta_{\hat{\xi}},z_-)=0$. Hence, estimate (\ref{boundary curve est}) and the inequality $|\eta_{\hat{\xi}}-z_\pm|\ge c_1>0$ imply ${\rm wind}(\eta'_{\hat{\xi}},z_+)=1$ and ${\rm wind}(\eta'_{\hat{\xi}},z_-)=0$. Note that the winding number ${\rm wind}(\eta'_{\hat{\xi}},\cdot)$ is constant on each connected component $\mathcal{D}'_\pm$. Now, the argument principle (\ref{argument principle}) shows that 
\begin{equation}
\label{connect comp 1}
\begin{split}
\mathfrak{p}_{\hat{\xi}}^{-1}&(\mathcal{D}'_-)\cap\mathcal{E}_{\rm M'}({\rm M}')=\varnothing, \\
\mathfrak{p}_{\hat{\xi}}^{-1}&(\mathcal{D}'_+)\cap\mathcal{E}_{\rm M'}({\rm M}')=\Pi_{\hat{\xi},\mathcal{D}}\cap\mathcal{E}_{\rm M'}({\rm int}{\rm M}'), \\
{\rm ord}&(w'_{\hat{\xi}}-z')=1 \qquad (z'\in\mathcal{D}'_+).
\end{split}
\end{equation}
In view of (\ref{connect comp 1}), the cylinder $\Pi_{\hat{\xi},\mathcal{D}}$ is projective for $\mathcal{E}_{\rm M'}({\rm M}')$. 

For points $z'\in\mathcal{D}$, introduce the coordinates $\Psi'(z'):=(s'(z'),\rho'(z'))$ by formulas (\ref{loc coordinates near b}), (\ref{loc coordinates near b inv}) with $z,s,\rho,\eta_{\hat{\xi}},\theta$  replaced by $z',s',\rho',\eta'_{\hat{\xi}},\theta'$, respectively. In view of(\ref{boundary curve est}) and (\ref{theta estimate}), there is the sufficiently small $\rho_0>0$ and an open segment $\tilde{\Gamma}\ni s_0$ of $\Upsilon_\mathcal{D}$ such that $\tilde{\Gamma}\times(-\rho_0,\rho_0)\subset \Psi(\mathcal{D})$ and the maps $\Psi^{'-1}$ are well-defined on $\tilde{\Gamma}\times(-\rho_0,\rho_0)$ for any $(M',g')\in\mathbb{B}(M,t)$. In what follows, we assume that $\tilde{\Gamma}\times(-\rho_0,\rho_0)$ coincides with $\Psi(\mathcal{D})$ and, in particular, $\tilde{\Gamma}$ coincides with $\Upsilon_\mathcal{D}$. Now, denote 
$$\varpi:=\Upsilon_\mathcal{D}\times[0,\rho_0)=\Psi(\mathcal{D}_+).$$ 
Due to (\ref{boundary curve est}) and (\ref{theta estimate}), there is the sub-domain $\tilde{\mathcal{D}}$ of $\mathcal{D}$ such that $\tilde{\mathcal{D}}\cap\mathcal{D}'_+$ is contained in $\Psi^{'-1}(\varpi)$ for any $(M',g')\in\mathbb{B}(M,t)$. In addition, (\ref{boundary curve est}), (\ref{theta estimate}) and (\ref{loc coordinates near b}), (\ref{loc coordinates near b inv}) yield the estimate
\begin{equation}
\label{near boundary map closeness}
\|\Psi^{'-1}-\Psi^{-1}\|_{C^{l}(\varpi;\mathbb{C})}+\|\Psi^{'-1}\circ\Psi\|_{C^{l}(\mathcal{D};\mathbb{C})}+\|\Psi^{-1}\circ\Psi'\|_{C^{l}(\tilde{\mathcal{D}}\cap\mathcal{D}'_+;\mathbb{C})}\le c_{\mathcal{E},l+1}t^{\frac{1}{3}}
\end{equation}
which is valid for each $l=0,1,\dots$. 

For $(s,\rho)\in \varpi$, $\rho>0$, denote the (unique) point $\xi\in\mathcal{E}({\rm M})$ with the projection $\mathfrak{p}_{\hat{\xi}}\xi=\Psi^{-1}(s,\rho)$ by $\Xi(s,\rho)$. Similarly, the (unique) point $\xi'\in\mathcal{E}_{\rm M'}({\rm M})$ with the projection $\mathfrak{p}_{\hat{\xi}}\xi'=\Psi^{'-1}(s,\rho)$ by $\Xi_{\rm M'}(s,\rho)$. Put $\mathfrak{A}_{\rm M'}(\xi)=\Xi_{\rm M'}\circ\Xi^{-1}(\xi)$ for $\xi\in\Pi_{\hat{\xi},\mathcal{D}}\cap\mathcal{E}({\rm M})$. Due to (\ref{connect comp 1}), the map $\mathfrak{A}_{\rm M'}$ is well-defined on $\Pi_{\hat{\xi},\mathcal{D}}\cap\mathcal{E}({\rm M})$ while the inverse map $\mathfrak{A}^{'-1}$ is well-defined at least on $\Pi_{\hat{\xi},\tilde{\mathcal{D}}}\cap\mathcal{E}_{\rm M'}({\rm M}')$. 

Let $\xi=\xi(z)$ be a point of $\mathcal{E}({\rm M})$ with the projection $z:=\mathfrak{p}_{\hat{\xi}}\xi$ and $\xi'=\xi(z')$ be a point of $\mathcal{E}_{\rm M'}({\rm M}')$ with the projection $z':=\mathfrak{p}_{\hat{\xi}}\xi'$. In the rest of the paragraph, we assume that the projections $z$ and $z'$ are related by $\Psi(z)=\Psi'(z')=:(s,\rho)\in\varpi$. Then $\xi=\Xi(s,\rho)$, $\xi'=\Xi'(s,\rho)=\mathfrak{A}_{\rm M'}(\xi)$. In view of the generalized argument principle (\ref{GAP application}), we have 
\begin{equation}
\label{GAP near boundary 1}
\partial_{z'}^{l}\xi'_k=\frac{1}{2\pi i}\partial_{z'}^{l}\int\limits_{\Upsilon}\frac{\eta'_{k}d\eta'_{\hat{\xi}}}{(\eta'_{\hat{\xi}}-z')^{l+1}}=\frac{l!}{2\pi i}\int\limits_{\Upsilon}\frac{\eta'_{k}d\eta'_{\hat{\xi}}}{(\eta'_{\hat{\xi}}-z')^{l+1}},
\end{equation}
where $l=0,1,\dots$. For $\sigma'\in\Upsilon$, denote $\mu':=\eta'_{\hat{\xi}}(\sigma')$. Let $L'_0$ be the tangent line to the curve $\eta'_{\hat{\xi}}(\Upsilon)$ in $\mathbb{C}$ at the point $\eta'_{\hat{\xi}}(s_0)$. Then $\zeta':=\eta'_{\hat{\xi}}(s'(z'))$ is a point of the curve $\eta'_{\hat{\xi}}(\Upsilon)$ which has the same orthogonal projection on $L'_0$ as $z'$. Put $s':=s'(z')$ and
\begin{equation}
\label{var kappa}
\varkappa':=\Re\frac{\mu'-\eta'_{\hat{\xi}}(s_0)}{\partial_\gamma \eta'_{\hat{\xi}}(s_0)}=\theta'(\sigma'), \qquad \kappa':=\Re\frac{\zeta'-\eta'_{\hat{\xi}}(s_0)}{\partial_\gamma \eta'_{\hat{\xi}}(s_0)}=\Re\frac{z'-\eta'_{\hat{\xi}}(s_0)}{\partial_\gamma \eta'_{\hat{\xi}}(s_0)}=\theta'(s').
\end{equation}
We also use the notation obtained by omitting primes in the formulas above. For $\eta'_{k}\circ\theta^{'-1}$, the Tailor formula 
$$\eta'_{k}\circ\theta^{'-1}(\varkappa')=\sum_{q=0}^{l}a'_q(\kappa')(\varkappa'-\kappa')^q+\check{\eta}'_{k,l}(\sigma',s')$$
is valid, where
\begin{equation}
\label{Taylor 2}
a'_q(\kappa'):=\frac{\partial^{q}_{\varkappa'}[\eta'_{k}\circ\theta^{'-1}](\kappa')}{q!}, \quad \check{\eta}'_{k,l}(\sigma',s'):=\frac{1}{l!}\int\limits_{\kappa'}^{\varkappa'}(\varkappa'-y)^l\partial^{l+1}_{\varkappa'}[\eta'_{k}\circ\theta^{'-1}](y)dy.
\end{equation}
In view of (\ref{var kappa}), the sum in the right-hand side is a polynomial $\mathcal{P}'_{k,l}(\mu'-\zeta',\overline{\mu'-\zeta'},s'(z'))$ in $\mu'-\zeta'$, $\overline{\mu'-\zeta'}$ of order $l$ with coefficients depending on $s'(z')$. So, for $\mu'\in\eta'_{\hat{\xi}}(\Upsilon)\cap\mathcal{D}$, formula (\ref{var kappa}) takes the form
\begin{equation}
\label{Taylor 3}
w'_k\circ w^{'-1}_{\hat{\xi}}(\mu')=\eta'_{k}(\sigma')=\mathcal{P}'_{k,l}(\mu'-\zeta',s')+\check{\eta}'_{k,l}(\sigma',s')
\end{equation}
Note that the function $w'_k\circ w^{'-1}_{\hat{\xi}}$ is well-defined and holomorphic on $\mathcal{D}'_+$ and smooth up to $\eta'_{\hat{\xi}}(\Upsilon)\cap\mathcal{D}$ while $\mathcal{P}'_{k,l}(\cdot,s')$ is its $l$-th Taylor polynomial at the point $\mu'$. Due to this, $\mathcal{P}'_{k,l}(\cdot,s')$ contains no terms proportional to $\overline{\mu'-\zeta'}$. Substituting (\ref{Taylor 3}) into (\ref{GAP near boundary 1}) and applying Cauchy's differentiation formula, we obtain
\begin{equation}
\label{GAP near boundary 2}
\begin{split}
\partial_{z'}^{l}\xi'_k=\frac{l!}{2\pi i}\int\limits_{\eta'_{\hat{\xi}}(\Upsilon)}\frac{\mathcal{P}'_{k,l}(\mu'-\zeta',s')d\mu'}{(\mu'-z')^{l+1}}+\frac{l!}{2\pi i}\int\limits_{\Upsilon}\frac{\tilde{\eta}'_{k,l}(\cdot,s')d\eta'_{\hat{\xi}}}{(\eta'_{\hat{\xi}}-z')^{l+1}}=\\
=\partial_{z'}^{l}\mathcal{P}'_{k,l}(\mu'-\zeta',s')+\frac{l!}{2\pi i}\int\limits_{\Upsilon}\frac{\tilde{\eta}'_{k,l}(\cdot,s')d\eta'_{\hat{\xi}}}{(\eta'_{\hat{\xi}}-z')^{l+1}}
\end{split}
\end{equation}
where 
\begin{equation}
\label{Taylor 4}
\tilde{\eta}'_{k,l}(\sigma',s'):=\eta'_{k}(\sigma')-\mathcal{P}'_{k,l}(\mu'-\zeta',s')
\end{equation}
does coincide with $\check{\eta}'_{k,l}(\sigma',s')$ for all $\sigma',s'\in\tilde{\Gamma}$. Note that formulas (\ref{GAP near boundary 1})-(\ref{GAP near boundary 2}) remain valid with omitted primes. In particular, we have 
\begin{equation}
\label{GAP near boundary 3}
\begin{split}
\partial_{z'}^{l}\xi'_k-\partial_{z}^{l}\xi_k=\partial_{z'}^{l}\mathcal{P}'_{k,l}(\mu'-\zeta',s)-\partial_{z}^{l}\mathcal{P}_{k,l}(\mu-\zeta,s)+\\
+\frac{l!}{2\pi i}\int\limits_{\Upsilon}\Big[\frac{\tilde{\eta}'_{k,l}(\cdot,s)d\eta'_{\hat{\xi}}}{(\eta'_{\hat{\xi}}-z')^{l+1}}-\frac{\tilde{\eta}_{k}(\cdot,s)d\eta_{\hat{\xi}}}{(\eta_{\hat{\xi}}-z)^{l+1}}\Big].
\end{split}
\end{equation}
Estimates (\ref{boundary curve est}) and (\ref{theta estimate}) imply 
\begin{equation}
\label{Taylor est 1}
\|a'_q-a_q\|_{C(\Upsilon_{\mathcal{D}};\mathbb{C})}\le c_{\mathcal{E},q+1}t^{\frac{1}{3}}, \quad \|\mathcal{P}'_{k,l}-\mathcal{P}_{k,l}\|_{C(B\times\Upsilon_{\mathcal{D}};\mathbb{C})}\le c_{\mathcal{E},l+1}t^{\frac{1}{3}},
\end{equation}
where $B$ is an arbitrary fixed bounded domain in $\mathbb{C}$. Due to (\ref{boundary curve est}), the curve $\eta'_{\hat{\xi}}(\Upsilon)$ is contained in an arbitrarily fixed neighbourhood of $\eta_{\hat{\xi}}(\Upsilon)$ for sufficiently small $t$. Thus, the argument $\mu'-\zeta'$ ($\mu-\zeta$) of $\mathcal{P}'_{k,l}$ ($\mathcal{P}_{k,l}$) is contained in some fixed bounded domain $B$ for any $(M',g')\in\mathbb{B}(M,t)$ and $s'\in\Upsilon_{\mathcal{D}}$. Therefore, formulas (\ref{Taylor 4}) and (\ref{Taylor est 1}) imply 
\begin{equation}
\label{Taylor est 2}
\|\tilde{\eta}'_{k,l}-\tilde{\eta}'_{k,l}\|_{C(\Upsilon\times\Upsilon_{\mathcal{D}};\mathbb{C})}\le c_{\mathcal{E},l+1}t^{\frac{1}{3}}.
\end{equation}
As a corollary of (\ref{boundary curve est}), (\ref{near boundary map closeness}), (\ref{Taylor est 2}), we obtain
\begin{equation}
\label{int est 1}
\Big|\int\limits_{\Upsilon\backslash\Gamma_\delta(s')}\Big[\frac{\tilde{\eta}'_{k,l}(\cdot,s)\partial_\gamma\eta'_{\hat{\xi}}}{(\eta'_{\hat{\xi}}-z')^{l+1}}-\frac{\tilde{\eta}_{k}(\cdot,s)\partial_\gamma\eta_{\hat{\xi}}}{(\eta_{\hat{\xi}}-z)^{l+1}}\Big]\Big|\le c_{\mathcal{E},l+1}t^{\frac{1}{3}},
\end{equation}
where $\delta>0$ can be chosen arbitrarily small and $\Gamma_\delta(s')$ is a $\delta$-neighbourhood of $s'$ in $\Upsilon$.

The change of variables in the integral in (\ref{Taylor 2}) yields 
\begin{equation}
\label{remainder 1}
\check{\eta}'_{k,l}(\sigma',s'):=(\varkappa'-\kappa')^{l+1}\acute{\eta}'_{k,l}(\sigma',s')
\end{equation}
where
\begin{equation}
\label{remainder 2}
\acute{\eta}'_{k,l}(\sigma',s'):=\frac{1}{l!}\int\limits_{0}^{1}(1-\tau)^l\partial^{l+1}_{\varkappa'}[\eta'_{k}\circ\theta^{'-1}]\big(\kappa'+(\varkappa'-\kappa')\tau\big)d\tau
\end{equation}
Formulas (\ref{remainder 1}), (\ref{remainder 2}) remain valid with omitted primes. From (\ref{remainder 2}), (\ref{var kappa}) and (\ref{boundary curve est}), (\ref{theta estimate}), it follows that 
\begin{equation}
\label{remainder est 1}
\|\acute{\eta}'_{k,l}-\acute{\eta}_{k,l}\|_{C(\Upsilon_{\mathcal{D}}\times\Upsilon_{\mathcal{D}};\mathbb{C})}\le c_{\mathcal{E},l+2}t^{\frac{1}{3}}.
\end{equation}
Next, in view of (\ref{var kappa}), we have
\begin{equation}
\label{fraction 1}
\frac{\varkappa'-\kappa'}{\eta'_{\hat{\xi}}(\sigma')-z'}=\frac{\Re a'(\sigma',s,\rho)}{a(\sigma',s,\rho)},
\end{equation}
where
$$a'(\sigma',s,\rho)=\frac{\eta'_{\hat{\xi}}(\sigma')-\Psi^{'-1}(s,\rho)}{\partial_\gamma \eta'_{\hat{\xi}}(s_0)}.$$
Note that
\begin{equation}
\label{fraction 3}
a'(l,l,0)=0, \quad \partial_{\rho}a'(\sigma',s,\rho)=-\frac{\partial_\rho \Psi^{'-1}(s,\rho)}{\partial_\gamma \eta'_{\hat{\xi}}(s_0)}=-i, \quad \partial_{\gamma}a'(\cdot,s,\rho)=\frac{\partial_{\gamma}\eta'_{\hat{\xi}}}{\partial_\gamma \eta'_{\hat{\xi}}(s_0)}
\end{equation}
due to (\ref{loc coordinates near b inv}). Hence, the formula
$$a'(\sigma,s,\rho)=\int_{s}^{\sigma}\partial_{\tau}a'(\tau,s,0)d\tau-i\rho=\tilde{a}'(\sigma,s)(\sigma-s)-i\rho, $$
holds with
\begin{equation}
\label{fraction 5}
\tilde{a}'(\sigma,s):=\partial_{\sigma'}a'(s,s,0)+\frac{1}{\sigma-s}\int_s^\sigma\int_s^{\tau}\partial^2_{\varsigma}a'(\varsigma,s,0)d\varsigma d\tau.
\end{equation}
Here, for simplicity, the points of $\Upsilon$ and their natural parameters are denoted by the same symbols $\tau$, $\varsigma$. Formulas (\ref{fraction 1})-(\ref{fraction 5}) remain valid with omitted primes. Also, estimates (\ref{boundary curve est}) and (\ref{near boundary map closeness}) imply
$$\|a'-a\|_{C^l(\Upsilon_{\mathcal{D}}\times\Upsilon_{\mathcal{D}}\times(-\rho_0,\rho_0);\mathbb{C})}\le c_{\mathcal{E},l+1}t^{\frac{1}{3}}.$$
Thus, function (\ref{fraction 5}) obeys
\begin{equation}
\label{fraction est 1}
\|\tilde{a}'-\tilde{a}\|_{C(\Upsilon_{\mathcal{D}}\times\Upsilon_{\mathcal{D}};\mathbb{C})}\le c_{\mathcal{E},3}t^{\frac{1}{3}}.
\end{equation}
In view of (\ref{fraction 3}), (\ref{fraction est 1}), for arbitrarily small $\epsilon>0$, one can choose $\delta>0$ and the diameter of $\mathcal{D}$ sufficiently small to obtain $|\tilde{a}(\sigma,s)-1|\le\epsilon/2$ for any $s\in\Upsilon_{\mathcal{D}}$ and $\sigma\in\Gamma_\delta(s)$. Then, for sufficiently small $t$ and any $(M',g')\in\mathbb{B}(M,t)$, we have $|\tilde{a}'(\sigma,s)-1|\le\epsilon$ for any $s\in\Upsilon_{\mathcal{D}}$ and $\sigma\in\Gamma_\delta(s)$. Therefore, for such $s$ and $\sigma$, we have
$$|\tilde{a}(\sigma,s)(\sigma-s)-i\rho|\ge cr, \qquad |\tilde{a}'(\sigma,s)(\sigma-s)-i\rho|\ge cr,$$
where $r=\sqrt{(\sigma-s)^2+\rho^2}$. Thus, estimate (\ref{fraction est 1}) implies
\begin{equation}
\label{fraction est 2}
\begin{split}
\Big|\frac{\varkappa'-\kappa'}{\eta'_{\hat{\xi}}(\sigma)-z'}-\frac{\varkappa-\kappa}{\eta_{\hat{\xi}}(\sigma)-z}\Big|=\Big|\frac{\Re \tilde{a}'(\sigma,s)(\sigma-s)}{\tilde{a}'(\sigma,s)(\sigma-s)-i\rho}-\frac{\Re \tilde{a}'(\sigma,s)(\sigma-s)}{\tilde{a}(\sigma,s)(\sigma-s)-i\rho}\Big|=\\
=\Big|\frac{(\sigma-s)\rho(\tilde{a}'(\sigma,s)-\tilde{a}(\sigma,s))}{(\tilde{a}'(\sigma,s)(\sigma-s)-i\rho)(\tilde{a}(\sigma,s)(\sigma-s)-i\rho)}\Big|\le c_{\mathcal{E},3}t^{\frac{1}{3}}.
\end{split}
\end{equation}
In vie of (\ref{remainder 1}), the integrand in (\ref{GAP near boundary 3}) can be rewritten as
\begin{align*}
\frac{\tilde{\eta}'_{k,l}(\sigma,s')\partial_\gamma\eta'_{\hat{\xi}}}{(\eta'_{\hat{\xi}}-z')^{l+1}}-\frac{\tilde{\eta}_{k}(\sigma,s)\partial_\gamma\eta_{\hat{\xi}}}{(\eta_{\hat{\xi}}-z)^{l+1}}&=\\
=\acute{\eta}'_{k,l}(\sigma,s)\partial_\gamma\eta'_{\hat{\xi}}\Big[\frac{\varkappa'-\kappa'}{\eta'_{\hat{\xi}}(\sigma)-z'}\Big]^{l+1}&-\acute{\eta}_{k,l}(\sigma,s)\partial_\gamma\eta_{\hat{\xi}}\Big[\frac{\varkappa-\kappa}{\eta_{\hat{\xi}}(\sigma)-z}\Big]^{l+1}.
\end{align*}
Then estimates (\ref{remainder est 1}) and (\ref{fraction est 2}) yields
\begin{align*}
\Big|\frac{\tilde{\eta}'_{k,l}(\sigma,s')\partial_\gamma\eta'_{\hat{\xi}}}{(\eta'_{\hat{\xi}}-z')^{l+1}}-\frac{\tilde{\eta}_{k}(\sigma,s)\partial_\gamma\eta_{\hat{\xi}}}{(\eta_{\hat{\xi}}-z)^{l+1}}\Big|\le c_{\mathcal{E},l+3}t^{\frac{1}{3}}
\end{align*}
for any $s\in\Upsilon_{\mathcal{D}}$ and $\sigma\in\Gamma_\delta(s)$. In particular,
\begin{equation}
\label{int est 2}
\Big|\int\limits_{\Gamma_\delta(s')}\Big[\frac{\tilde{\eta}'_{k,l}(\sigma,s)d\eta'_{\hat{\xi}}}{(\eta'_{\hat{\xi}}-z')^{l+1}}-\frac{\tilde{\eta}_{k}(\sigma,s)d\eta_{\hat{\xi}}}{(\eta_{\hat{\xi}}-z)^{l+1}}\Big]\Big|\le c_{\mathcal{E},l+3}t^{\frac{1}{3}}.
\end{equation}
Combining (\ref{GAP near boundary 3}) and inequalities (\ref{Taylor est 1}), (\ref{int est 1}), (\ref{int est 2}), we arrive at
\begin{equation}
\label{closeness near boundary fin 1}
|\partial_{z'}^{l}\xi'_k-\partial_{z}^{l}\xi_k|\le c_{\mathcal{E},l+3}t^{\frac{1}{3}}.
\end{equation}
For $l=0$, estimate (\ref{closeness near boundary fin 1}) takes the form
\begin{equation}
\label{closeness near boundary fin 2}
|\mathfrak{A}_{\rm M'}(\xi)-\xi|\le c_{\mathcal{E},3}t^{\frac{1}{3}}, \qquad \xi\in\Pi_{\hat{\xi},\mathcal{D}}\cap\mathcal{E}({\rm M}).
\end{equation}
Also, estimates (\ref{near boundary map closeness}) and (\ref{closeness near boundary fin 1}) imply
\begin{equation}
\label{closeness near boundary fin 3}
\|\Xi_{\rm M'}-\Xi\|_{C^l(\varpi;\mathbb{C}^n)}\le c_{\mathcal{E},l+3}t^{\frac{1}{3}}.
\end{equation}
So, for any point $\xi_0$ of $\mathcal{E}(\Upsilon)$ there are the cylinder $\Pi_{\hat{\xi},\mathcal{D}}$ and its sub-cylinder $\Pi_{\hat{\xi},\tilde{\mathcal{D}}}$ containing $\xi_0$, the set $\varpi=\Upsilon_{\mathcal{D}}\times[0,\rho_0)$, where $\rho_0>0$ and $\Upsilon_{\mathcal{D}}$ is a segment of $\Upsilon$, and the diffeomorphism $\Xi: \ \varpi\to\mathcal{E}({\rm M})\cap\Pi_{\hat{\xi},\mathcal{D}}$ such that $\Xi(s,0)=\mathcal{E}(s)$ and, for sufficiently small $t\in(0,t(\xi_0))$ and any $(M',g')\in\mathbb{B}(M,t)$, the following properties are valid:
\begin{enumerate}[(1'')]
\item $\Pi_{\hat{\xi},\mathcal{D}}$ is projective for both $\mathcal{E}({\rm M})$ and $\mathcal{E}_{\rm M'}({\rm M}')$.
\item there is the diffeomorphism $\Xi_{\rm M'}: \ \varpi\to\mathcal{E}_{\rm M'}({\rm M})\cap\Pi_{\hat{\xi},\mathcal{D}}$ satisfying estimate (\ref{closeness near boundary fin 3}) and such that $\Xi'_{\rm M'}(s,0)=\mathcal{E}_{\rm M'}(s)$.
\item the diffeomorphism $\mathfrak{A}_{\rm M'}=\Xi'\circ\Xi^{-1}$ obeys (\ref{closeness near boundary fin 2}) and the inverse map $\mathfrak{A}_{\rm M'}^{-1}$ is well-defined on $\mathcal{E}_{\rm M'}({\rm M}')\cap\Pi_{\hat{\xi},\tilde{\mathcal{D}}}$.
\end{enumerate}

\subsection{Completing the proof of (\ref{emb closeness}).} 
\label{ssec Hausdorff stability end of proof}
Let $\xi_0\in\Upsilon$. As showed in the previous paragraph, there is the cylinder $\Pi_{\hat{\xi},\mathcal{D}}=\Pi_{\hat{\xi}(\xi_0),\mathcal{D}(\xi_0)}$ possessing properties (1'')-(3''). By decreasing the diameter of its base $\mathcal{D}(\xi_0)$ one can ensure that (3'') is satisfied with $\tilde{\mathcal{D}}=\mathcal{D}(\xi_0)$. The cylinders $\{\Pi_{\hat{\xi}(\xi_0),\mathcal{D}(\xi_0)}\}_{\xi_0\in\Upsilon}$ constitute an open cover of $\mathcal{E}(\Upsilon)$. Since $\mathcal{E}(\Upsilon)$ is a compact subset of $\mathbb{C}^n$, one can chose a finite sub-cover $\{\Pi_{\hat{\xi}(\xi_0(k)),\mathcal{D}(\xi_0(k))}\}_{k=1}^{K_1}$ of $\mathcal{E}(\Upsilon)$. Next, the set 
$$\widetilde{\mathcal{E}({\rm M})}:=\mathcal{E}({\rm M})\backslash(\bigcup_{k=1}^{K_1}\Pi_{\hat{\xi}(\xi_0(k)),\mathcal{D}(\xi_0(k))})$$
is compact in $\mathbb{C}^n$. 

Now, let $\xi_0\in \widetilde{\mathcal{E}({\rm M})}$. As showed at the beginning of the section, there is the cylinder $\Pi_{\hat{\xi},\mathcal{D}}=\Pi_{\hat{\xi}(\xi_0),\mathcal{D}(\xi_0)}$ possessing properties (1')-(3'). The cylinders $\{\Pi_{\hat{\xi}(\xi_0),\mathcal{D}(\xi_0)}\}_{\xi_0\in\widetilde{\mathcal{E}({\rm M})}}$ constitute an open cover of the compact subset $\widetilde{\mathcal{E}({\rm M})}$ of $\mathbb{C}^n$. Thus, one can chose a finite sub-cover $\{\Pi_{\hat{\xi}(\xi_0(k)),\mathcal{D}(\xi_0(k))}\}_{k=K_1+1}^{K_2}$ of $\widetilde{\mathcal{E}({\rm M})}$. So, $\{\Pi_{\hat{\xi}(\xi_0(k)),\mathcal{D}(\xi_0(k))}\}_{k=1}^{K_2}$ is a finite open cover of $\mathcal{E}({\rm M})$. Then, considering only the cylinders from this cover, one can assume that estimates (\ref{closenes out of b 1}) and (\ref{closeness near boundary fin 2}) (provided by properties (3') and (3''), respectively) are valid for sufficiently small $t\in (0,t_0)$ with constants depending only on $({\rm M},{\rm g})$ and $\mathcal{E}$. 

Put $\mathfrak{O}:=\bigcup_{k=1}^{K_2}\Pi_{\hat{\xi}(\xi_0(k)),\mathcal{D}(\xi_0(k))}$. Each point $\xi\in\mathcal{E}({\rm M})$ is contained in some cylinder $\Pi_{\hat{\xi}(\xi_0(k)),\mathcal{D}(\xi_0(k))}$, and, due to estimate (\ref{closenes out of b 1}) or (\ref{closeness near boundary fin 2}) (provided by (3') or (3'')), it is contained in $c_{\mathcal{E},l+3}t^{\frac{1}{3}}-$neighbourhood (in $\mathbb{C}^n$) of the point $\xi'=\mathfrak{A}^{(\xi_0(k))}_{\rm M'}(\xi)$ of $\mathcal{E}_{\rm M'}({\rm M}')$. Conversely, each point $\xi'\in\mathcal{E}_{\rm M'}({\rm M}')\cap\mathfrak{O}$ is contained in some cylinder $\Pi_{\hat{\xi}(\xi_0(k)),\mathcal{D}(\xi_0(k))}$, and, due to (\ref{closenes out of b 1}) or (\ref{closeness near boundary fin 2}), it is contained in $c_{\mathcal{E},l+3}t^{\frac{1}{3}}-$neighbourhood of the point $\xi=(\mathfrak{A}^{(\xi_0(k))}_{\rm M'})^{-1}(\xi)$ of $\mathcal{E}({\rm M})$. So, for any $t\in (0,t_0)$ and $(M',g')\in\mathbb{B}(M,t)$, the image $\mathcal{E}({\rm M})$ is contained in $c_{\mathcal{E},3}t^{\frac{1}{3}}-$neighbourhood of $\mathcal{E}_{\rm M'}({\rm M}')$ while the set $\mathcal{E}_{\rm M'}({\rm M}')\cap\mathfrak{O}$ is contained in $c_{\mathcal{E},3}t^{\frac{1}{3}}-$neighbourhood of $\mathcal{E}({\rm M})$. By definition of the Hausdorff distance, formula (\ref{emb closeness}) is valid with $\mathcal{E}_{\rm M'}({\rm M}')$ replaced by $\mathcal{E}_{\rm M'}({\rm M}')\cap\mathfrak{O}$. Since $\mathcal{E}_{\rm M'}({\rm M}')$ is path-connected and $\mathcal{E}_{\rm M'}(\Upsilon)$ is contained in $\mathfrak{O}$ due to (\ref{emb b closeness}), by decreasing $t_0$, we obtain
\begin{equation}
\label{unioun of cylinders}
\mathcal{E}_{\rm M'}({\rm M}')\cap\mathfrak{O}=\mathcal{E}_{\rm M'}({\rm M}') \qquad \forall (M',g')\in\mathbb{B}(M,t_0).
\end{equation}
By this, we complete the proof of (\ref{emb closeness}). Also due to (\ref{unioun of cylinders}) and properties (1'), (1''), each point $\xi'\in \mathcal{E}_{\rm M'}({\rm M}')$ is contained in some projective (for $\mathcal{E}_{\rm M'}({\rm M}')$) cylinder $\Pi_{\hat{\xi}(\xi_0(k)),\mathcal{D}(\xi_0(k))}$. Therefore, for any $(M',g')\in\mathbb{B}(M,t_0)$, the map $\mathcal{E}_{\rm M'}$ is a projective embedding and $\mathcal{E}_{\rm M'}({\rm M}')$ is a surface. Also, since $w'_k=\iota_{M,M'}w_k\in\mathcal{H}_+({\rm M}')$, the embedding $\mathcal{E}_{\rm M'}$ is symmetric. These facts justify the use of the term `induced embedding' for the map $\mathcal{E}_{\rm M'}$.

\section{Constructing near-isometric diffeomorphism between $\mathcal{E}({\rm M})$ and $\mathcal{E}_{\rm M'}({\rm M}')$}
\label{sec map alpha}
Let $(M,g)$ be a non-orientable surface with the boundary $\Gamma$, $({\rm M},{\rm g})$ be its covering space, and $\mathcal{E}=\{w_1,\dots,w_n\}$ be a fixed embedding of $({\rm M},{\rm g})$. Let also  $t\in (0,t_0)$ (where $t_0=t_0(({\rm M},{\rm g},\mathcal{E})$), $(M',g')$ be a surface from $\mathbb{B}(M,t)$, and $\mathcal{E}_{M'}$ be the embedding of its covering space $({\rm M}',{\rm g}')$ induced by $\mathcal{E}$. The images $\mathcal{E}({\rm M})$ and $\mathcal{E}'({\rm M}')$ are the surface in $\mathbb{C}^n$ endowed with the metrics induced by $\mathbb{C}^n$. In this section, we construct the diffeomorphism $\alpha_{\rm M, M'}: \ \mathcal{E}({\rm M})\to\mathcal{E}_{M'}({\rm M}')$ which is near-isometric for small $t$. 

The map $\alpha_{\rm M, M'}$ is constructed in the following way. First, define the positive function $\mathscr{E}^{int}_{\rm M'}: \ \mathcal{E}({\rm M})\times\mathcal{E}_{M'}({\rm M}')\to [0,+\infty)$ by 
\begin{equation}
\label{minimizing function int}
\mathscr{E}^{int}_{\rm M'}(\xi,\xi'):=|\xi'-\xi|^2.
\end{equation}
For each $\xi\in\mathcal{E}({\rm M})$, the minimum of $\mathscr{E}^{0}_{\rm M'}(\xi,\cdot)$ is attained at a point $\xi'$ of $\mathcal{E}_{M'}({\rm M}')$ closest to $\xi$ in $\mathbb{C}^n$. Due to the symmetry of the embeddings $\mathcal{E}$ and $\mathcal{E}_{\rm M'}$, we have $\overline{\mathcal{E}({\rm M})}=\mathcal{E}({\rm M})$, $\overline{\mathcal{E}_{\rm M'}({\rm M}')}=\mathcal{E}_{\rm M'}({\rm M}')$. Hence, we have
$$\mathscr{E}^{int}_{\rm M'}(\overline{\xi},\overline{\xi'})=\mathscr{E}^{int}_{\rm M'}(\xi,\xi'), \qquad \xi\in\mathcal{E}({\rm M}), \ \xi'\in\mathcal{E}_{M'}({\rm M}').$$
Next, for $\xi\in\mathcal{E}({\rm M})$, let $L_\xi$ be a shortest geodesics in $\mathcal{E}({\rm M})$ connecting $\xi$ with the boundary $\mathcal{E}(\Upsilon)$ of $\mathcal{E}({\rm M})$, $r=r(\xi)$ be its length, and $\varsigma:=\varsigma(\xi)$ be the point of $\Upsilon$ such that $\mathcal{E}(\tau)$ is the end of the curve $L_\xi$. Then $(\varsigma,r)$ are semi-geodesic coordinates on $\mathcal{E}({\rm M})$, which are regular in some near-boundary strip $r<r_0$. The semi-geodesic coordinates $(\varsigma',r')$ on $\mathcal{E}_{M'}({\rm M}')$ are defined in the similar way. In the following, we show that the numbers $t_0$ and $r_0$ can be chosen so small that the coordinates $(\varsigma',r')$ are regular in the strip $r'<r_0$ for any $(M',g')\in\mathbb{B}(M,t_0)$. Introduce the positive function $\mathscr{E}^{b}_{\rm M'}: \ \mathcal{E}({\rm M})\times\mathcal{E}_{M'}({\rm M}')\to [0,+\infty)$ by 
\begin{equation}
\label{minimizing function near b}
\mathscr{E}^{b}_{\rm M'}(\xi,\xi'):={\rm dist}_{\Upsilon}(\varsigma(\xi),\varsigma'(\xi'))^2+|r'(\xi')-r(\xi)|^2.
\end{equation}
For each $\xi\in\mathcal{E}({\rm M})$ such that $r(\xi)<r_0$, the minimum of $\mathscr{E}^{b}_{\rm M'}(\xi,\cdot)$ is attained at the point $\xi'$ whose semi-geodesics coordinates obey $\varsigma'(\xi')=\varsigma(\xi)$ and $r'(\xi')=r(\xi)$. Due to the symmetry of the embeddings $\mathcal{E}$ and $\mathcal{E}_{\rm M'}$, their images obey $\overline{\mathcal{E}({\rm M})}=\mathcal{E}({\rm M})$, $\overline{\mathcal{E}_{\rm M'}({\rm M}')}=\mathcal{E}_{\rm M'}({\rm M}')$. Hence, we have
\begin{align*}
\varsigma(\overline{\xi})=\tau(\varsigma(\xi)), \quad r(\overline{\xi})=r(\xi)& \qquad (\xi\in\mathcal{E}({\rm M}')),\\
\varsigma'(\overline{\xi})=\tau'(\varsigma'(\xi)), \quad r'(\overline{\xi})=r'(\xi)& \qquad (\xi'\in\mathcal{E}_{M'}({\rm M}')).
\end{align*}
Therefore,
$$\mathscr{E}^{b}_{\rm M'}(\overline{\xi},\overline{\xi'})=\mathscr{E}^{b}_{\rm M'}(\xi,\xi'), \qquad \xi\in\mathcal{E}({\rm M}), \ \xi'\in\mathcal{E}_{M'}({\rm M}').$$
Let $\chi$ be a smooth cutoff function on $\mathcal{E}({\rm M})$ equal to one in the strip $r\le r_0/3$ and to zero in the domain $r\ge 2r_0/3$. We introduce the positive function $\mathscr{E}_{\rm M'}: \ \mathcal{E}({\rm M})\times\mathcal{E}_{M'}({\rm M}')\to [0,+\infty)$ by 
\begin{equation}
\label{minimizing function}
\mathscr{E}_{\rm M'}(\xi,\xi'):=(1-\chi)\mathscr{E}^{int}_{\rm M'}+\chi\mathscr{E}^{b}_{\rm M'}.
\end{equation}
Then $\mathscr{E}_{\rm M'}$ obeys the symmetry condition
\begin{equation}
\label{minimizing function symmetry}
\mathscr{E}_{\rm M'}(\overline{\xi},\overline{\xi'})=\mathscr{E}_{\rm M'}(\xi,\xi'), \qquad \xi\in\mathcal{E}({\rm M}), \ \xi'\in\mathcal{E}_{M'}({\rm M}').
\end{equation}
Along with $\mathscr{E}_{\rm M'}$, we consider the `unperturbed' function $\mathscr{E}_{\rm M}$ obtained by replacing ${\rm M'}$, $\varsigma'$, $r'$ by ${\rm M}$, $\varsigma$, $r$ in (\ref{minimizing function int}), (\ref{minimizing function near b}), (\ref{minimizing function}). Note that, for any $\xi\in\mathcal{E}({\rm M})$, the function $\mathscr{E}_{\rm M}(\cdot,\xi)$ attains minimum at the (unique) point $\xi'=\xi$. For $\xi\in\mathcal{E}({\rm M})$, we define $\alpha_{\rm M, M'}(\xi)$ as the point of minimum of the function $\mathscr{E}_{\rm M'}(\xi,\cdot)$ on $\mathcal{E}_{M'}({\rm M}')$. By the use of (\ref{closenes out of b}) and (\ref{closeness near boundary fin 3}), we prove that $t_0$ can be chosen so small that, for any $(M',g')\in\mathbb{B}(M,t)$ and $\xi\in\mathcal{E}({\rm M})$, the minimum point $\alpha_{\rm M, M'}(\xi)$ is unique and the map $\xi\mapsto\alpha_{\rm M, M'}(\xi)$ is a diffeomorphism from $\mathcal{E}({\rm M})$ onto $\mathcal{E}_{M'}({\rm M}')$. Also, we prove that the map $\alpha_{\rm M, M'}$ is near-isometric for small $t\in (0,t_0)$ and any $(M',g')\in\mathbb{B}(M,t)$. In view of the symmetry (\ref{minimizing function symmetry}) of the function $\mathscr{E}_{\rm M'}$, the map $\alpha_{\rm M, M'}$ satisfies 
\begin{equation}
\label{alpha symmetry}
\alpha_{\rm M, M'}(\overline{\xi})=\overline{\alpha_{\rm M, M'}(\xi)}.
\end{equation}

Finally, we introduce the diffeomorphism
\begin{equation}
\label{beta}
\beta_{\rm M, M'}:=\mathcal{E}_{M'}^{-1}\circ\alpha_{\rm M, M'}\circ\mathcal{E}: \ {\rm M}\to{\rm M}',
\end{equation}
where $\mathcal{E}_{M'}^{-1}: \ \mathcal{E}_{M'}({\rm M'})\to M'$ is the map inverse to $\mathcal{E}_{M'}$. Since the maps $\mathcal{E}: \ {\rm M}\to \mathcal{E}({\rm M})$, $\mathcal{E}_{M'}: {\rm M}'\to \mathcal{E}_{M'}({\rm M}')$ are conformal and the diffeomorphism $\alpha_{\rm M, M'}: \ \mathcal{E}({\rm M})\to \mathcal{E}_{M'}({\rm M}')$ is near-isometric, the diffeomorphism $\beta_{\rm M, M'}$ is near-conformal for small $t\in (0,t_0)$ and any $(M',g')\in\mathbb{B}(M,t)$. Also, since the embeddings $\mathcal{E}$ and $\mathcal{E}_{M'}$ are symmetric, we have $\mathcal{E}\circ\tau=\overline{\mathcal{E}}$ and $\mathcal{E}_{M'}\circ\tau'=\overline{\mathcal{E}_{M'}}$, where $\tau$ and $\tau'$ are the involutions on ${\rm M}$ and ${\rm M}'$, respectively. Due to this, formula (\ref{alpha symmetry}) shows that $\beta_{\rm M, M'}$ is involution-preserving i.e.
\begin{equation}
\label{beta symmetry}
\beta_{\rm M, M'}\circ\tau=\tau'\circ\beta_{\rm M, M'}.
\end{equation}
In view of (\ref{beta symmetry}), one can introduce the diffeomorphism $\beta_{M, M'}: \ M\mapsto M'$ by requiring that 
$$\beta_{M, M'}\circ\pi=\pi'\circ\beta_{\rm M, M'}.$$ 
Since the projections $\pi: \ ({\rm M},{\rm g})\to (M,g)$ and $\pi': \ ({\rm M}',{\rm g}')\to (M',g')$ $\beta_{M, M'}$ are local isometric, the diffeomorphism $\beta_{M, M'}$ is near-conformal for small $t\in (0,t_0)$ and any $(M',g')\in\mathbb{B}(M,t)$. Due to this, $[(M',g')]$ is close to $[(M,g)]$ in Teichm\"uller distance. By this, we prove Theorem \ref{main theorem}.

\subsection{Constructing $\alpha_{\rm M, M'}$ in a zone separated from $\mathcal{E}(\Upsilon)$.} 
\label{ssec alpha separated from boundary}
Let $\xi_0\in\mathcal{E}({\rm M})$ and $\Pi_{\hat{\xi},\mathcal{D}}$ be a cylinder containing $\xi_0$ and possessing properties (1')-(3'). We assume that $\Pi_{\hat{\xi},\mathcal{D}}$ does not intersect near-boundary strip $r>r_0$ of $\mathcal{E}({\rm M})$. Then the function $\mathscr{E}_{\rm M'}$ does coincide with $\mathscr{E}^{int}_{\rm M'}$ on $(\mathcal{E}({\rm M})\cap\Pi_{\hat{\xi},\mathcal{D}})\times(\mathcal{E}_{\rm M'}({\rm M}')\cap\Pi_{\hat{\xi},\mathcal{D}})$. In the rest of the paragraph, we consider the projections $z:=\mathfrak{p}_{\hat{\xi}}\xi$, $z':=\mathfrak{p}_{\hat{\xi}}\xi'$, $z'':=\mathfrak{p}_{\hat{\xi}}\xi''$ as local coordinates of the points $\xi,\xi''\in\mathcal{E}({\rm M})\cap\Pi_{\hat{\xi},\mathcal{D}}$ and $\xi'\in\mathcal{E}_{\rm M'}({\rm M}')\cap\Pi_{\hat{\xi},\mathcal{D}}$. In these coordinates, the functions $\mathscr{E}_{\rm M'}$ and $\mathscr{E}_{\rm M}$ are of the form
\begin{equation*}
\begin{split}
\mathscr{E}^{loc}_{\rm M'}(z,z'):=\mathscr{E}_{\rm M'}(\Xi(z),\Xi'(z))=(\Xi'(z')-\Xi(z))\overline{(\Xi'(z')-\Xi(z))}, \\
\mathscr{E}^{loc}_{\rm M}(z,z''):=\mathscr{E}_{\rm M}(\Xi(z),\Xi(z''))=(\Xi(z'')-\Xi(z))\overline{(\Xi(z'')-\Xi(z))},
\end{split}
\end{equation*}
where the functions $\Xi$ and $\Xi'$ are provided by property (2') and obey estimate (\ref{closenes out of b}).  Due to this estimate, we have 
\begin{equation}
\label{min functions closeness}
\|\mathscr{E}^{loc}_{\rm M'}-\mathscr{E}^{loc}_{\rm M}\|_{C^l(\mathcal{D}\times\mathcal{D};\mathbb{R})}\le c_{\mathcal{E},1}t^{\frac{1}{3}}.
\end{equation}
Note that, for any $z''=z$ is a unique and non-degenerate point of the global minimum of $\mathscr{E}^{loc}_{\rm M}(z,\cdot)$. In what follows, we assume that the diameter of $\mathcal{D}$ is sufficiently small (hence, the surface piece $\mathcal{E}({\rm M})\cap\Pi_{\hat{\xi},\mathcal{D}}$ is approximately plane) that for any $z\in\mathcal{D}$, the function $\mathscr{E}^{loc}_{\rm M}(z,\cdot)$ has no other extremal points except described above.

To construct the map $\alpha_{\rm M, M'}$ on $\mathcal{E}({\rm M})\cap\Pi_{\hat{\xi},\mathcal{D}}$, we need the following lemma.
\begin{lemma}
\label{inverse function lemma}
Let $X_0,X,Y$ be domains with compact closures in a finite-dimen\-sional linear space $\mathbb{E}$, $\overline{X}\subset X_0$, and $E_0\in C^2(X_0\times Y;\mathbb{E})$. Suppose that the zero set of $E_0$ is the graph of the function $\kappa_0\in C^{1}(X_0;Y)$ and the Jacoby matrix $E'_{0,y}(x,\cdot)$ of the function $E_0(x,\cdot)$ obeys ${\rm det}E'_{0,y}(x,\kappa_0(x))\ne 0$ for any $x\in X_0$. Then, for sufficiently small $\varepsilon\in (0,\varepsilon_0)$ and any function $E\in C^1(X_0\times Y;\mathbb{E})$ such that $\|E-E_0\|_{C^1(X_0\times Y;\mathbb{E})}\le\varepsilon$, the zero set of $E$ on $X\times Y$ is the graph of a function $\kappa\in C^{1}(\overline{X};Y)$. Moreover, the estimate $\|\kappa-\kappa_0\|_{C^{1}(\overline{X};Y)}\le c\varepsilon$ holds with the constant independent of $E$ and $\varepsilon$.
\end{lemma}
In other words, Lemma \ref{inverse function lemma} states that the result of the well-known implicit function theorem is stable under small perturbations of the implicit function $\mathscr{E}_0$. It is proved by a slight modification of the proof of the implicit function theorem (for a detailed proof, see Lemma 4 and Appendix B, \cite{BKTeich}). We rewrite Lemma \ref{inverse function lemma} in the following form.
\begin{lemma}
\label{inverse function lemma 1}
Let $X_0,X,Y$ be domains with compact closures in a finite dimensional space $\mathbb{E}$, $\overline{X}\subset X_0$, and $\mathscr{E}_0\in C^3(X_0\times Y;\mathbb{R})$. Suppose that each extremum point of $\mathscr{E}_0(z,\cdot)$ with $z\in X_0$ is non-degenerate and the set of such points is the graph of the function $\kappa_0\in C^{1}(X_0;Y)$. Then, for sufficiently small $\varepsilon\in (0,\varepsilon_0)$ and any function $\mathscr{E}\in C^2(X_0\times Y;\mathbb{R})$ such that $\|\mathscr{E}-\mathscr{E}_0\|_{C^2(X_0\times Y;\mathbb{R})}\le\varepsilon$, the set of extremum points of $\mathscr{E}$ on $X\times Y$ is the graph of a function $\kappa\in C^{1}(\overline{X};Y)$. Moreover, the estimate $\|\kappa-\kappa_0\|_{C^{1}(\overline{X};Y)}\le c\varepsilon$ holds with the constant independent of $\mathscr{E}$ and $\varepsilon$.
\end{lemma}
The proof of Lemma \ref{inverse function lemma 1} is reduced to the application of Lemma \ref{inverse function lemma 1} for $E_0(x,y)=\nabla_y\mathscr{E}_0(x,y)$, $E(x,y)=\nabla_y\mathscr{E}(x,y)$. Let $\tilde{\mathcal{D}}$ be a domain in $\mathbb{C}$ such that $\overline{\tilde{\mathcal{D}}}\subset\mathcal{D}$ and $\mathfrak{p}_{\hat{\xi}}\xi\in\tilde{\mathcal{D}}$. Put $X_0=Y=\mathcal{D}$, $X=\tilde{\mathcal{D}}$, $\mathscr{E}_0=\mathscr{E}^{loc}_{\rm M}$, $\mathscr{E}=\mathscr{E}^{loc}_{\rm M'}$, and $\kappa_0(z)=I$, where $I$ is the identity on $\mathcal{D}$ in Lemma \ref{inverse function lemma 1}. Due to estimate (\ref{min functions closeness}), Lemma \ref{inverse function lemma 1} provides the map $\tilde{\alpha}_{\rm M, M'}: \ \tilde{\mathcal{D}}\to\mathbb{C}$ such that, for each $z\in\tilde{\mathcal{D}}$, the $\tilde{\alpha}_{\rm M, M'}(z)$ is a (unique) extremum point of $\mathscr{E}_{\rm M'}(z,\cdot)$ on $\mathcal{D}$. Also we obtain the estimate
\begin{equation}
\label{inverse function est}
\|\tilde{\alpha}_{\rm M, M'}-I\|_{C^{1}(\tilde{\mathcal{D}};\mathbb{C})}\le c_{\mathcal{E},1}t^{\frac{1}{3}}.
\end{equation}
In particular, for small $t$, the map $\tilde{\alpha}_{\rm M, M'}$ is a diffeomorphism from $\tilde{\mathcal{D}}$ onto its image in $\mathcal{D}$. Define the map $\alpha_{\rm M, M'}: \ \mathcal{E}({\rm M})\cap\Pi_{\hat{\xi},\tilde{\mathcal{D}}}\to \mathcal{E}_{\rm M'}({\rm M}')$ by
$$\alpha_{\rm M, M'}(\xi):=\Xi'\circ\tilde{\alpha}_{\rm M, M'}(\mathfrak{p}_{\hat{\xi}}\xi).$$
In view of (\ref{emb closeness}) and the definition of $\tilde{\alpha}_{\rm M, M'}$, the map $\alpha_{\rm M, M'}$ sends the point $\xi\in\mathcal{E}({\rm M})\cap\Pi_{\hat{\xi},\tilde{\mathcal{D}}}$ to the (unique) point of $\mathcal{E}_{\rm M'}({\rm M}')$ which is closest to $\xi$ in $\mathbb{C}^n$. Also, for small $t$, the map $\alpha_{\rm M, M'}$ is a diffeomorphism from $\mathcal{E}({\rm M})\cap\Pi_{\hat{\xi},\tilde{\mathcal{D}}}$ onto its image in $\mathcal{E}_{\rm M'}({\rm M}')$. By decreasing the diameter of $\mathcal{D}$, one can obtain that the facts above are valid with $\tilde{\mathcal{D}}$ replaced by $\mathcal{D}$. Then, from (\ref{closeness near boundary fin 3}) and (\ref{inverse function est}), we obtain 
\begin{equation}
\label{point closeness 1}
|\alpha_{\rm M, M'}(\xi)-\xi|\le c_{\mathcal{E},1}t^{\frac{1}{3}}, \qquad \xi\in\mathcal{E}({\rm M})\cap\Pi_{\hat{\xi},\mathcal{D}}.
\end{equation}

Let $\phi: \ [0,1) \to \mathcal{E}({\rm M})$ be a curve with the beginning at $\xi\in\mathcal{E}({\rm M})\cap\Pi_{\hat{\xi},\mathcal{D}}$ and the tangent vector $a$ at $\xi$. Denote $z:=\mathfrak{p}_{\hat{\xi}}\xi$, $z':=\mathfrak{p}_{\hat{\xi}}\phi'(0)$ and $\tilde{a}=\frac{d\mathfrak{p}_{\hat{\xi}}\phi}{dt}(0)$, $\tilde{a}':=\frac{d\mathfrak{p}_{\hat{\xi}}\phi'}{dt}(0)=\tilde{\alpha}_{\rm M, M'}(z)\tilde{a}$. Introduce the curves $\phi':=\alpha_{\rm M, M'}\circ\phi$ in $\mathcal{E}_{\rm M'}({\rm M}')$ and denote it tangent vector at $\xi'=\phi'(0)=\alpha_{\rm M, M'}(\xi)$ by $a'$; then $a':=(\alpha_{\rm M, M'})_* a$. Since the components $\Xi_k=\mathfrak{p}_{\hat{\xi}}\Xi=w_{k}\circ w_{\hat{\xi}}^{-1}$ and $\Xi'_k=\mathfrak{p}_{\hat{\xi}}\Xi'=w'_{k}\circ w_{\hat{\xi}}^{'-1}$ of $\Xi$ and $\Xi'$ are holomorphic on $\mathcal{D}$, the lengths (in the metrics induced by $\mathbb{C}^n$ on $\mathcal{E}({\rm M})$ and $\mathcal{E}_{\rm M'}({\rm M}')$) of the tangent vectors $a$, $a'$ are given by
\begin{align*}
|a|_{T_\xi\mathcal{E}({\rm M})}^2&=\sum_{k=1}^{n}|\partial_z\Xi_k(z)|^2|\tilde{a}|^2, \\
|a'|_{T_{\xi'}\mathcal{E}_{\rm M'}({\rm M}')}^{2}&=\sum_{k=1}^{n}|\partial_z\Xi'_k(z')|^2|\tilde{a}'|^2=\\
&=\sum_{k=1}^{n}|(\partial_z\Xi'_k)\circ\tilde{\alpha}_{\rm M, M'}(z)\cdot\partial_z\tilde{\alpha}_{\rm M, M'}(z)|^2|\tilde{a}|^2,
\end{align*}
respectively. Hence, estimates (\ref{inverse function est}) and (\ref{closenes out of b}) (see condition (3'')) yield
\begin{equation}
\label{near isometric 1}
\Big|\frac{|(\alpha_{\rm M, M'})_* a|_{T_{\alpha_{\rm M, M'}(\xi)}\mathcal{E}_{\rm M'}({\rm M}')}}{|a|_{T_\xi\mathcal{E}({\rm M})}}-1\Big|\le c_{\mathcal{E},1}t^{\frac{1}{3}}, \quad \forall a\in T_\xi\mathcal{E}({\rm M}), \ \xi\in\mathcal{E}({\rm M})\cap\Pi_{\hat{\xi},\mathcal{D}}
\end{equation}
for sufficiently small $t\in(0,t_0)$ and any $(M',g')\in\mathbb{B}(M,t)$. So, the map $\alpha_{\rm M, M'}: \ \mathcal{E}({\rm M})\cap\Pi_{\hat{\xi},\mathcal{D}}\to\mathcal{E}_{\rm M'}({\rm M}')$ is near-isometric for small $t$ and any $(M',g')\in\mathbb{B}(M,t)$.

So, we have proved that any $\xi_0\in\mathcal{E}({\rm M})$ such that $r(\xi)>r_0$ (where $r_0>0$ is an arbitrarily fixed number) is contained in some cylinders $\Pi_{\hat{\xi},\mathcal{D}}$, $\Pi_{\hat{\xi}_k,\mathcal{D}_k}$ such that, for sufficiently small $t\in(0,t(\xi_0))$ and any $(M',g')\in\mathbb{B}(M,t)$, the following properties are valid are valid along with (1')-(3'):
\begin{enumerate}[(1')]
\setcounter{enumi}{3}
\item $\overline{\Pi_{\hat{\xi},\tilde{\mathcal{D}}}}\subset\Pi_{\hat{\xi}_k,\mathcal{D}_k}$ and $\Pi_{\hat{\xi}_k,\mathcal{D}_k}$ does not intersect the near-boundary strip $r(\xi)<r_0$ in $\mathcal{E}({\rm M})$.
\item there is the diffeomorphism $\alpha^{(k)}_{\rm M, M'}$ from $\mathcal{E}({\rm M})\cap\Pi_{\hat{\xi}_k,\mathcal{D}_k}$ onto its image in $\mathcal{E}_{\rm M'}({\rm M}')$ obeying (\ref{point closeness 1}), (\ref{near isometric 1}).
\item for $\xi\in \mathcal{E}({\rm M})\cap\Pi_{\hat{\xi}_k,\tilde{\mathcal{D}}_k}$, the $\alpha^{(k)}_{\rm M, M'}(\xi)$ is the unique extremum point of the function $\mathscr{E}_{\rm M'}(\xi,\cdot)$ on $\mathcal{E}_{\rm M'}({\rm M}')\cap\Pi_{\hat{\xi}_k,\mathcal{D}_k}$.
\end{enumerate}
Indeed, due to the result of Subsection \ref{1st type est}, any $\xi_0\in\mathcal{E}({\rm M})$ such that $r(\xi)>r_0$ is contained in some cylinder $\Pi_{\hat{\xi}_k,\mathcal{D}_k}$. Then, as explained above, by decreasing the diameter of the base $\mathcal{D}_k$ of $\Pi_{\hat{\xi}_k,\mathcal{D}_k}$, one can obey (5'). In addition to this, one satisfies the inequality $r(\xi)>r_0$ on $\mathcal{E}({\rm M})\cap \Pi_{\hat{\xi}_k,\mathcal{D}_k}$. Now, choosing any sub-cylinder $\Pi_{\hat{\xi},\tilde{\mathcal{D}}}$ obeying (4') and repeating the reasoning above, one obtain (6').

\subsection{Constructing $\alpha_{\rm M, M'}$ in a zone near $\mathcal{E}(\Upsilon)$.} 
\label{ssec alpha near boundary}
Suppose that the point $\xi_0\in\mathcal{E}(\Upsilon)$ and the cylinders $\Pi_{\hat{\xi},\mathcal{D}}$, $\Pi_{\hat{\xi},\tilde{\mathcal{D}}}$, the set $\varpi=\Upsilon_{\mathcal{D}}\times [0,\rho_0)$, and the map $\Xi$ are such that $\xi_0\in\Pi_{\hat{\xi},\tilde{\mathcal{D}}}$ and properties (1'')-(3'') are valid. In what follows, we consider $(s(\xi),\rho(\xi)):=\Xi^{-1}(\xi)$ and $(s'(\xi'),\rho'(\xi)):=\Xi^{'-1}(\xi')$ (where $\Xi'$ is provided by (2'')) as local coordinates of points $\xi\in\mathcal{E}({\rm M})\cap\Pi_{\hat{\xi},\mathcal{D}}$ and $\xi'\in\Xi'(\varphi)\supset\mathcal{E}_{\rm M'}({\rm M}')\cap\Pi_{\hat{\xi},\tilde{\mathcal{D}}}$, respectively, where $\Xi'$ is the map provided by property (2''). Foe simplicity, we also identify the points $s,s'$ of the segment $\Upsilon_{\mathcal{D}}$ with their natural parameters. With this agreement, one can consider the set $\varpi$ as a sub-domain of the closed upper half-plane in $\mathbb{R}^2$.

Let $\phi^0=(\phi^0_s,\phi^0_\rho): \ [0,1)\to\varpi$ be a curve with the beginning at $(s_0,\rho_0)$. Introduce the curves $\phi:=\Xi\circ\phi^0$ and $\phi':=\Xi'\circ\phi^0$ in $\mathcal{E}({\rm M})$ and $\mathcal{E}_{\rm M'}({\rm M}')$, respectively. Let $a$ be the tangent vector to the curve $\phi$ at the point $\xi_0:=\phi(0)=\Xi(s_0,\rho_0)$ and $a'$ be the tangent vector to $\phi'$ at $\xi'_0:=\phi'(0)=\Xi'(s_0,\rho_0)$. In the metrics induced by $\mathbb{C}^n$ on $\mathcal{E}({\rm M})$ and $\mathcal{E}_{\rm M'}({\rm M}')$, the lengths of $a$ and $a'$ are given by
\begin{equation}
\label{metrics near boundary 1}
\begin{split}
|a|_{T_\xi\mathcal{E}({\rm M})}^2&=\sum_{k=1}^{n}\big|\partial_s\Xi_k(\phi^0(\sigma))\frac{d\phi^0_s(\sigma)}{d\sigma}+\partial_\rho\Xi_k(\phi^0(\sigma))\frac{d\phi^0_\rho(\sigma)}{d\sigma}\big|^2\Big|_{\sigma=0}, \\
|a'|_{T_{\xi'}\mathcal{E}_{\rm M'}({\rm M}')}^{2}&=\sum_{k=1}^{n}\big|\partial_s\Xi'_k(\phi^0(\sigma))\frac{d\phi^0_s(\sigma)}{d\sigma}+\partial_\rho\Xi'_k(\phi^0(\sigma))\frac{d\phi^0_\rho(\sigma)}{d\sigma}\big|^2\Big|_{\sigma=0},
\end{split}
\end{equation}
respectively. Let $\tilde{g}$ be the matrix of the metric tensor on $\mathcal{E}({\rm M})$ in local coordinates $(s,\rho)$ and $\tilde{g}'$ be the matrix of the metric tensor on $\mathcal{E}_{\rm M'}({\rm M}')$ in local coordinates $(s',\rho')$. Then formula (\ref{metrics near boundary 1}) and the estimates (\ref{closeness near boundary fin 3}) provided by (2'') imply
\begin{equation}
\label{metrics near boundary 1 est}
\|\tilde{g}'-\tilde{g}\|_{C^{l}(\varpi;M^{2\times 2})}\le c_{\mathcal{E},l+4}t^{\frac{1}{3}}.
\end{equation}
Decreasing the diameter of the base $\mathcal{D}$ of $\Pi_{\hat{\xi},\mathcal{D}}$ if needed, one can assume that the semi-geodesic coordinates $(\varsigma,r)$ are regular on $\mathcal{E}({\rm M})\cap\Pi_{\hat{\xi},\mathcal{D}}$. By halving the diameter of $\mathcal{D}$, one can obtain that the semi-geodesic coordinates $(\varsigma',r')$ are regular on $\mathcal{E}_{\rm M'}({\rm M}')\cap\Pi_{\hat{\xi},\mathcal{D}}$ for sufficiently small $t\in(0,t_0)$ and any $(M',g')\in\mathbb{B}(M,t)$. This fact is a corollary of (\ref{metrics near boundary 1 est}) and the following lemma.
\begin{lemma}
\label{geodesic stability lemma}
Let $\mathfrak{g}$ be a smooth metric tensor defined on the upper half-plane $\mathbb{R}^2_+$ of $\mathbb{R}^2$ and $\mathfrak{x}:=(\mu,\varrho)$ be semi-geodesic coordinates on $\mathbb{R}^2_+$ corresponding to $\mathfrak{g}$ which are regular in the rectangle ${\rm R}_0:=[-2q_0,2q_0]\times[0,2q_0]$, where $q_0>0$. Then there exist $\epsilon_0>0$, such that, for any metrics $\mathfrak{g}'$ on $\mathbb{R}^2_+$ obeying $\|\mathfrak{g}'-\mathfrak{g}\|_{C^{3}({\rm R}_0;M^{2\times 2})}\le \epsilon$, the semi-geodesic coordinates $\mathfrak{x}':=(\mu',\varrho')$ on $\mathbb{R}^2_+$ corresponding to $\mathfrak{g}'$ are regular on the rectangle ${\rm R}:=[-q_0,q_0]\times[0,q_0]$ and obey $\|\mathfrak{x}'-\mathfrak{x}\|_{C^2({\rm R};\mathbb{R}^2)}\le c\epsilon$, where $c$ does not depend on $\mathfrak{g}'$.
\end{lemma}
In other words, Lemma \ref{geodesic stability lemma} states that solutions to the Cauchy problem for the geodesic equation are stable under small perturbations of the metric. It is proved by slight modifications of arguments used in the proof of local solvability of the Cauchy problem (for a detailed proof, see Lemma 6 and Appendix C, \cite{BKTeich}). 

For $(s,\rho)\in\varpi$, denote 
$$\mathfrak{x}(s,\rho):=(\varsigma(\Xi(s,\rho)),r(\Xi(s,\rho))), \qquad \mathfrak{x}'(s,\rho):=(\varsigma'(\Xi'(s,\rho)),r'(\Xi'(s,\rho))).$$ 
Then (\ref{metrics near boundary 1 est}) and Lemma \ref{geodesic stability lemma} lead to the estimates 
\begin{equation}
\label{closenes of geodesic coordinates}
\begin{split}
\|\mathfrak{x}'-\mathfrak{x}\|_{C^2(\varpi;\mathbb{R}^2)}\le c_{\mathcal{E},7}t^{\frac{1}{3}},\\
\|\mathfrak{x}^{'-1}\circ\mathfrak{x}-I\|_{C^2(\varpi;\mathbb{R}^2)}\le c_{\mathcal{E},7}t^{\frac{1}{3}},
\end{split}
\end{equation}
where $I$ is the identity on $\varpi$. Let $\varpi_{0}$ be the set of all $(s,\rho)\in\varpi$ such that $r(\Xi(s,\rho))<r_0/3$. Since $\chi(r(\Xi(s,\rho)))=1$ for $(s,\rho)\in\varpi_0$, formulas (\ref{minimizing function}), (\ref{minimizing function near b}) mean that the global minimum of the function $\mathscr{E}^{loc}_{\rm M'}(\xi,\cdot)$, $\xi=\Xi(s,\rho)$ is attained at the (unique) point $\alpha_{\rm M, M'}(\xi)=\Xi'(s',\rho')$ such that $\mathfrak{x}'(s',\rho')=\mathfrak{x}(s,\rho)$. By this, the map $\alpha_{\rm M, M'}$ is defined on $\Xi(\varpi_0)$. In addition, the map $\tilde{\alpha}_{\rm M, M'}:=\Xi^{-1}\circ\alpha_{\rm M, M'}\circ\Xi=\mathfrak{x}^{'-1}\circ\mathfrak{x}$ is well-defined on $\varpi_0$ and obeys
\begin{equation}
\label{map alpha closenes in loc c 1}
\|\tilde{\alpha}_{\rm M, M'}-I\|_{C^2(\varpi_0;\mathbb{R}^2)}\le c_{\mathcal{E},7}t^{\frac{1}{3}}
\end{equation}
due to the second estimate of (\ref{closenes of geodesic coordinates}). 

Now, let $\varpi_{1}$ be the set of all $(s,\rho)\in\varpi$ such that $r(\Xi(s,\rho))>r_0/5$. Similarly, let $\varpi_{2}$ be the set of all $(s,\rho)\in\varpi$ such that $r(\Xi(s,\rho))>r_0/4$. For $(s,\rho)\in\varpi$  and $(s',\rho')\in\varpi$, put
\begin{align*}
\mathscr{E}^{loc}_{\rm M'}(s,\rho,s',\rho'):=\mathscr{E}_{\rm M'}(\Xi(s,\rho);\Xi'(s',\rho')),\\
\mathscr{E}^{loc}_{\rm M}(s,\rho,s',\rho'):=\mathscr{E}_{\rm M}(\Xi(s,\rho);\Xi(s',\rho')).
\end{align*}
Then estimates (\ref{closeness near boundary fin 3}) and (\ref{closenes of geodesic coordinates}) yield
\begin{equation}
\label{closeness of min func near b}
\|\mathscr{E}^{loc}_{\rm M'}-\mathscr{E}^{loc}_{\rm M}\|_{C^2(\varpi;\mathbb{R})}\le c_{\mathcal{E},7}t^{\frac{1}{3}}. 
\end{equation}
Note that $(s',\rho')=(s,\rho)$ is a point of global minimum of $\mathscr{E}^{loc}_{\rm M}(s,\rho,\cdot,\cdot)$. By decreasing the diameter of the base $\mathcal{D}$ of $\Pi_{\hat{\xi},\mathcal{D}}$, we obtain that, for any $(s,\rho)\in\varpi_{1}$, the function  $\mathscr{E}^{loc}_{\rm M}(s,\rho,\cdot,\cdot)$ has no other extremal points except described above. Due to this, the set of the extremal points of $\mathscr{E}^{loc}_{\rm M}(s,\rho,\cdot,\cdot)$ with $(s,\rho)\in\varpi_{1}$ coincides with the graph of $I: \ (s,\rho)\to(c,\rho)$. Due to (\ref{closeness of min func near b}), the application of Lemma \ref{inverse function lemma 1} (with $X_0=\varpi_{1}$, $X=\varpi_{2}$, $Y=\varpi$, $\mathscr{E}_0=\mathscr{E}_{\rm M}$, $\mathscr{E}_0=\mathscr{E}^{loc}_{\rm M}$, $\mathscr{E}=\mathscr{E}^{loc}_{\rm M'}$, and $\kappa_0(z)=I$) provides the injective map $\tilde{\alpha}_{\rm M, M'}: \ \varpi_{2}\to\varpi$ which sends $(s,\rho)\in\varpi_{2}$ to the (unique) extremum point of $\mathscr{E}^{loc}_{\rm M'}(s,\rho,\cdot,\cdot)$. Also, the estimate
\begin{equation}
\label{map alpha closenes in loc c 2}
\|\tilde{\alpha}_{\rm M, M'}-I\|_{C^1(\varpi_2;\mathbb{R}^2)}\le c_{\mathcal{E},7}t^{\frac{1}{3}}.
\end{equation}
is valid. Note that, for $(s,\rho)\in\varpi_{0}\cap\varpi_{2}$, the (unique) extremum point of $\mathscr{E}^{loc}_{\rm M'}(s,\rho,\cdot,\cdot)$ is $(s',\rho')=(s,\rho)$. Due to this, the map $\tilde{\alpha}_{\rm M, M'}$ the definition of $\tilde{\alpha}_{\rm M, M'}$ is consistent with the one introduced before (\ref{map alpha closenes in loc c 1}). Thus, the map $\tilde{\alpha}_{\rm M, M'}$ is well-defined on $\varpi$. Also, from (\ref{map alpha closenes in loc c 1}) and (\ref{map alpha closenes in loc c 2}), it follows that
\begin{equation}
\label{map alpha closenes in loc c}
\|\tilde{\alpha}_{\rm M, M'}-I\|_{C^1(\varpi;\mathbb{R}^2)}\le c_{\mathcal{E},7}t^{\frac{1}{3}}.
\end{equation}
By construction, the map
$$\alpha_{\rm M, M'}:=\Xi'\circ\tilde{\alpha}_{\rm M, M'}\circ\Xi^{-1}$$ 
is a diffeomorphism from $\mathcal{E}({\rm M})\cap\Pi_{\hat{\xi},\mathcal{D}}$ onto its image in $\mathcal{E}_{\rm M'}({\rm M'})$ which sends the points close to the boundary $\mathcal{E}(\Upsilon)$ to the points of $\mathcal{E}_{\rm M'}({\rm M'})$ with the same values of semi-geodesic coordinates. Also, estimates (\ref{closeness near boundary fin 3}), (\ref{metrics near boundary 1 est}), and (\ref{map alpha closenes in loc c}) imply
\begin{equation}
\label{points closeness 2}
|\alpha_{\rm M, M'}(\xi)-\xi|\le c_{\mathcal{E},7}t^{\frac{1}{3}}, \quad \forall \xi\in\mathcal{E}({\rm M})\cap\Pi_{\hat{\xi},\mathcal{D}}
\end{equation}
and
\begin{equation}
\label{near isometric 2}
\Big|\frac{|(\alpha_{\rm M, M'})_* a|_{T_{\alpha_{\rm M, M'}(\xi)}\mathcal{E}_{\rm M'}({\rm M}')}}{|a|_{T_\xi\mathcal{E}({\rm M})}}-1\Big|\le c_{\mathcal{E},7}t^{\frac{1}{3}}, \quad \forall a\in T_\xi\mathcal{E}({\rm M}), \ \xi\in\mathcal{E}({\rm M})\cap\Pi_{\hat{\xi},\mathcal{D}}
\end{equation}
for sufficiently small $t\in(0,t_0)$ and any $(M',g')\in\mathbb{B}(M,t)$. So, the map $\alpha_{\rm M, M'}: \ \mathcal{E}({\rm M})\cap\Pi_{\hat{\xi},\mathcal{D}}\to\mathcal{E}_{\rm M'}({\rm M}')$ is near-isometric for small $t$ and any $(M',g')\in\mathbb{B}(M,t)$. 

So, we have proved that any $\xi_0\in\mathcal{E}(\Upsilon)$ is contained in some cylinders $\Pi_{\hat{\xi},\tilde{\mathcal{D}}}$, $\Pi_{\hat{\xi},\mathcal{D}}$ such that, for arbitrarily fixed $r_0>0$ in (\ref{minimizing function}), sufficiently small $t\in(0,t(\xi_0))$, and any $(M',g')\in\mathbb{B}(M,t)$, the following properties are valid are valid along with (1'')-(3''):
\begin{enumerate}[(1'')]
\setcounter{enumi}{3}
\item $\overline{\Pi_{\hat{\xi},\tilde{\mathcal{D}}}}\subset\Pi_{\hat{\xi}_k,\mathcal{D}_k}$ and the cylinder $\Pi_{\hat{\xi},\tilde{\mathcal{D}}}$ contains the part $\{\xi\in\mathcal{E}({\rm M}) \ | \ r(\xi)\le r_1, \ \varsigma(\xi)\in\Upsilon_1\}$ of the surface $\mathcal{E}({\rm M})$, where $r_1>0$ and $\Upsilon_1$ is a segment of $\Upsilon$ containing $\xi_0$.
\item there is the diffeomorphism $\alpha_{\rm M, M'}$ from $\mathcal{E}({\rm M})\cap\Pi_{\hat{\xi}_k,\mathcal{D}}$ onto its image in $\mathcal{E}_{\rm M'}({\rm M}')$ obeying (\ref{points closeness 2}) and (\ref{near isometric 2}).
\item for $\xi\in\Pi_{\hat{\xi},\tilde{\mathcal{D}}}$, the $\alpha_{\rm M, M'}(\xi)$ is the unique extremum point of the function $\mathscr{E}_{\rm M'}(\xi,\cdot)$ on $\mathcal{E}_{\rm M'}({\rm M}')\cap\Pi_{\hat{\xi},\mathcal{D}}$. In particular, if $ r(\xi)\le r_0/3$, then the semi-geodeisic coordinates of $\alpha_{\rm M, M'}(\xi)$ and $\xi$ on $\mathcal{E}_{\rm M'}({\rm M}')$ and $\mathcal{E}({\rm M})$, respectively, do coincide.
\end{enumerate}
Indeed, any $\xi_0\in\mathcal{E}(\Upsilon)$ is contained in some cylinders $\Pi_{\hat{\xi},\tilde{\mathcal{D}}}$, $\Pi_{\hat{\xi},\mathcal{D}}$ obeying (1'')-(3''), while properties (5''), (6'') can be satisfied by decreasing of the diameters of $\Pi_{\hat{\xi},\tilde{\mathcal{D}}}$, $\Pi_{\hat{\xi},\mathcal{D}}$. Also, property (4'') is valid since the base $\tilde{\mathcal{D}}$ of $\Pi_{\hat{\xi},\tilde{\mathcal{D}}}$ has non-zero diameter.

\subsection{Global construction of $\alpha_{\rm M, M'}$ and completing the proof of Theorem \ref{main theorem}.} 
Let $\xi_0\in\mathcal{E}(\Upsilon)$. As showed in Subsection \ref{ssec alpha near boundary}, there are the cylinders $\Pi_{\hat{\xi},\mathcal{D}}=\Pi_{\hat{\xi}(\xi_0),\mathcal{D}(\xi_0)}$, $\Pi_{\hat{\xi},\tilde{\mathcal{D}}}=\Pi_{\hat{\xi}(\xi_0),\tilde{\mathcal{D}}(\xi_0)}$ which contain $\xi_0$ and satisfy (1'')-(6''). The cylinders $\{\Pi_{\hat{\xi}(\xi_0),\tilde{\mathcal{D}}(\xi_0)}\}_{\xi_0\in\mathcal{E}(\Upsilon)}$ constitute an open cover of $\mathcal{E}(\Upsilon)$. Since $\mathcal{E}(\Upsilon)$ is a compact subset of $\mathbb{C}^n$, one can chose a finite sub-cover $\{\Pi_{\hat{\xi}(\xi_0(k)),\tilde{\mathcal{D}}(\xi_0(k))}\}_{k=1}^{K_1}$ of $\mathcal{E}(\Upsilon)$. Due to property (4''), there is $r_0>0$ such that the near-boundary strip $r\le 2r_0$ of $\mathcal{E}({\rm M})$ is contained in the union $\bigcup_{k=1}^{K_1}\Pi_{\hat{\xi}(\xi_0(k)),\tilde{\mathcal{D}}(\xi_0(k))}$. In what follows, we choose such $r_0$ in definition (\ref{minimizing function}) of $\mathscr{E}_{\rm M'}$.

Now, let $\xi_0\in \mathcal{E}({\rm M})$ be such that $r(\xi_0)\ge r_0$. As showed in Subsection \ref{ssec alpha separated from boundary}, there are the cylinders $\Pi_{\hat{\xi},\mathcal{D}}=\Pi_{\hat{\xi}(\xi_0),\mathcal{D}(\xi_0)}$, $\Pi_{\hat{\xi},\tilde{\mathcal{D}}}=\Pi_{\hat{\xi}(\xi_0),\tilde{\mathcal{D}}(\xi_0)}$ which contain $\xi_0$ and obey (1')-(6'). The cylinders $\{\Pi_{\hat{\xi}(\xi_0),\tilde{\mathcal{D}}(\xi_0)}\}_{\xi_0\in\mathcal{E}({\rm M}), \ r(\xi_0)> r_0}$ constitute an open cover of the compact subset $\mathcal{E}({\rm M})_-:=\{\xi_0\in \mathcal{E}({\rm M}) \ | \  r(\xi_0)\ge r_0\}$ of $\mathbb{C}^n$. Therefore, there is a finite sub-cover $\{\Pi_{\hat{\xi}(\xi_0(k)),\tilde{\mathcal{D}}(\xi_0(k))}\}_{k=K_1+1}^{K_2}$ of $\mathcal{E}({\rm M})_-$. Then $\{\Pi_{\hat{\xi}(\xi_0(k)),\tilde{\mathcal{D}}(\xi_0(k))}\}_{k=1}^{K_2}$ is a finite cover of $\mathcal{E}({\rm M})$. Since the number of cylinders of the cover is finite, one can assume that there is $\epsilon_0>0$ such that each cylinder $\Pi_{\hat{\xi}(\xi_0(k)),\mathcal{D}(\xi_0(k))}$ contain $\epsilon_0-$neighbourhood of $\Pi_{\hat{\xi}(\xi_0(k)),\tilde{\mathcal{D}}(\xi_0(k))}$. Also, we assume that that constant $c$ in estimates (\ref{points closeness 2}), (\ref{near isometric 1}), (\ref{point closeness 1}) and (\ref{near isometric 2}) provided by (5'), (5'') do not depend on the number $k$ of cylinder. Similarly, we assume that, for each $k=1,\dots K_2$, properties (1')-(6') and (1'')-(6'') are valid for all $t\in (0,t_0)$ while $c_{\mathcal{E},7}t_0^{\frac{1}{3}}\le \epsilon_0/3$. 

For each cylinder $\Pi_{\hat{\xi}(\xi_0(k)),\mathcal{D}(\xi_0(k))}$, denote the map $\alpha_{\rm M, M'}$ provided by property (5'), (6') or (5''), (6'') by $\alpha^{(k)}_{\rm M, M'}(\xi)$. Let $\xi\in\mathcal{E}({\rm M})$ belongs to the intersection of the cylinders $\Pi_{\hat{\xi}(\xi_0(k)),\tilde{\mathcal{D}}(\xi_0(k))}$ and $\Pi_{\hat{\xi}(\xi_0(j)),\tilde{\mathcal{D}}(\xi_0(j))}$. In view of estimates (\ref{points closeness 2}), (\ref{point closeness 1}) provided by (5'), (5'') and the inequality $c_{\mathcal{E},7}t_0^{\frac{1}{3}}\le \epsilon_0/3$, the points $\alpha^{(k)}_{\rm M, M'}(\xi)$ and $\alpha^{(j)}_{\rm M, M'}(\xi)$ belong to the intersection of the cylinders $\Pi_{\hat{\xi}(\xi_0(k)),\mathcal{D}(\xi_0(k))}$ and $\Pi_{\hat{\xi}(\xi_0(j)),\mathcal{D}(\xi_0(j))}$. Then $\alpha^{(k)}_{\rm M, M'}(\xi)=\alpha^{(k')}_{\rm M, M'}(\xi)$ in view of (6') and (6''). So, the map $\alpha_{\rm M, M'}$ introduced by the rule 
$$\alpha_{\rm M, M'}|_{\mathcal{E}({\rm M})\cap\Pi_{\hat{\xi}_k,\mathcal{D}_k}}=\alpha^{(k)}_{\rm M, M'}$$
is well-defined on the whole $\mathcal{E}({\rm M})$ and it is a local diffeomorphism due to (5') and (5'').

Now, let $\xi_1\in\mathcal{E}({\rm M})\cap\Pi_{\hat{\xi}(\xi_0(k)),\tilde{\mathcal{D}}(\xi_0(k))}$, $\xi_2\in\mathcal{E}({\rm M})$, and $\alpha_{\rm M, M'}(\xi_1)=\alpha_{\rm M, M'}(\xi_2)=\xi'$. In view of estimates (\ref{points closeness 2}), (\ref{point closeness 1}) provided by (5'), (5''), we have $|\xi_1-\xi_2|\le |\xi_1-\xi'|+|\xi_2-\xi'|\le 2c_{\mathcal{E},7}t_0^{\frac{1}{3}}<\epsilon_0$. Thus, $\xi_1,\xi_2\in\Pi_{\hat{\xi}(\xi_0(k)),\mathcal{D}(\xi_0(k))}$ and $\xi'=\alpha^{(k)}_{\rm M, M'}(\xi_1)=\alpha^{(k)}_{\rm M, M'}(\xi)(\xi_2)$. Since $\alpha^{(k)}_{\rm M, M'}: \ $ is an injection, we have $\xi_1=\xi_2$. Therefore, the map $\alpha_{\rm M, M'}: \ \mathcal{E}({\rm M})\to \mathcal{E}_{\rm M'}({\rm M}')$ is an injection.

Since the map $\alpha_{\rm M, M'}$ is a local diffeomorphism, its differential non-degenerate of $\mathcal{E}({\rm M})$. Then its image $\alpha_{\rm M, M'}(\mathcal{E}({\rm M}))$ is open in $\mathcal{E}_{\rm M'}({\rm M}')$. On the contrary, since $\alpha_{\rm M, M'}$ is continuous and the surface $\mathcal{E}({\rm M})$ is compact, the image $\alpha_{\rm M, M'}(\mathcal{E}({\rm M}))$ is a compact subset of $\mathcal{E}_{\rm M'}({\rm M}')$ and, thus, it is closed in $\mathcal{E}_{\rm M'}({\rm M}')$. Since $\mathcal{E}_{\rm M'}({\rm M}')$ is connected, we obtain $\alpha_{\rm M, M'}(\mathcal{E}({\rm M}))=\mathcal{E}_{\rm M'}({\rm M}')$. Therefore, the map $\alpha_{\rm M, M'}: \ \mathcal{E}({\rm M})\to \mathcal{E}_{\rm M'}({\rm M}')$ is a surjection.

So, we have proved that $\alpha_{\rm M, M'}$ from $\mathcal{E}({\rm M})$ onto $\mathcal{E}_{\rm M'}({\rm M}')$. In view of properties (6'), (6''), and the symmetry (\ref{minimizing function symmetry}) of the function $\mathscr{E}_{\rm M'}$, the map $\alpha_{\rm M, M'}$ obeys symmetry condition (\ref{alpha symmetry}). Also, in view of property (6''), we have
$$(\varsigma'(\alpha_{\rm M, M'}(\xi)),r'(\alpha_{\rm M, M'}(\xi)))=(\varsigma(\xi),r(\xi)), \qquad \xi\in\mathcal{E}({\rm M}), \ r(\xi)\le r_0/3.$$
In particular, 
\begin{equation}
\label{boungary to boundary 1}
\alpha_{\rm M, M'}\circ\mathcal{E}(x)=\mathcal{E}_{\rm M'}(x), \qquad x\in\Upsilon.
\end{equation}
Finally, in view of estimates (\ref{near isometric 1}) and (\ref{near isometric 2}) provided by (5') and (5''), we have 
\begin{equation}
\label{near isometric map}
\Big|\frac{|(\alpha_{\rm M, M'})_* a|_{T_{\alpha_{\rm M, M'}(\xi)}\mathcal{E}_{\rm M'}({\rm M}')}}{|a|_{T_\xi\mathcal{E}({\rm M})}}-1\Big|\le c_{\mathcal{E},7}t^{\frac{1}{3}}, \quad \forall a\in T_\xi\mathcal{E}({\rm M}), \ \xi\in\mathcal{E}({\rm M}).
\end{equation}
Thus, the map $\alpha_{\rm M, M'}$ is near-isometric for small $t>0$ and any $(M',g')\in\mathbb{B}(M,t)$.

Now, introduce the diffeomorphism $\beta_{\rm M, M'}: \ {\rm M}\to {\rm M}'$ by (\ref{beta}). In view of symmetry condition (\ref{alpha symmetry}), the map $\beta_{\rm M, M'}$ is involution-preserving i.e. obeys (\ref{beta symmetry}). Formula (\ref{boungary to boundary 1}) means that $\beta_{\rm M, M'}$ does not move the points of the boundary $\Upsilon$ i.e. $\beta_{\rm M, M'}(x)=x$ for any $x\in\Upsilon$. Next, let $x$ be an arbitrary point of $M$, $b_1,b_2\in T_x M$ be arbitrary tangent vectors at $x$, and $|b_1|_{T_x {\rm M}}=|b_2|_{T_x {\rm M}}=1$. Denote $a_j:=\mathcal{E}_*b_j$, $a'_j:=(\alpha_{\rm M, M'})_* a_j$, and $b'_j:=(\mathcal{E}_{\rm M'}^{-1})_* a'_j$. Since the maps $\mathcal{E}: \ {\rm M}\to\mathcal{E}({\rm M})$ and $\mathcal{E}_{\rm M'}: \ {\rm M}'\to\mathcal{E}_{\rm M'}({\rm M}')$ are conformal, we have 
$$\frac{|a_1|_{T_\xi\mathcal{E}({\rm M})}}{|b_1|_{T_x {\rm M}}}=\frac{|a_2|_{T_\xi\mathcal{E}({\rm M})}}{|b_2|_{T_x {\rm M}}}, \qquad \frac{|a'_1|_{T_{\xi'}\mathcal{E}_{\rm M'}({\rm M}')}}{|b'_1|_{T_{\beta_{\rm M, M'}(x)}{\rm M}'}}=\frac{|a'_2|_{T_{\xi'}\mathcal{E}_{\rm M'}({\rm M}')}}{|b'_2|_{T_{\beta_{\rm M, M'}(x)}{\rm M}'}}.$$
Hence,
$$\Big|\frac{|b'_1|_{T_{\beta_{\rm M, M'}(x)}{\rm M}'}}{|b'_2|_{T_{\beta_{\rm M, M'}(x)}{\rm M}'}}-1\Big|=\Big|\frac{|a'_1|_{T_{\alpha_{\rm M, M'}(\xi)}\mathcal{E}_{\rm M'}({\rm M}')}}{|a_1|_{T_\xi\mathcal{E}({\rm M})}}\cdot\frac{|a_2|_{T_\xi\mathcal{E}({\rm M})}}{|a'_2|_{T_{\alpha_{\rm M, M'}(\xi)}\mathcal{E}_{\rm M'}({\rm M}')}}-1\Big|\le c_{\mathcal{E},7}t^{\frac{1}{3}}$$
due to estimate (\ref{near isometric map}). The last inequality means that the dilatation $K_{\beta_{\rm M, M'}}(x)$ of the map $\beta_{\rm M, M'}$ satisfies
\begin{equation}
\label{dilatation est}
{\rm log}K_{\beta_{\rm M, M'}}(x)\le c_{\mathcal{E},7}t^{\frac{1}{3}}, \qquad \forall x\in {\rm M}.
\end{equation}
Thus, the diffeomorphism $\beta_{\rm M, M'}$ is near-conformal for small $t>0$ and any $(M',g')\in\mathbb{B}(M,t)$.

Finally, introduce the diffeomorphism $\beta_{M,M'}: \ M\to M'$ in the following way. For $y\in M$, its pre-image $\pi^{-1}(\{x\})$ in ${\rm M}$ consists of two points $x$ and $\tau(x)$. Due to (\ref{beta symmetry}) the points $\beta_{\rm M, M'}(x)$ and $\beta_{\rm M, M'}(\tau(x))$ of ${\rm M'}$ are connected by $\beta_{\rm M, M'}(\tau(x))=\tau'(\beta_{\rm M, M'}(x))$ and, thus, have the joint projection $\beta_{M,M'}(y):=\pi'(\beta_{\rm M, M'}(x))=\pi'(\beta_{\rm M, M'}(\tau(x)))$ on $M'$. By this, the map $\beta_{M,M'}: \ M\to M'$ is defined. Since $\pi: \ {\rm M}\to M$ and $\pi': \ {\rm M}'\to M'$ are local diffeomorphisms, the map $\beta_{M,M'}$ is also a local diffeomorophism admitting the representation
\begin{equation}
\label{beta local}
\beta_{M,M'}=\pi'\circ\beta_{\rm M, M'}\circ\pi^{-1}.
\end{equation}
If $\beta_{M,M'}(y_1)=\beta_{M,M'}(y_2)=y'$, then 
$$\pi^{'-1}(\{y'\})=\beta_{\rm M, M'}(\pi^{-1}(\{y_1\}))=\beta_{\rm M, M'}(\pi^{-1}(\{y_2\})),$$ 
and, since $\beta_{\rm M, M'}$ is a bijection, $\pi^{-1}(\{y_1\})=\pi^{-1}(\{y_2\})$ and, hence, $y_1=y_2$. Therefore, $\beta_{M,M'}$ is an injection. Since $\pi^{-1}(M)={\rm M}$, $\beta_{M,M'}({\rm M})={\rm M}'$, and $\pi'({\rm M}')=M'$, the map $\beta_{M,M'}$ is a surjection. Thus, $\beta_{M,M'}: \ M\to M'$ is a diffeomorphism. Also, if $y\in\Gamma$, then its pre-image $\pi^{-1}(\{y\})\subset\Upsilon\equiv\Gamma\times\{\pm\}$ and consists of two points $(y,+)$ and $(y,-)$. Since $\beta_{\rm M, M'}$ does not move the points of $\Upsilon$, we have $\beta_{\rm M, M'}((y,\pm))=(y,\pm)$ and $\beta_{M,M'}(y)=\pi'(\beta_{\rm M, M'}((y,\pm))=y$. Therefore, $\beta_{M,M'}$ does not move the points of $\Gamma$. In addition, since $\pi: \ ({\rm M},{\rm g})\to (M,g)$ and $\pi': \ ({\rm M}',{\rm g}')\to (M',g')$ are local isometries, from (\ref{dilatation est}) and (\ref{beta local}) it follows that the dilatation $K_{\beta_{M, M'}}$ of $\beta_{M, M'}$ satisfies
$$\frac{1}{2}{\rm log}K_{\beta_{M, M'}}(x)\le c_{\mathcal{E},7}t^{\frac{1}{3}}, \qquad \forall x\in M,$$
which is valid for small $t>0$ and any $(M',g')\in\mathbb{B}(M,t)$. By this, we have proved estimate (\ref{stability estimate}). Hence according to definition (\ref{Tdist}) of the Teichm\"uller distance, we have 
$$d_T([(M,g)],[(M',g')])\le \frac{1}{2}{\rm log}K_{\beta_{M, M'}} \le c_{\mathcal{E},7}t^{\frac{1}{3}}, \quad \forall (M',g')\in\mathbb{B}(M,t).$$
In particular, the map $\mathscr{R}^{-}_{m,\Gamma}: \ (\mathscr{D}^{-}_{m,\Gamma},d_{op})\to (\mathscr{T}^{-}_{m,\Gamma},d_T)$ is continuous at the point $\Lambda$ (DN map of $(M,g)$). Since $(M,g)\in\mathscr{T}^{-}_{m,\Gamma}$ is arbitrary, Theorem \ref{main theorem} is proved.

\subsection{Stability estimates.} In both orientable and non-orientable cases, the arguments above yield not only the continuity of the solving map $\mathscr{R}^{o}_{m,\Gamma}$ but also the following continuity estimates.
\begin{prop}
\label{continuity estmates}
For each $\Lambda\in\mathscr{D}^{o}_{m,\Gamma}$, the map $\mathscr{R}^{o}_{m,\Gamma}: \ (\mathscr{D}^{o}_{m,\Gamma},d_{op})\to (\mathscr{T}^{o}_{m,\Gamma},d_T)$ satisfies 
\begin{align}
\label{stability estimate orient}
c(\Lambda)d_{op}(\Lambda',\Lambda)\le d_T(\mathscr{R}^{o}_{m,\Gamma}(\Lambda'),\mathscr{R}^{o}_{m,\Gamma}(\Lambda))\le C(\Lambda)d_{op}(\Lambda',\Lambda), \qquad o=+,\\
\label{stability estimate nonorient}
c(\Lambda)d_{op}(\Lambda',\Lambda)\le d_T(\mathscr{R}^{o}_{m,\Gamma}(\Lambda'),\mathscr{R}^{o}_{m,\Gamma}(\Lambda))\le C(\Lambda)d_{op}(\Lambda',\Lambda)^{\frac{1}{3}}, \qquad o=-.
\end{align}
Here $\Lambda'\in\mathscr{D}^{o}_{m,\Gamma}$ obeys $d_{op}(\Lambda',\Lambda)\le \tilde{c}(\Lambda)$; the positive constants $c(\Lambda)$, $C(\Lambda)$, $\tilde{c}(\Lambda)$ depend only on $\Lambda$.
\end{prop}
\begin{proof}
In the non-orientable case, the continuity of the solving map $\mathscr{R}^{+}_{m,\Gamma}: \ (\mathscr{D}^{+}_{m,\Gamma},d_{op})\to (\mathscr{T}^{+}_{m,\Gamma},d_T)$ and the second part of estimate (\ref{stability estimate nonorient}) are proved above. In the orientable case, the similar but more simpler proof of the continuity of the solving map $\mathscr{R}^{+}_{m,\Gamma}: \ (\mathscr{D}^{+}_{m,\Gamma},d_{op})\to (\mathscr{T}^{+}_{m,\Gamma},d_T)$ is obtained in \cite{BKTeich}. In this case, the arguments of Sections \ref{sec embeddings} and \ref{sec map alpha} can be implemented for projective holomorphic embeddings $\mathcal{E}: \ M\to \mathbb{C}^n$ and $\mathcal{E}_{M'}: \ M'\to \mathbb{C}^n$ of the surfaces $(M,g)$ and $(M',g')$ themselves. At the same time, the construction of the operator $\iota_{M,M'}$ which maps holomorphic functions on $(M,g)$ to holomorphic functions on $(M',g')$ with close traces on $\Gamma$ for small $t=d(\Lambda',\Lambda)$ requires arguments much more simpler than those used in Section \ref{sec map iota}. The reason is that the traces $\eta$ of holomorphic functions on $(M,g)$ (and, similarly, on $(M',g')$) are found from the linear equations $\mathfrak{D}\Re\eta=0$, $\Im\eta=J\Lambda\Re\eta+{\rm const}$ provided by Proposition \ref{trace-DN-map connection}, {\it a}). In addition, in the orientable case, the following estimate 
$$\sup_{(M',g')\in \mathbb{B}(M,t)}\|{\rm Tr}'\iota_{M,M'}w-{\rm Tr}w\|_{C^l(\Gamma;\mathbb{C})}\le c_{[(M,g)],l}\|{\rm Tr}w\|_{C^{l+1}(\Gamma;\mathbb{C})}d_{op}(\Lambda',\Lambda)$$
is valid (see Lemma 2, \cite{BKTeich}). By using this estimate instead of (\ref{application of main lemma - est 2}) in the arguments of Sections \ref{sec embeddings} and \ref{sec map alpha}, one constructs, for any surface $(M',g')\in \mathscr{R}^{+}_{m,\Gamma}$ such that its DN map $\Lambda'$ is sufficiently close to the DN map $\Lambda$ of $(M,g)$, the map $\beta_{M,M'}: \ M\to M'$ whose dilatation is subject to the estimate
$${\rm log}K_{\beta_{M,M'}}\le c_{[(M,g)]}t$$
(cf. with (\ref{stability estimate})). From the last inequality and the definition (\ref{Tdist}) of the Teichm\"uller distance, one obtains the second part of estimate (\ref{stability estimate orient}).

Now, let $(M,g)$ and $(M',g')$ be surfaces with DN maps $\Lambda$ and $\Lambda'$, respectively, and $[(M,g)]\in\mathscr{T}^{o}_{m,\Gamma}$, $[(M',g')]\in\mathscr{T}^{o}_{m,\Gamma}$. Also, let $\beta: \ M\to M'$ be a diffeomorphism minimizing the dilatation in (\ref{Tdist}) i.e. such that $\frac{1}{2}{\rm log}K_\beta=d_T([(M',g')],[(M,g)])$ and $\beta$ does not move the points of $\Gamma$. Since $[(M',g')]=[(M,\beta^*g')]$, one can assume that $M'=M$ and $\beta$ is the identity on $M$. Let $u$ and $u'$ be smooth harmonic functions on $(M,g)$ and $(M,g')$, respectively, and $u|_\Gamma=u'|)\Gamma=f$. In view of the divergence theorem, one has
\begin{equation}
\label{optimality 1}
((\Lambda'-\Lambda)f,f)_{L_2(\Gamma)}=(\nabla_{g'}u',\nabla_{g'}u)_{\vec{L}_2(M,g')}-(\nabla_{g}u',\nabla_{g}u)_{\vec{L}_2(M,g)}.
\end{equation}
Recall that one can chose an atlas $\{U_k,\phi_k\}_{k=1}^{N}$ on $M$ such that $\phi_k: \ U_k\to\mathbb{R}^2$ are isothermal coordinates on $U$ in which the components of the tensor $g$ are of the form $g_{ij}=\delta_{ij}$. Let $\{\chi_k\}_{k=1}^{M}$ be a partition of unity subordinate to the atlas $\{U_k,\phi_k\}_{k=1}^{N}$. Then the right-hand side of (\ref{optimality 1}) is equal to $\sum_{k=1}^M\mathscr{J}_k$, where 
$$\mathscr{J}_k:=\int\limits_{\phi_k(U_k)}\chi_k\circ\phi_{k}^{-1}[g^{'ij}\sqrt{|g'|}-g^{ij}\sqrt{|g|}]\partial_i(u'\circ\phi_k^{-1})\partial_j(u\circ\phi_k^{-1})dx$$
Here $g_{ij}=\delta_{ij}$ and $g'_{ij}$ are components of the metrics $g$ and $g'$ in local coordinates $\phi_k$, $\{g^{ij}\}_{i,j=1}^2$ and $\{g^{'ij}\}_{i,j=1}^2$ are matrices inverse to $\{g_{ij}\}_{i,j=1}^2$ and $\{g_{'ij}\}_{i,j=1}^2$, and $|g|$ and $|g'|$ are determinants of $\{g_{ij}\}_{i,j=1}^2$ and $\{g_{'ij}\}_{i,j=1}^2$, respectively. According to the definition of the dilatation $K_{\beta}(y)$ at the point $y=\phi_k^{-1}(x)\in U$, we have $K_{\beta}(y)^2=\lambda_{\rm max}(g'(x))/\lambda_{\rm min}(g'(x))$, where $\lambda_{\rm max}(g')$ (resp., $\lambda_{\rm min}(g')$) denotes the maximal (resp., minimal) eigenvalue of the matrix $\{g_{'ij}(x)\}_{i,j=1}^2$. Introduce the matrices $I:=\{g^{ij}\sqrt{|g|}\}_{i,j=1}^2=\{\delta_{ij}\}_{i,j=1}^2$ and $I'(x):=\{g^{'ij}(x)\sqrt{|g'|(x)}\}_{i,j=1}^2$. Then 
$$K_{\beta}(y)^2=\lambda_{\rm max}(I'(x))/\lambda_{\rm min}(I'(x)), \quad \lambda_{\rm max}(I'(x))\lambda_{\rm min}(I'(x))={\rm det\,}I'(x)=1.$$
Hence, $\lambda_{\rm max}(I'(x))=K_{\beta}(y)$ and $\lambda_{\rm min}(I'(x))=1/K_{\beta}(y)$. If the value of $d_T([(M',g')],[(M,g)])=\frac{1}{2}{\rm log}K_\beta$ is sufficiently small, we have 
$$|I'(x)-I|_{M^{2\times 2}}\le c|\lambda_{\rm max}(I'(x))-1|+|\lambda_{\rm min}(I'(x))-1|\le cd_T([(M',g')],[(M,g)])$$
for any $x\in\phi_k(U_k)$. Therefore, 
\begin{align*}
|\mathscr{J}_k|\le cd_T([(M',g')],[(M,g)])\Big[\int\limits_{\phi_k(U_k)}\chi_k\circ\phi_{k}^{-1}\sum_{i=1}^{2}|\partial_i(u'\circ\phi_k^{-1})|^2 dx+\\
+\int\limits_{\phi_k(U_k)}\chi_k\circ\phi_{k}^{-1}\sum_{i=1}^{2}|\partial_i(u\circ\phi_k^{-1})|^2 dx\Big].
\end{align*}
Note that 
\begin{align*}
\int\limits_{\phi_k(U_k)}\chi_k\circ\phi_{k}^{-1}\sum_{i=1}^{2}|\partial_i(u\circ\phi_k^{-1})|^2 dx=(\nabla_{\chi_k g}u,\nabla_{\chi_k g}u)_{\vec{L}_2(U,\chi_k g)},\\
[\inf\limits_{x\in\phi_k(U_k)}\lambda_{\rm min}(I'(x))]\int\limits_{\phi_k(U_k)}\chi_k\circ\phi_{k}^{-1}\sum_{i=1}^{2}|\partial_i(u'\circ\phi_k^{-1})|^2 dx\le \\ \le \int\limits_{\phi_k(U_k)}\chi_k\circ\phi_{k}^{-1}\sum_{i=1}^{2}g^{'ij}\partial_i(u'\circ\phi_k^{-1})\partial_j(u'\circ\phi_k^{-1})\sqrt{|g'|}dx=\\
=(\nabla_{\chi_k g'}u',\nabla_{\chi_k g'}u')_{\vec{L}_2(U,\chi_k g')}.
\end{align*}
Combining the last two formulas, one arrives at
\begin{align*}
\sum_{k=1}^{N}|\mathscr{J}_k|\le cd_T([(M',g')],[(M,g)])\big[\|\nabla_{g}u\|^2_{\vec{L}_2(U,g)}+\|\nabla_{g'}u'\|^2_{\vec{L}_2(U,g')}\big].
\end{align*}
Applying the divergence theorem again, one obtains
\begin{align*}
\sum_{k=1}^{N}|\mathscr{J}_k|\le cd_T([(M',g')],[(M,g)])((\Lambda'+\Lambda)f,f)_{L_2(\Gamma)}.
\end{align*}
Combining the last estimate with (\ref{optimality 1}) yields
\begin{equation}
\label{optimality 2}
|((\Lambda'-\Lambda)f,f)_{L_2(\Gamma)}|\le cd_T([(M',g')],[(M,g)])((\Lambda'+\Lambda)f,f)_{L_2(\Gamma)}.
\end{equation}
In (\ref{optimality 2}), let us rewrite $\Lambda'+\Lambda=2\Lambda+(\Lambda'-\Lambda)$ in the right-hand side and take supremum over all smooth $f$ obeying $\|f\|_{H^1(\Gamma)}=1$. As a result, we arrive at
$$d_{op}(\Lambda',\Lambda)=\|\Lambda'-\Lambda\|_{H^1(\Gamma)\to L_2(\Gamma)}\le cd_T([(M',g')],[(M,g)]).$$
By this, we have proved the first parts of inequalities (\ref{stability estimate orient}) and (\ref{stability estimate nonorient}).
\end{proof}

\subsection{Remarks on stability of surface topology} 
Let $\Lambda$ be DN map $\Lambda$ of non-orientable surface $(M,g)$, and $m=\chi(M)$. Denote a ball of positive radius $t$ with center $\Lambda$ in the space of bounded operators acting from $H^{1}(\Gamma)$ to $L_2(\Gamma)$ by $B(\Lambda,t)$. Note that, in the construction of the diffeomorphism $\beta_{M,M'}$ above, we actually do not use the assumption that the surface $(M',g')$ is non-orientable. Thus, there is a ball $B(\Lambda,t_0)$ which does not contain DN maps of orientable surfaces with Euler characteristic $m$. Below, we provide a more straightforward proof of this fact along with results on the stability of the surface topology under small perturbation of its DN map.
\begin{prop}
\
\begin{enumerate}[{\rm t}1.]
\item Let $\Lambda$ be a DN map of an orientable surface $(M,g)$ with boundary $\Gamma$, and $m=\chi(M)$. Then there is a ball $B(\Lambda,t_0)$ which does not contain DN maps of orientable surfaces whose Euler characteristics are more than $m$. At the same time, any ball $B(\Lambda,t)$ contains DN maps of orientable as well as non-orientable surfaces with arbitrary Euler characteristics less than $m$.
\item Let $\Lambda$ be a DN map of a non-orientable surface $(M,g)$ with boundary $\Gamma$, and $m=\chi(M)$. Then, for any given $m'$, there is a ball $B(\Lambda,t_0)$ which does not contain DN maps of orientable surfaces whose Euler characteristics are more than $m'$. Also, there is the ball $B(\Lambda,t_1)$ which does not contain DN maps of non-orientable surfaces whose Euler characteristics are more than $m$. At the same time, any ball $B(\Lambda,t)$ contains DN maps of non-orientable surfaces with arbitrary Euler characteristics less than $m$.
\end{enumerate}
\end{prop}
\begin{proof}
Statement t1. is proved in \cite{ZNS}. In \cite{ZNS}, it also shown that the surface $(M',g')$ with Euler characteristic $\chi(M')<m$ and DN map $\Lambda'$ arbitrarily close to $\Lambda$ can be constructed by removing a finite number of small disks close to each other (in the metrics $g$) from $M$ then attaching the handles/M\"obius strips to it. This construction does not depend on the orientability of $M$.

Now, suppose that $\Lambda$ is a DN map of non-orientable surface $(M,g)$ with boundary $\Gamma$, and $m=\chi(M)$. Let $m'$ be arbitrarily large negative integer number. In view of Proposition \ref{trace-DN-map connection}, {\it a}), the kernel of the operator $\mathfrak{D}:=\partial_\gamma+\Lambda J\Lambda$ consists only of constants. So, there are smooth functions $f_1,\dots,f_{1-m'}$ on $\Gamma$ such that $\|\mathfrak{D}\sum_{k=1}^{1-m'}c_k f_k\|_{L_2(\Gamma;\mathbb{R})}>c_0>0$ for any $\vec{c}=(c_1,\dots,c_{1-m'})^T\in S^{-m'}$. Then, in view of (\ref{closeness of NDG}), there exists a sufficiently small ball $B(\Lambda,t_0)$ such that the operator $\mathfrak{D}':=\partial_\gamma+\Lambda' J\Lambda'$ obeys $\|\mathfrak{D}'\sum_{k=1}^{1-m'}c_k f_k\|_{L_2(\Gamma;\mathbb{R})}>c_0/2$ for any $\vec{c}\in S^{-m'}$ and $\Lambda'\in B(\Lambda,t_0)$. In particular, ${\rm dim}\frac{C^{\infty}(\Gamma;\mathbb{R})}{{\rm Ker}\mathfrak{D}'}\ge 1-m'$. Suppose that $\Lambda'\in B(\Lambda,t_0)$ is a DN map of an orientable surface $(M',g')$ with boundary $\Gamma$. As shown in \cite{B}, the topology of $(M',g')$ is connected with its DN map $\Lambda'$ by ${\rm dim}\frac{C^{\infty}(\Gamma;\mathbb{R})}{{\rm Ker}\mathfrak{D}'}=1-\chi(M')$. Hence, $\chi(M')<m'$. Thus, $B(\Lambda,t_0)$ does not contain DN maps of orientable surfaces whose Euler characteristics are more than $m'$.

Next, recall that the map $\mathcal{G}$ associated with $\Lambda$ via (\ref{non-linear map}) is $(-m,3)-$admissible. Hence, we have ${\rm dim}\frac{C^{\infty}(\Gamma;\mathbb{R})}{\mathcal{G}^{-1}(\{0\})}=-m$. So, there are smooth functions $f_1,\dots,f_{-m}$ on $\Gamma$ such that $\|\mathcal{G}(\sum_{k=1}^{-m}c_k f_k)\|_{L_2(\Gamma;\mathbb{R})}>c_0>0$ for any $\vec{c}\in S^{-m-1}$. In view of (\ref{closeness of NDG}), there exists a sufficiently small ball $B(\Lambda,t_1)$ such that the operator $\mathcal{G}'$ associated with $\Lambda'$ via (\ref{non-linear map}) satisfies $\|\mathcal{G}(\sum_{k=1}^{-m}c_k f_k)\|_{L_2(\Gamma;\mathbb{R})}>c_0/2$ for any $\vec{c}\in S^{-m'}$ and $\Lambda'\in B(\Lambda,t_1)$. In particular, we have ${\rm dim}\frac{C^{\infty}(\Gamma;\mathbb{R})}{\mathcal{G}^{-1}(\{0\})}\ge -m$. Suppose that $\Lambda'\in B(\Lambda,t_1)$ is a DN map of a non-orientable surface $(M',g')$ with boundary $\Gamma$. Then $\mathcal{G}'$ is a $(-\chi(M'),3)-$admissible operator and, hence, $-\chi(M')={\rm dim}\frac{C^{\infty}(\Gamma;\mathbb{R})}{\mathcal{G}^{-1}(\{0\})}\ge -m$. Thus, $B(\Lambda,t_1)$ does not contain DN maps of non-orientable surfaces whose Euler characteristics are more than $m$. 
\end{proof}

\section{Generalizations}
\label{sec generalizations}
\subsection{Surfaces with unknown electrically isolated components of boundary}
Let $(M,g)$ be a surface with boundary $\partial M$ and $\Gamma$ be a connected component of $\partial M$. In this paragraph, we make no assumptions on orientability of $M$. Denote the length element induced by the metric $g$ on $\Gamma$ by $dl$ and the outward normal vector on $\Gamma$ by $\nu$. Let $u^{f,(is)}$ be a harmonic smooth function in $(M,g)$ obeying $u^{f,(is)}=f$ on $\Gamma$ and $\partial_\nu u^{f}_{is}=0$ on $\partial M\backslash\overline{\Gamma}$. With $(M,g)$ and $\Gamma$ we associate the DN map $\Lambda_{is}$, acting on smooth functions on $\Gamma$ by the rule $\Lambda_{is}f:=\partial_{\nu}u^{f}_{is}|_\Gamma$. 

Let $(M',g')$ be another surface, $\partial M'\supset\Gamma$ and the metric $g'$ induces on $\Gamma$ the same length element $dl$. We write $[(M',g')]=[(M,g)]$ if there is the conformal diffeomorphism $\beta$ between $(M,g)$ and $(M',g')$ which does not move points of $\Gamma$. As proved in \cite{BKor_IP}, the DN map $\Lambda'_{is}$ of $(M',g')$ coincides with $\Lambda_{is}$ if and only if $[(M',g')]=[(M,g)]$. 

Let us write $[(M,g)]\in\mathscr{T}^{+}_{m,q,\Gamma}$ (resp., $[(M,g)]\in\mathscr{T}^{-}_{m,q,\Gamma}$) is the Euler characteristic $\chi(M)$ of $M$ is equal to $m$, the number of connected components of $\partial M$ is equal to $q$, $\Gamma$ is a connected component of $\partial M$. The set $\mathscr{T}^{o}_{m,q,\Gamma}$ ($o=\pm$) is endowed with the Teichm\"uller-type metric $d_T$ given by (\ref{Tdist}). At the same time, the set $\mathscr{D}^\mathscr{o,is}_{m,q,\Gamma}$ of DN maps $\Lambda_{is}$ of surfaces whose conformal classes belong to $\mathscr{T}^{o}_{m,q,\Gamma}$ is endowed with the metric $d_{op}$ given by (\ref{D op}). Introduce the (bijective) map $\mathscr{R}^\mathscr{o,is}_{m,q,\Gamma}: \ {D}^\mathscr{o,is}_{m,q,\Gamma}\to\mathscr{T}^{o}_{m,q,\Gamma}$ by the rule $\mathscr{R}^\mathscr{o,is}_{m,\Gamma}(\Lambda_{is})=[(M,g)]$, where $(M,g)$ is a surface with the DN map $\Lambda_{is}$. 
\begin{theorem}
\label{holes is theorem}
The map $\mathscr{R}^{o,is}_{m,q,\Gamma}: \ (\mathscr{D}^{o,is}_{m,q,\Gamma},d_{op})\to (\mathscr{T}^{o}_{m,q,\Gamma},d_T)$ {\rm(}$o=\pm${\rm)} is continuous. Moreover, for arbitrarily fixed $\Lambda_{is}\in\mathscr{D}^{o,is}_{m,q,\Gamma}$ and any $\Lambda'_{is}\in\mathscr{D}^{o,is}_{m,q,\Gamma}$, the estimate 
\begin{equation}
\label{stab est holes is}
c(\Lambda_{is})d_{op}(\Lambda'_{is},\Lambda_{is})\le d_T(\mathscr{R}^{o,is}_{m,q,\Gamma}(\Lambda'_{is}),\mathscr{R}^{o,is}_{m,q,\Gamma}(\Lambda_{is}))\le C(\Lambda_{is})d_{op}(\Lambda'_{is},\Lambda_{is})^{\frac{1}{3}}
\end{equation}
holds with the constants $c(\Lambda_{is})$, $C(\Lambda_{is})$ depending only on $\Lambda_{is}$.
\end{theorem}
Theorem \ref{holes is theorem} is proved by applying the arguments presented in Sections \ref{sec preliminaries}-\ref{sec map alpha}; the only difference is in the choice of the covering space $({\rm M},{\rm g})$. Namely, in the orientable case $o=+$, the covering space $({\rm M},{\rm g})$ is obtained by gluing two copies of $(M,g)$ along the parts $\partial M\backslash\Gamma$ of their boundaries. In the non-orientable case $o=-$, the covering space $({\rm M},{\rm g})$ is obtained from the orientable double cover of $({\rm M},{\rm g})$ by identifying the points which have the joint projection belonging to $\partial M\backslash\Gamma$. In both cases, the boundary ${\rm M}$ is the union of two copies $\Gamma_+$, $\Gamma_-$ of $\Gamma$. The set of boundary traces of holomorphic functions on $({\rm M},{\rm g})$ which are symmetric with respect to the involution on ${\rm M}$ is still described by Proposition \ref{trace-DN-map connection}, {\it b}) and formula (\ref{numerator denominator}), where $\Lambda$ is replaced by $\Lambda_{is}$ (see Lemma 2, \cite{BKor_IP}). Due to this fact, the arguments of Sections \ref{sec preliminaries}-\ref{sec map alpha} provide the near-conformal diffeomorphism $\beta_{M,M'}: \ (M,g)\to (M',g')$ obeying (\ref{stability estimate}) for arbitrarily fixed surface $(M,g)$ ($[(M,g)]\in\mathscr{T}^{o}_{m,q,\Gamma}$) and any surface $(M',g')$ such that $\Gamma\subset\partial M'$, $\chi(M')=m$, $(M',g')$ is orientable if $o=+$ and non-orientable if $o=-$, and the DN maps $\Lambda'_{is}$, $\Lambda$ of $(M',g')$ and $(M,g)$ obey $\parallel\Lambda'_{is}-\Lambda_{is}\parallel_{H^{1}(\Gamma)\to L^{2}(\Gamma)}=t$ with sufficiently small $t\in (0,t_0)$. Note that the arguments of Sections \ref{sec preliminaries}-\ref{sec map alpha} do not use the assumption that the numbers of connected components of the boundaries $\partial M$ and $\partial M'$ do coincide but only the assumption $\chi({\rm M}')=\chi({\rm M})$ on the Euler characteristics of the covering spaces ${\rm M}$ and ${\rm M'}$. Hence, along with Theorem \ref{holes is theorem}, we obtain also the following topological stability result.
\begin{prop}
\label{top stab isol}
Let $(M,g)$ be a surface such that $[(M,g)]\in\mathscr{T}^{o}_{m,q,\Gamma}$, and $\Lambda_is\in\mathscr{D}^\mathscr{o,is}_{m,q,\Gamma}$ be a DN map of $(M,g)$. Then there exists $t_0=t_0(\Lambda_{is})$, such that, for any surface obeying $\Gamma\subset\partial M'$, $\chi(M')=m$, $(M',g')$ is orientable if $o=+$ and non-orientable if $o=-$, the inequality $\parallel\Lambda'_{is}-\Lambda_{is}\parallel_{H^{1}(\Gamma)\to L^{2}(\Gamma)}<t_0$ for the DN map $\Lambda'_{is}$ of $(M',g')$ implies that $M'$ is diffeomorphic to $M$. In other words, we have
$$\overline{\mathscr{D}^\mathscr{o,is}_{m,q,\Gamma}}\cap\overline{\mathscr{D}^\mathscr{o,is}_{m,q',\Gamma}}=\varnothing, \qquad q'\ne q$$
{\rm(}the closures in the space of bounded operators acting from $H^{1}(\Gamma)$ to $L^{2}(\Gamma)${\rm)}.
\end{prop}

\subsection{Orientable surfaces with unknown grounded components of boundary}
In this subsection, we assume that all surfaces are orientable. In the notation of the previous subsection, let $u^{f,(gr)}$ be a harmonic smooth function in $(M,g)$ obeying $u^{f,(gr)}=f$ on $\Gamma$ and $u^{f}_{gr}=0$ on $\partial M\backslash\overline{\Gamma}$. Also, with $(M,g)$ and $\Gamma$ we associate the DN map $\Lambda_{gr}$, acting on smooth functions on $\Gamma$ by the rule $\Lambda_{gr}f:=\partial_{\nu}u^{f}_{gr}|_\Gamma$. As proved in \cite{BKor_IP}, the DN map $\Lambda'_{gr}$ of $(M',g')$ coincides with $\Lambda_{gr}$ if and only if $[(M',g')]=[(M,g)]$. Denote the set DN maps $\Lambda_{gr}$ of surfaces whose conformal classes belong to $\mathscr{T}^{+}_{m,q,\Gamma}$ by $\mathscr{D}^\mathscr{+,gr}_{m,q,\Gamma}$. Introduce the (bijective) map $\mathscr{R}^\mathscr{+,gr}_{m,q,\Gamma}: \ {D}^\mathscr{+,gr}_{m,q,\Gamma}\to\mathscr{T}^{+}_{m,q,\Gamma}$ by the rule $\mathscr{R}^\mathscr{+,gr}_{m,\Gamma}(\Lambda_{gr})=[(M,g)]$, where $(M,g)$ is a surface with the DN map $\Lambda_{gr}$. 
\begin{theorem}
\label{holes gr theorem}
The map $\mathscr{R}^{+,gr}_{m,q,\Gamma}: \ (\mathscr{D}^{+,gr}_{m,q,\Gamma},d_{op})\to (\mathscr{T}^{+}_{m,q,\Gamma},d_T)$ is continuous. Moreover, for arbitrarily fixed $\Lambda_{gr}\in\mathscr{D}^{+,gr}_{m,q,\Gamma}$ and any $\Lambda'_{gr}\in\mathscr{D}^{+,gr}_{m,q,\Gamma}$, the estimate 
\begin{equation}
\label{stab est holes gr}
c(\Lambda_{gr})d_{op}(\Lambda'_{gr},\Lambda_{gr})\le d_T(\mathscr{R}^{+,gr}_{m,q,\Gamma}(\Lambda'_{gr}),\mathscr{R}^{+,gr}_{m,q,\Gamma}(\Lambda_{gr}))\le C(\Lambda_{gr})d_{op}(\Lambda'_{gr},\Lambda_{gr})^{\frac{1}{2}}
\end{equation}
holds with the constants $c(\Lambda_{gr})$, $C(\Lambda_{gr})$ depending only on $\Lambda_{gr}$. Also, we have
$$\overline{\mathscr{D}^\mathscr{+,gr}_{m,q,\Gamma}}\cap\overline{\mathscr{D}^\mathscr{+,gr}_{m,q',\Gamma}}=\varnothing, \qquad q'\ne q$$
{\rm(}the closures in the space of bounded operators acting from $H^{1}(\Gamma)$ to $L^{2}(\Gamma)${\rm)}.
\end{theorem}
The proof of Theorem \ref{holes gr theorem} is the same as the proof of Theorem \ref{holes is theorem} and Proposition \ref{top stab isol}. The only difference is that, to determine the elements of $\mathcal{H}_+({\rm M})$ i.e. the boundary traces of holomorphic functions on $({\rm M},{\rm g})$ which are symmetric with respect to the involution on ${\rm M}$, one should use the following fact instead of Proposition \ref{trace-DN-map connection}, {\it b}).
\begin{prop}[See Lemma 1, \cite{BKor_IP}]
\label{trace-DN-map connection}
The function $\eta:=f+ih$ with $f,h\in C^{\infty}(\Gamma;\mathbb{R})$ is a trace on $\Gamma\equiv\Gamma_+$ of symmetric holomorphic function on ${\rm M}$ if and only if $\mathcal{G}(h)=0$ and $f=-J\Lambda_{gr}h+c$ {\rm(}$c\in\mathbb{R}${\rm)}, where the map $\mathcal{G}: \ C^{\infty}(\Gamma;\mathbb{R})\to C^{\infty}(\Gamma;\mathbb{R})$ is given by
$$ \mathcal{G}(h)=\partial_{\gamma}[(J\Lambda_{gr} h)^2-h^2]-2\Lambda_{gr}(hJ\Lambda_{gr} h).$$
\end{prop}
In contract to the map (\ref{non-linear map}), the map $\mathcal{G}$ is homogeneous of degree 2 i.e. $\mathcal{G}(ch)=c^2\mathcal{G}(h)$ for any $h\in C^{\infty}(\Gamma;\mathbb{R})$ and $c\in\mathbb{R}$. Also, $\mathcal{G}(f+h)=\mathcal{G}(f)+\mathcal{G}(h)+\mathcal{G}^{(1)}_f(h)$, where $\mathcal{G}^{(1)}_f$ is a linear operator given by
$$\mathcal{G}^{(1)}_f(h):=2(\partial_{\gamma}[(J\Lambda_{gr} f)(J\Lambda_{gr} h)-fh]-2\Lambda_{gr}(hJ\Lambda_{gr} f-fJ\Lambda_{gr}h)).$$
In addition, the codimension of $\mathcal{G}^{-1}(\{0\})$ in $C^{\infty}(\Gamma,\mathbb{R})$ is equal to $1-\chi(M)$. Hence, the map $\mathcal{G}$ is $(1-\chi(M),2)$-admissible. 

Let $\mathcal{G}$ and $\mathcal{G}'$ be maps associated with the DN maps $\Lambda_{gr}\in \mathscr{D}^{o,gr}_{m,q,\Gamma}$ and $\Lambda'_{gr}\in \mathscr{D}^{o,gr}_{m,q,\Gamma}$, respectively. Then the application of Lemma \ref{main lemma} and the arguments leading to formula (\ref{application of main lemma - est 1}) provide the linear map $\mathfrak{Y}_{\mathcal{G},\mathcal{G}'}: \ \mathcal{G}^{-1}(\{0\})\to \mathcal{G}^{'-1}(\{0\})$ satisfying
$$\|\mathfrak{Y}_{\mathcal{G},\mathcal{G}'}f-f\|_{C^{l}(\Gamma)}\le c_{f,l}t^{1/2}, \qquad \forall f\in \mathcal{G}^{-1}(\{0\}), \ l=1,2,\dots,$$
where $t:=\parallel\Lambda'_{gr}-\Lambda_{gr}\parallel_{H^{1}(\Gamma)\to L^{2}(\Gamma)}$. We define the operator $\tilde{\iota}_{M,M'}: \ {\rm Tr}\mathcal{H}_+({\rm M})\to {\rm Tr}\mathcal{H}_+({\rm M}')$ by the rule 
$$\tilde{\iota}_{M,M'}\eta:= -J\Lambda_{gr}h\mathfrak{Y}_{\mathcal{G},\mathcal{G}'}h+c+i\mathfrak{Y}_{\mathcal{G},\mathcal{G}'}h$$
where $\eta=-J\Lambda_{gr}h+c+ih$, $h\in\mathcal{G}^{-1}(\{0\})$, and $c\in\mathbb{R}$. Then the estimate
$$\sup_{(M',g')\in \mathbb{B}(M,t)}\|\tilde{\iota}_{M,M'}\eta-\eta\|_{C^{l}(\Gamma;\mathbb{R})}\le c_{\eta,l}t^{\frac{1}{2}}$$
is valid instead of (\ref{application of main lemma - est 2}), where $\eta\in{\rm Tr}\mathcal{H}_+({\rm M})$. Now, by repeating the arguments of Sections \ref{sec embeddings} and \ref{sec map alpha}, one proves (\ref{stab est holes gr}) (cf. with (\ref{stab est holes is})).

In conclusion, we note that the results described in this section and in the paper \cite{BKor_IP} can easily be generalized to the case in which the DN map is given only on an arbitrarily small joint segment of the boundaries of the surfaces while the homogeneous Dirichlet or Neumann condition is given on the remainders of the boundaries.


\begin{thebibliography}{99}

\bibitem{Alf}
L.V.Ahlfors.
\newblock {Lectures on Quasiconformal Mappings: Second Edition.}
\newblock {\em University Lecture Series, vol. 38 {\rm(}2nd ed.{\rm)}},
University Lecture Series, Providence, R.I.

\bibitem{BKor_IP}
A.V.Badanin, M.I.Belishev, D.V.Korikov.
\newblock{Electric impedance tomography problem for surfaces with internal
holes.}
\newblock {\em  Inverse Problems}, 2021, Vol. 37, no. 10. Doi
10.1088/1361-6420/ac245c.

\bibitem{B}
M.I.Belishev.
\newblock {The Calderon problem for two-dimensional manifolds
by the BC-method.}
\newblock {\em SIAM Journal of Mathematical Analysis}, 35, no 1:
172--182, 2003.

\bibitem{BKor_JIIPP}
M.I.Belishev, D.V.Korikov.
\newblock {On the EIT problem for nonorientable
surfaces.}
\newblock {\em Journal of Inverse and Ill-posed
Problems Journal of Inverse and Ill-posed Problems}, 18, December
2020. doi:10.1515/jiip-2020-0129.

\bibitem{BKor_SIAM}
M.I.Belishev, D.V.Korikov.
\newblock {On Determination of Nonorientable
Surface via its Diriclet-to-Neumann Operator.}
\newblock {\em  SIAM Journal on
Mathematical Analysis}, 2021, Vol. 53, no. 5, 5278-5287. DOI
10.1137/20M137762.

\bibitem{BKstab}
M.I.Belishev, D.V.Korikov.
\newblock {On stability of determination of Riemann surface from its DN-map.} {\it
arXiv}:\,arXiv:2112.14816v2 [math-ph] 29 Dec 2021.

\bibitem{BKTeich}
M.I.Belishev, D.V.Korikov.
\newblock {Stability of determination of Riemann surface from its DN-map in terms of Teichm\"uller distance.} {\it
arXiv}:\,arXiv:2208.00298 [math-ph] 30 Jul 2022.

\bibitem{Chirka}
Chirka~E.~M.,
\newblock {Riemann surfaces},
\newblock {\em Lektsionnye Kursy NOTs, vyp. 1}, Moscow, 2006. (Russian)

\bibitem{Forster}
O.Forster.
\newblock {Lectures on Riemann Surfaces.}
\newblock {\em Graduate Texts in Mathematics; 81},
Springer-Verlag, New-York, Heidelberg, Berlin, 1980. ISSN
0-387-90617-7.

\bibitem{Ga}
F.P.Gardiner,
\newblock {Teichm\"uller theory and quadratic differentials},
\newblock {\em Pure Appl. Math. (N. Y.)},
John Wiley \& Sons, New York, 1987.

\bibitem{G}
M.Gromov.
\newblock {Metric structures for Riemannian and non-Riemannian spaces.}
\newblock {\em Modern Birkh\"auser Classics},
Birkh\"auser Boston, 2007, 1980. ISSN 2197-1803.


\bibitem{ZNS}
D.V.Korikov.
\newblock {On the topology of surfaces with a common boundary and close DN-maps.}
\newblock {\em Zapiski Nauchnykh Seminarov POMI}, 506:
57--66, 2021.\quad{(in Russian)}  


\bibitem{LU}
M.Lassas, G.Uhlmann.
\newblock {On determining a Riemannian manifold from the
Dirichlet-to-Neumann map.}
\newblock {\em Ann. Scient. Ec. Norm. Sup.}, 34(5): 771-- 787,
2001.

\bibitem{Lee}
J.Lee.
\newblock {Introduction to smooth manifolds.}
\newblock {\em Grad. Texts in Math.}, vol. 218, Spri\-n\-ger-Verlag, New York, 2013.

\bibitem{LeeU}
J.M.Lee, G. Uhlmann.
\newblock {Determining anisotropic real‐analytic conductivities by boundary measurements}.
\newblock {\em Comm. Pure Appl. Math.}, 42 (1989), 1097--1112.

\bibitem{M}
R.Miranda,
\newblock {Algebraic curves and Riemann surfaces.}
Graduate Texts in Mathematics, {\bf 5}, AMS, 1995.

\bibitem{Nag}
S.Nag, M.Nafissi,
\newblock {The complex analytic theory of Teichm\"uller spaces},
\newblock {\em Wiley-Interscience}, 1988.

\bibitem{Sch}
G.Schwarz. 
\newblock {Hodge Decomposition—A Method for Solving Boundary Value Problems.}
\newblock {\em Lecture Notes in Math.}, 1607, Springer-Verlag, Berlin, 1995.


\end{thebibliography}
\end{document}